\documentclass[11pt,english]{smfart} 
\usepackage{bm}
\usepackage[english,francais]{babel}
\usepackage{amscd}
\usepackage{amssymb,url,xspace,smfthm}
\usepackage{paralist}
\usepackage{datetime} 
\usepackage{colortbl}
\usepackage{color}
\usepackage[dvipsnames]{xcolor}
\usepackage[hypertexnames=true,linktocpage]{hyperref} 
\usepackage{mathtools}
\usepackage{euscript}
\usepackage[all]{xy}
\usepackage{graphicx} 

\usepackage{nameref}
\usepackage[T1]{fontenc}
\usepackage{newtxtext}
\usepackage{newtxmath}
\usepackage{textcomp}
\usepackage[text={150mm,251mm},centering, marginparwidth=75pt]{geometry} 
\usepackage{comment}  
\usepackage{cite}  
\usepackage{thmtools}
\usepackage{enumitem}
\usepackage{letltxmacro} 
\usepackage{nameref}
\usepackage{cleveref}
\usepackage{calc}
\usepackage{interval}
\usepackage{manfnt}
\usepackage{tikz-cd}  
\usepackage[utf8]{inputenc}
\usepackage{marginnote}
\usepackage{euscript}

\usepackage{smfhyperref}
\usepackage[hyperpageref]{backref}

\usepackage{faktor}

\theoremstyle{plain}
\newtheorem{thmx}{Theorem}
\renewcommand{\thethmx}{\Alph{thmx}} 
\newtheorem{thm}{Theorem}[section]  
\newtheorem{lem}[thm]{Lemma}
\newtheorem{claim}[thm]{Claim}

\newtheorem{proposition}[thm]{Proposition}
\newtheorem{cor}[thm]{Corollary}

\newtheorem{conjecture}[thm]{Conjecture}
\theoremstyle{definition}
\newtheorem{dfn}[thm]{Definition}

\theoremstyle{remark}
\newtheorem{rem}[thm]{Remark}


\numberwithin{equation}{section}  

\theoremstyle{plain}
\newlist{thmlist}{enumerate}{1}
\setlist[thmlist]{wide = 0pt, labelwidth = 2em, labelsep*=0em, itemindent = 0pt, leftmargin = \dimexpr\labelwidth + \labelsep\relax, noitemsep,topsep = 1ex, font=\normalfont, label=(\roman*), ref=\thethm.(\roman{thmlisti})}

\addtotheorempostheadhook[thm]{\crefalias{thmlisti}{thm}}

\addtotheorempostheadhook[assumpsion]{\crefalias{thmlisti}{assumption}}

\addtotheorempostheadhook[cor]{\crefalias{thmlisti}{cor}}

\addtotheorempostheadhook[proposition]{\crefalias{thmlisti}{proposition}}

\addtotheorempostheadhook[dfn]{\crefalias{thmlisti}{dfn}}

\addtotheorempostheadhook[lem]{\crefalias{thmlisti}{lem}}
\addtotheorempostheadhook[main]{\crefalias{thmlisti}{main}}

\addtotheorempostheadhook[rem]{\crefalias{thmlisti}{rem}}

\newlist{thmenum}{enumerate}{1} % also creates a counter called 'propenumi'
\setlist[thmenum]{wide = 0pt, labelwidth = 2em, labelsep*=0em, itemindent = 0pt, leftmargin = \dimexpr\labelwidth + \labelsep\relax, noitemsep,topsep = 1ex, font=\normalfont, label=(\roman*), ref=\thethmx.(\roman{thmenumi})}%{label=\alph*), ref=\thethmx~(\alph*)}
\crefalias{thmenumi}{thmx} 

\newlist{corlist}{enumerate}{1} % also creates a counter called 'propenumi'
\setlist[corlist]{wide = 0pt, labelwidth = 2em, labelsep*=0em, itemindent = 0pt, leftmargin = \dimexpr\labelwidth + \labelsep\relax, noitemsep,topsep = 1ex, font=\normalfont, label=(\roman*), ref=\thecorx.(\roman{corlisti})}%{label=\alph*), ref=\thethmx~(\alph*)}
\crefalias{corlisti}{corx} 

\crefname{lem}{Lemma}{Lemmas} 
\crefname{conjecture}{Conjecture}{Conjectures}
\crefname{thm}{Theorem}{Theorems}
\crefname{proposition}{Proposition}{Propositions}
\crefname{dfn}{Definition}{Definitions}
\crefname{rem}{Remark}{Remarks}
\crefname{cor}{Corollary}{Corollaries}
\crefname{corx}{Corollary}{Corollaries}
\crefname{problem}{Problem}{Problems}
\crefname{thmx}{Theorem}{Theorems}
\crefname{claim}{Claim}{Claims}
\crefname{assumption}{Assumption}{Assumptions}
\crefname{main}{Main Theorem}{Main Theorems}

\def\ep{\varepsilon}

\def\D{\mathscr{D}} 

\def\D{D}

\newcommand{\da}{\rightharpoonup}

\makeatletter
\newcommand*{\rom}[1]{\expandafter\@slowromancap\romannumeral #1@}
\makeatother

\makeatletter     %Replace the Section ...  by the symbol \S ...  in Cref
\newcommand{\crefnames}[3]{%
	\@for\next:=#1\do{%
		\expandafter\crefname\expandafter{\next}{#2}{#3}%
	}%
}
\makeatother

\crefnames{part,chapter,section}{\S}{\S\S}

\newcommand{\cA}{\mathcal A}

\newcommand{\cO}{\mathcal O}

\newcommand{\bC}{\mathbb{C}}
\newcommand{\bD}{\mathbb{D}}

\newcommand{\bP}{\mathbb{P}}

\newcommand{\bZ}{\mathbb{Z}}

\newcommand{\pN}[1]{\bar{N}_{#1}^{\partial}}

  \def\spec{\textrm{Spec}\,}

\newcommand{\Spab}{\mathrm{Sp}_{\mathrm{sab}}}
\newcommand{\Sph}{\mathrm{Sp}_{\mathrm{h}}}
\newcommand{\Spp}{\mathrm{Sp}_{\mathrm{p}}}
\newcommand{\Spalg}{\mathrm{Sp}_{\mathrm{alg}}}

\newcommand{\ord}{{\rm ord}\,}

\newcommand{\ram}{{\rm ram}\,}

\title[On the GGL conjecture]{Hyperbolicity and fundamental groups of complex quasi-projective varieties (I): Maximal quasi-Albanese dimension by Nevanlinna theory}

 \date{\today} 

\author[B. Cadorel]{Beno\^{i}t Cadorel} 
\email{benoit.cadorel@univ-lorraine.fr}
\address{Institut \'Elie Cartan de Lorraine, Universit\'e de Lorraine, F-54000 Nancy,
	France}
\urladdr{http://www.normalesup.org/~bcadorel/} 
\author[Y. Deng]{Ya Deng}

\email{ya.deng@math.cnrs.fr, deng@imj-prg.fr}
\address{CNRS,  
	Institut de Math\'ematiques de Jussieu-Paris Rive Gauche,
	Sorbonne Universit\'e, Campus Pierre et Marie Curie,
	4 place Jussieu, 75252 Paris Cedex 05, France}
\urladdr{https://ydeng.perso.math.cnrs.fr}

\author[K. Yamanoi]{Katsutoshi Yamanoi}

\email{yamanoi@math.sci.osaka-u.ac.jp}
\address{Department of Mathematics, Graduate School of Science, Osaka University, Toyonaka,  Osaka 560-0043, Japan} 
\urladdr{https://sites.google.com/site/yamanoimath/} 
\begin{document} 
%\begin{abstract}In this paper we establish a Second Main Theorem for holomorphic maps $f:Y \to X$, where $Y$ is a proper ramified covering of $\bC_{>1}:=\{z\in\bC \mid |z|>1\}$ and $X$ is a complex quasi-projective variety of maximal quasi-Albanese dimension. As a consequence, we prove the generalized Green–Griffiths–Lang conjecture for such $X$, thereby extending the 2015 result of the third author from the projective to the quasi-projective case. \end{abstract}
	
	\begin{abstract}
	This is the first part of a series of three papers.
  In this paper, we establish a Big Picard type theorem for holomorphic maps $f:Y \to X$, where $Y$ is a ramified covering of the punctured disc $\bD^*$ with small ramification and $X$ is a complex quasi-projective variety of log-general type and of maximal quasi-Albanese dimension.
As a byproduct, we prove the generalized Green–Griffiths–Lang conjecture for such $X$. 
This paper summarizes the parts of the three-paper series that are based primarily on Nevanlinna theory.
	\end{abstract}
	
	\maketitle

\tableofcontents
\section{Introduction}
\subsection{Generalized Green-Griffiths-Lang conjecture}
The generalized version of the Green–Griffiths and Lang conjectures is the complex analogue of the Bombieri–Lang conjecture about the Zariski density of the set of rational points of an algebraic variety of general type.  It is a fundamental problem in the study of the hyperbolicity of algebraic varieties (cf. \cite[I, 3.5]{Lan97} and \cite[VIII, Conj. 1.3]{Lan97}). To formulate this conjecture, we first introduce several notions of non-hyperbolicity loci, usually referred to as special subsets in various senses.
\begin{dfn}[Special subsets] \label{def:special2}
	Let $X$ be a smooth quasi-projective variety.
	\begin{thmlist}
		\item $\Spab(X) := \overline{\bigcup_{f}f(A_0)}^{\mathrm{Zar}}$, where $f$ ranges over all non-constant rational maps $f:A\dashrightarrow X$ from all non-trivial semi-abelian varieties $A$ to $X$ such that $f$ is regular $A_0\to X$ on a Zariski open subset $A_0\subset A$ whose complement $A\backslash A_0$ has codimension at least two;
		\item $\Sph(X) := \overline{\bigcup_{f}f(\mathbb{C})}^{\mathrm{Zar}}$, where $f$ ranges over all non-constant holomorphic maps from $\mathbb{C}$ to $X$;
		\item $\Spalg(X) := \overline{\bigcup_{V} V}^{\mathrm{Zar}}$, where $V$ ranges over all positive-dimensional closed subvarieties of $X$ which are not of log general type;
		\item $\Spp(X) := \overline{\bigcup_{f}f(\bD^*)}^{\mathrm{Zar}}$, where $f$ ranges over all holomorphic maps from the punctured disk $\bD^*$ to $X$ with essential singularity at the origin, i.e., $f$ has no holomorphic extension $\bar{f}:\mathbb D\to\overline{X}$ to a projective compactification $\overline{X}$.
	\end{thmlist}
\end{dfn}
The \emph{generalized  Green-Griffiths-Lang conjecture} can be stated as follows.  
\begin{conjecture}[Generalized  Green-Griffiths-Lang]\label{conj:GGL}
	Let $X$ be a smooth quasi-projective variety.  Then the following properties are equivalent:
	\begin{thmlist} 
		\item  $X$ is of log general type; 
		\item  $X$ is \emph{strongly of general type},  i.e.,
		$\Spalg(X)\subsetneqq X$;
		\item  $X$ is \emph{pseudo Picard hyperbolic},  i.e., $\Spp(X)\subsetneqq X$; 
		\item  $X$ is \emph{pseudo Brody hyperbolic},  i.e., $\Sph(X)\subsetneqq X$; 
		\item  
		$\Spab(X)\subsetneqq X$. 
	\end{thmlist}
\end{conjecture} 

The first two sets $\Spab(X)$ and $\Sph(X)$ are introduced by Lang for the compact case.
He made the following two conjectures (cf. \cite[I, 3.5]{Lan97} and \cite[VIII, Conjecture 1.3]{Lan97}):
\begin{itemize}
	\item
	$\Spab(X)\subsetneqq X$ if and only if $X$ is of general type.
	\item
	$\Spab(X)=\Sph(X)$.
\end{itemize}
The first assertion implicitly include the following third conjecture:
\begin{itemize}
	\item
	$\Spab(X)=\Spalg(X)$.
\end{itemize}  
The original two conjectures imply the famous strong Green-Griffiths conjecture that varieties of (log) general type are pseudo Brody hyperbolic.  

Let us discuss some known cases of the conjecture.
\cref{conj:GGL} is known to hold for curves. However, when $\dim X \geq 2$, there has been only few progress. We summarize below four classes of varieties for which \cref{conj:GGL} has been established:

\begin{itemize}
	\item Complex projective surfaces with big cotangent bundles, proved by McQuillan \cite{McQ}, and later extended to quasi-projective surfaces with big logarithmic cotangent bundles by El Goul \cite{ElG}.
	\item Subvarieties of abelian varieties, by the classical theorem of Bloch–Ochiai–Kawamata, and more generally subvarieties of semi-abelian varieties, by Noguchi \cite{Nog81}.
	\item Projective varieties of maximal Albanese dimension, by Kawamata \cite{Kaw81} and the third author \cite{Yam15}.
	\item General hypersurfaces in projective space $\bP^n$ ($n \geq 3$) of sufficiently high degree, proved in \cite{DMR} (based on the strategy of Siu \cite{Siu02}), with degree bounds subsequently improved in \cite{Dar16,MT,BK24,Cad24}; as well as complements of general hypersurfaces of high degree in $\bP^n$ ($n \geq 2$), proved in \cite{Dar16,BD19}, to mention only a few.
\end{itemize}

In this paper, we establish \cref{conj:GGL} for quasi-projective varieties of maximal quasi-Albanese dimension 
(see \cref{cor:20221102}), thus generalizing the result of \cite{Yam15} into the non-compact setting.
Here we say that a quasi-projective manifold $X$ has maximal quasi-Albanese dimension if the quasi-Albanese map $\alpha:X\to A(X)$ is generically finite onto $\overline{\alpha(X)}$.
By the universal property of the quasi-Albanese maps, this is equivalent to the existence of a morphism $a:X\to A$ to a semi-Abelian variety $A$ such that $\dim X=\dim a(X)$.

\subsection{Main results}
We consider a Riemann surface $Y$ equipped with a proper surjective holomorphic map 
$\pi:Y\to\mathbb C_{>\delta}$, 
where $\mathbb C_{>\delta}:=\{z\in\bC \mid |z|>\delta\}$
for some positive constant $\delta>0$.
	Our first result is a Big Picard type theorem for holomorphic maps $f:Y \to X$, where $X$ is a smooth complex quasi-projective variety of log-general type and of maximal quasi-Albanese dimension.
	This theorem holds provided that $\pi:Y\to\mathbb C_{>\delta}$ has small ramifications compared with the growth order of $f$, i.e., $N_{\ram\pi}(r)=O(\log r)+o(T_f(r))||$.
	The notations are from Nevanlinna theory; specifically, the symbol $||$ means that the stated estimate is valid as $r\to\infty$ with the possible exception of a set of finite Lebesgue measure. 
	We will introduce other notations in \cref{subsec:notion Nevanlinna}.

	 	\begin{thmx}\label{thm2nd}
		Let $X$ be a smooth quasi-projective variety which is of log general type.
		Assume that there is a morphism $a:X\to A$ to a semi-Abelian variety $A$ such that $\dim X=\dim a(X)$.
		Then there exists a proper Zariski closed set $\Xi\subsetneqq X$ with the following property:
		Let $f:Y\to X$ be a holomorphic map such that $N_{\ram\pi}(r)=O(\log r)+o(T_f(r))||$ and that $f(Y)\not\subset \Xi$.
		Then $f$ does not have essential singularity over $\infty$, i.e.,	there exists an extension $\overline{f}:\overline{Y}\to\overline{X}$ of $f$, where $\overline{Y}$ is a Riemann surface such that  $\pi:Y\to \mathbb C_{>\delta}$ extends to a proper map $\overline{\pi}:\overline{Y}\to \mathbb C_{>\delta}\cup\{\infty\}$  
		and $\overline{X}$ is a compactification of $X$.	
	\end{thmx}

	Note that the identity map $\mathbb C_{>\delta}\to \mathbb C_{>\delta}$ trivially satisfies the small ramification condition in the theorem.
	Thus, we can apply \cref{thm2nd} when $Y$ is $\mathbb C_{>\delta}$ itself, which implies that $X$ in the theorem is Picard hyperbolic. 
 Nevertheless, for further applications it is essential to study Nevanlinna theory not only on $\mathbb{C}_{>\delta}$ itself, but also on its ramified coverings. 
This broader perspective is crucial for establishing the hyperbolicity of algebraic varieties. 
Such applications appear, for instance, in the recent work of Kebekus and Rousseau \cite{KR24} on the hyperbolicity of Campana orbifolds (or $\mathcal{C}$-pairs), as well as in Part~(II) of this paper series, where we study the hyperbolicity of complex quasi-projective varieties carrying a big and semi-simple complex local system (see \cite[Theorems~A, C \& D]{CDY22}).  
In particular, for further applications  in \cite{CDY22}, we will make use of the following consequence of \cref{thm2nd}.
See \cref{sec:11} for more discussion about the condition stated in the theorem.

\begin{thmx}[=\cref{thm:20250911}]\label{thm:theorem b}
	Let $X$ be a smooth quasi-projective variety with a smooth projective compactification $\overline{X}$ such that the boundary divisor $D=\overline{X}\setminus X$ is a simple normal crossing divisor. 
	Assume that there exists a finite surjective morphism $p:\overline{\Sigma}\to \overline{X}$ from a normal projective variety $\overline{\Sigma}$ together with non-zero sections $\tau_1,\ldots,\tau_l \in H^0(\overline{\Sigma}, p^*\Omega_{\overline{X}}(\log D))$ satisfying the following two conditions:
	\begin{thmenum} 
		\item  
		setting $\Sigma=p^{-1}(X)$, the variety $\Sigma$ is of log-general type and admits a morphism $a:\Sigma\to A$ into a semi-abelian variety $A$ with $\dim \Sigma=\dim a(\Sigma)$;
		\item  
		defining
		\[ 
		R:=\{s\in \overline{\Sigma} \mid \exists\, i \ \text{such that } \tau_i(s)=0 \}, 
		\]
		then $R\subsetneqq \overline{\Sigma}$ is a proper Zariski closed subset, and $p:\overline{\Sigma}\to \overline{X}$ is \'etale outside $R$.
	\end{thmenum}
	Then every holomorphic map $f:\mathbb D^*\to X$ with Zariski dense image has a holomorphic extension $\bar{f}:\mathbb D\to\overline{X}$.
\end{thmx}

	Next we state our result on \Cref{conj:GGL}, which we prove as a consequence of \cref{thm2nd}.
	\begin{thmx}\label{cor:20221102}
		Let $X$ be a smooth quasi-projective variety. 
		Assume that there is a morphism $a:X\to A$ to a semi-Abelian variety $A$ such that $\dim X=\dim a(X)$.
		Then the following properties are equivalent:  
		\begin{enumerate}[wide = 0pt,  noitemsep,  font=\normalfont, label=(\alph*)] 
			\item \label{being general type1} $X$ is of log general type; 
			 	\item \label{strong LGT1} 
			$X$ is strongly of log general type;  
			\item  \label{pseudo Picard1} $X$ is pseudo Picard hyperbolic;
			\item \label{pseudo Brody} $X$ is pseudo Brody hyperbolic;
			\item \label{spab}
			$\Spab(X)\subsetneqq X$. 
		\end{enumerate}
	\end{thmx}

	It is worth noting that the equivalence of \ref{being general type1} and \ref{strong LGT1} is a purely algebro-geometric statement.
	However our implication \ref{being general type1}$\implies$\ref{strong LGT1} involves the analytic statement \ref{being general type1}$\implies$\ref{pseudo Picard1} (namely, \cref{thm2nd}) and is therefore not purely algebro-geometric.
	\cref{cor:20221102} also plays a crucial role in the second part of this paper series \cite{CDY22}.

	\medspace
	
The paper is organized as follows. We first prove some results on varieties with maximal quasi-Albanese dimension and logarithmic Kodaira dimension zero, and then deduce \cref{cor:20221102} from \cref{thm2nd}. After introducing the necessary notations and estimates from Nevanlinna theory in \cref{subsec:notion Nevanlinna} and \cref{sec:20250904}, we present the proof of \cref{thm2nd} in \cref{subsec:4.2}–\cref{subsec:4.7}. 
The proof relies on the arguments of \cite{NWY13} and \cite{Yam15}.
In \cite{NWY13}, the non-compact setting is treated, while it only considers entire curves $\bC\to X$ with Zariski dense images. In contrast, \cite{Yam15} handles holomorphic maps with non-Zariski dense images, but is restricted to the compact setting.
Our situation requires us to bridge these two cases.
In \cref{rem:20250904}, we briefly discuss the additional difficulties of our proof compared to the proofs in \cite{NWY13} and \cite{Yam15}, coming from the simultaneous consideration of non-compactness and non-Zariski dense images.
Finally in \S \ref{sec:11}, we prove \cref{thm:theorem b} using \cref{thm2nd}.

This paper forms the first part of the long preprint on arXiv \cite{CDY22original}, which has been split into three parts for journal submission.  
It corresponds to Sections 3 and 4 of \cite{CDY22original}.
This paper treats the part of the series which is mainly based on Nevanlinna theory.
 In Part Two, we shall treat Sections 2 and 5–9, and in Part Three, Sections 10–12 of \cite{CDY22original}.

\subsection*{Acknowledgment.}  
B.C. and Y.D. acknowledge support from the ANR grant Karmapolis (ANR-21-CE40-0010).  
K.Y. acknowledges support from JSPS Grant-in-Aid for Scientific Research (C) 22K03286.

 \section{Varieties with logarithmic Kodaira dimension zero and maximal quasi-Albanese dimension} \label{sec:quasi} 
We first recall a fundamental result for semiabelian varieties. 
\begin{lem}[\protecting{\cite[Propositions 5.6.21 \& 5.6.22]{NW13}}] \label{prop:Koddimabb}
	Let \(\cA\) be a semi-abelian variety. 
	\begin{enumerate}[label=(\alph*)]
		\item Let \(X \subset \cA\) be a closed subvariety. Then \(\overline{\kappa}(X) \geq 0\) with equality if and only if \(X\) is a translate of a semi-abelian subvariety.
		\item Let \(Z \subset \cA\) be a Zariski closed subset. Then \(\overline{\kappa}(\cA - Z) \geq 0\) with equality if and only if \(Z\) has no component of codimension \(1\).  \qed
	\end{enumerate}
\end{lem}

Next we prove two lemmas that will be used to discuss the structures on varieties with maximal quasi-Albanese dimension and logarithmic Kodaira dimension zero.

\begin{lem}  \label{lem:same Kd}
	Let $f:X\to Y$ be a dominant birational morphism of  quasi-projective manifolds. 
	Let $E\subset X$ be a Zariski closed subset such that $\overline{f(E)}$ has codimension at least two.  
	Assume $\bar{\kappa}(Y)\geq 0$.
	Then $\bar{\kappa}(X-E)=\bar{\kappa}(X)$. 
\end{lem}
\begin{proof}   
	Let $\bar{f}:\bar{X}\to \bar{Y}$ be a proper birational morphism which extends $f:X\to Y$, where $\bar{X}$ and $\bar{Y}$ are smooth projective compactifications of $X$ and $Y$ such that   $B=\bar{Y}-Y$ is a simple normal crossing divisor. Since the logarithmic Kodaira dimension is  birationally invariant, we may further blow-up  \(\bar{X}\) to assume that 
	\begin{itemize}
		\item $\bar{E}+(\bar{X}-X)$ is a simple normal crossing  divisor,
	\end{itemize} 
	where $\bar{E}\subset \bar{X}$ is the closure of $E$.
	Then we note that $\bar{X}-X$ is also simple normal crossing.
	Write $\bar{X}-X=\bar{f}^{-1}(B)+D$, where $D$ is a reduced divisor on $\bar{X}$ and $\bar{f}^{-1}(B)=\mathrm{supp}\bar{f}^*B$.
	Then we have $f^*K_{\bar{Y}}(B)+\bar{E}\leq K_{\bar{X}}(\bar{f}^{-1}(B))$.
	So we write $K_{\bar{X}}(\bar{f}^{-1}(B))=f^*K_{\bar{Y}}(B)+\bar{E}+F$, where $F$ is effective.
	Then we have
	$$
	K_{\bar{X}}(\bar{f}^{-1}(B))+D+\bar{E}=\bar{f}^*K_{\bar{Y}}(B)+D+2\bar{E}+F
	$$
	By the assumption $\bar{\kappa}(Y)\geq 0$, one has $nK_{\bar{Y}}(B)\geq 0$ for some positive integer $n>0$.
	Therefore,
	$$
	n(\bar{f}^*K_{\bar{Y}}(B)+D+2\bar{E}+F)\leq 2n(\bar{f}^*K_{\bar{Y}}(B)+D+\bar{E}+F)=2n(K_{\bar{X}}(\bar{f}^{-1}(B))+D).
	$$
	Thus 
	$n(K_{\bar{X}}(\bar{f}^{-1}(B))+D+\bar{E})\leq 2n(K_{\bar{X}}(\bar{f}^{-1}(B))+D) $. 
	Recall that $\bar{f}^{-1}(B)+D+\bar{E}$ and $\bar{f}^{-1}(B)+D$ are both simple normal crossing divisors.  	It follows 
	$\bar{\kappa}(X-E)\leq \bar{\kappa}(X)$. 
	Hence $\bar{\kappa}(X-E)=\bar{\kappa}(X)$.  
\end{proof}

\begin{lem}\label{lem:abelian pi0}
	Let $\alpha:X\to \cA$ be  a  (possibly non-proper)  birational morphism  from a quasi-projective manifold $X$ to a semi-abelian variety $\cA$ with   $\overline{\kappa}(X)=0$.  Then  there exists a Zariski closed subset $Z\subset \cA$ of codimension at least two such  that $\alpha$ is   isomorphic
	over $\cA\backslash Z$.   
\end{lem}

\begin{proof}
	Since $\alpha:X\to \cA$ is birational,  we remove the exceptional locus of $X\to \mathcal{A}$ from $X$ to get $X^\circ\subset X$ such that $\alpha: X^\circ\to \alpha(X^\circ)$ is an isomorphism. 
	By $\bar{\kappa}(\cA)=0$, we may apply \cref{lem:same Kd} to get $\bar{\kappa}(X^\circ)=\bar{\kappa}(X)=0$.  
	By \cref{prop:Koddimabb}, $\mathcal{A}-\alpha(X^\circ)$ has codimension  at least two. \end{proof}

We establish several results concerning structures on varieties with maximal quasi-Albanese dimension and logarithmic Kodaira dimension zero.

\begin{lem}\label{lem:abelian pi}
	Let $\alpha:X\to \cA$ be  a  (possibly non-proper)  morphism  from a quasi-projective manifold $X$ to a semi-abelian variety $\cA$ with   $\overline{\kappa}(X)=0$. 
	Assume that $\dim X=\dim \alpha(X)$.	
	Then  $ \pi_1(X)$ is   abelian.  
\end{lem}
\begin{proof}
	{\em Step 1. We may assume that \(\alpha\) is dominant and birational.}
	Consider the quasi-Stein factorisation $X\stackrel{h}{\to} Y\stackrel{k}{\to} \cA$  of $\alpha$, where $h$  is birational (might not proper), and $k$ is a finite morphism.  
	This is constructed as follows.
	Let $X\subset X'$ be a partial smooth compactification of $X$ such that $\alpha$ extends to a proper map $\alpha':X'\to \cA$.
	Let $X'\to Y\to\cA$ be the Stein factorization of $\alpha'$, where $h':X'\to Y$ have connected fibers and $k:Y\to\cA$ is finite.
	By $\dim X'=\dim Y$, the map $h'$ is birational.
	Let $h:X\to Y$ be the restriction of $h'$ to $X$.
	Then $h$ is birational.

	Now since   $\bar{\kappa}(X)= 0$, one has $$0=\bar{\kappa}(X)\geq\bar{\kappa}(Y)\geq \bar{\kappa}(k(Y))\geq 0$$ where the last inequality follows from \cref{prop:Koddimabb}.     
	Hence $\bar{\kappa}(Y)=\bar{\kappa}(k(Y))=0$.  By  \cref{prop:Koddimabb}    and Kawamata \cite[Theorem~26]{Kaw81},  $k(Y)$ is a semi-abelian variety and  $k$ is finite \'etale. In conclusion, it now suffices to prove the proposition with \(\cA\) replaced by \(Y\). 
	Hence in the following, we may assume that $\alpha$ is dominant and birational.
	
	\noindent
	{\em Step 2.}
	We apply \cref{lem:abelian pi0} to get the isomorphism $\alpha|_{X^{\circ}}:X^{\circ}\to \cA^{\circ}$, where $\cA-\cA^\circ$ of codimension at least two.
	Let $\bar{\alpha}:\overline{X}\to \cA$  be a  proper birational morphism which extends $\alpha:X\to \cA$ such that $\overline{X}$ is smooth.   One gets the following commutative diagram
	\begin{equation*}
		\begin{tikzcd}
			\pi_1( X^{\circ})\arrow[r, "\simeq"] \arrow[d] & \pi_1(\cA^{\circ})\arrow[dd, "\simeq"] \\
			\pi_1(X)  \arrow[d]& \\
			\pi_1(\bar{\alpha}^{-1}(\cA)) \arrow[r, "\simeq"] & \pi_1(\cA)
		\end{tikzcd}
	\end{equation*}
	where the two rows are isomorphisms because they are induced by   proper birational morphisms  between smooth quasi-projective varieties. The map \(\pi_{1}(X^\circ) \to \pi_{1}(X)\) is surjective since it is induced by the inclusion of a dense Zariski open subset. It follows that all the groups in the previous diagram are isomorphic, so \(\pi_{1}(X) \cong \pi_{1}(\cA)\) is abelian.
\end{proof}

\begin{lem}\label{lem:ZD}
	Let $\alpha:X\to \cA$ be  a  (possibly non-proper)  morphism  from a quasi-projective manifold $X$ to a semi-abelian variety $\cA$ with   $\overline{\kappa}(X)=0$. 
	Assume that $\dim X=\dim \alpha(X)$ and $\dim X>0$.
	Then $X$ admits a Zariski dense entire curve $\mathbb C\to X$. 
	In particular, $\Sph(X)=X$.
\end{lem}

\begin{proof} 
	As in the step 1 of the proof of \cref{lem:abelian pi}, we may assume that $\alpha:X\to \cA$ is birational.
	We apply \cref{lem:abelian pi0} to get the isomorphism $\alpha|_{X^{\circ}}:X^{\circ}\to \cA^{\circ}$, where $\cA-\cA^\circ$ of codimension at least two.
	Set $Z=\cA-\cA^{\circ}$.
	We shall show that $\cA^{\circ}=\cA-Z$ contains a Zariski dense entire curve.
	Let $\varphi:\mathbb C\to \cA$ be a one parameter group such that $\varphi(\mathbb C)\subset \cA$ is Zariski dense.
	We define $F:\mathbb C\times Z\to \cA$ by $F(c,z)=z+\varphi(c)$.
	Then $F$ is a holomorphic map.
	Let $Z_1\subset Z_2\subset \cdots$ be a sequence of compact subsets of $Z$ such that $\cup_nZ_n=Z$.
	Let $K_n=\{|z|\leq n\}\subset \mathbb C$.
	Then $F(K_n\times Z_n)\subset \cA$ is a compact subset.
	Set $O_n=\cA\backslash F(K_n\times Z_n)$.
	Then $O_n\subset \cA$ is an open subset.
	Note that $\dim (\mathbb C\times Z)<\dim \cA$.
	Hence $O_n$ is a dense subset.
	Hence by the Baire category theorem, the intersection $\cap_nO_n$ is dense in $\cA$.
	In particular, we may take $x\in \cap_nO_n$.
	We define $f:\mathbb C\to \cA$ by $f(c)=x+\varphi(c)$.
	Then $f(\mathbb C)\cap Z=\emptyset$.
	Indeed suppose contary $f(c)\in Z$.
	Then $F(-c,f(c))=f(c)+\varphi(-c)=x$.
	Since $(-c,f(c))\in K_n\times Z_n$ for sufficiently large $n$, we have $x\in F(K_n\times Z_n)$, so $x\not\in O_n$.
	This contradicts the choice of $x$.  
\end{proof}

\begin{lem}\label{lem:20230509} 
	Let $\alpha:X\to \cA$ be  a  (possibly non-proper)  morphism  from a quasi-projective manifold $X$ to a semi-abelian variety $\cA$ with   $\overline{\kappa}(X)=0$. 
	Assume that $\dim X=\dim \alpha(X)$ and $\dim X>0$.
	Then $\Spab(X)=X$.
\end{lem}

\begin{proof}
	As in the step 1 of the proof of \cref{lem:abelian pi}, we may assume that $\alpha:X\to \cA$ is birational.
	We apply \cref{lem:abelian pi0} to get the isomorphism $\alpha|_{X^{\circ}}:X^{\circ}\to \cA^{\circ}$, where $\cA-\cA^\circ$ of codimension at least two.
	Then we have the inverse $\cA^{\circ}\to X$ of the isomorphism $X^{\circ}\to \cA^{\circ}$.
	Thus $\Spab(X)=X$.
\end{proof}
 
	\section{Proof of \cref{cor:20221102} from \cref{thm2nd}}  
We now apply the results of \cref{sec:quasi} to deduce \cref{cor:20221102} from \cref{thm2nd}. We begin with some general facts about special subsets.
\begin{lem}\label{lem:inclusion} 
	Let $X$ be a smooth quasi-projective   variety. Then  one has  $\Spab\subseteq\Sph\subseteq\Spp$. 
\end{lem}
\begin{proof}
Let $A_0\subset A$ be a Zariski open subset of a non-trivial semi-Abelian variety $A$ such that the complement $A\backslash A_0$ has codimension at least two.
By \cref{prop:Koddimabb}, we have $\bar{\kappa}(A_0)=0$.
Thus $A_0$ admits a Zariski dense entire curve $\mathbb C\to A_0$ by \cref{lem:ZD}, which shows $\Spab\subseteq\Sph$.
	
	For any non-constant holomorphic map $f:\bC\to X$,   the holomorphic map 
	\begin{align*}
		g:\bC^*\to X\\
		z\mapsto f(\exp(\frac{1}{z}))
	\end{align*} has essential singularity at 0. Note that the Zariski closure of $g(\bD^*)$ coincides with the Zariski closure of $f$. It follows that   $\Sph\subseteq\Spp$. 
\end{proof}

\begin{proof}[Proof of \cref{cor:20221102}]  \ref{being general type1}$\implies$\ref{pseudo Picard1}.  It follows from \cref{thm2nd} directly. 
	
	\ref{pseudo Picard1}$\implies$\ref{pseudo Brody}. 	This follows from \cref{lem:inclusion}.

	\ref{pseudo Brody}$\implies$\ref{spab}.
	This follows from \cref{lem:inclusion}.
	
	\ref{spab}$\implies$\ref{strong LGT1}.  
	Let $E\subsetneqq X$ be a proper Zariski closed subset such that $a|_{X\backslash E}:X\backslash E\to A$ is quasi-finite.
	Set $\Xi=\Spab(X)\cup E$.
	Under the hypothesis  \ref{spab}, we have $\Xi\subsetneqq X$. 
	
	Now we prove that all closed subvarieties $V\subset X$ with $V\not\subset \Xi$ are of log-general type.
	Note that $\dim V=\dim a(V)$. 
	Hence $\bar{\kappa}(V)\geq 0$ by \cref{prop:Koddimabb}. 
	We prove that the logarithmic Iitaka fibration $V\dashrightarrow W$ is birational.
	So assume contrary that a very general fiber $F$ is positive dimensional.
	We have $F\not\subset \Xi$, hence $F\not\subset \Spab(X)$, and the logarithmic Kodaira dimension satisfies $\bar{\kappa}(F)=0$.
	Then $a|_F:F\to A$ satisfies $\dim a(F)=\dim F$. 
	Hence by \cref{lem:20230509}, we have $\Spab{F}=F$.
	This contradicts to $F\not\subset \Spab(X)$.
	Hence $V$ is of log general type. 
	
	\ref{strong LGT1}$\implies$\ref{being general type1}.  This is obvious. 
\end{proof} 
Therefore, \cref{thm2nd}   proves  \cref{conj:GGL} for quasi-projective varieties with maximal quasi-Albanese dimension. 
After the preparations in \cref{subsec:4.2}–\cref{subsec:4.7a}, we prove \cref{thm2nd} in \cref{subsec:4.7}.

	\section{Some notions in Nevanlinna theory}\label{subsec:notion Nevanlinna}
	We shall recall some elements of Nevanlinna theory (cf. \cite{NW13}, \cite{yamanoi2015kobayashi}).  
	Let $\delta>0$ be a positive constant and let $Y$ be a Riemann surface equipped with a proper surjective holomorphic map 
$\pi:Y\to\mathbb C_{>\delta}$.
	For $r>2\delta$, define
	$
	Y(r)=\pi^{-1}\big( \mathbb C_{>2\delta}(r) \big)
	$ 
	where $\mathbb C_{>2\delta}(r)=\{z\in \bC\mid  2\delta<|z|<r\}$. 
	In the following, we always assume that $r>2\delta$.
	The \emph{ramification counting function} of the covering $\pi:Y\to  \bC_{>\delta}$ is defined by
	$$
	N_{{\rm ram}\,  \pi}(r):=\frac{1}{{\rm deg} \pi}\int_{2\delta}^{r}\left[\sum_{y\in {Y}(t)} \ord_y \ram \pi \right]\frac{dt}{t},
	$$ 
	where $\ram \pi\subset Y$ is the ramification divisor of $\pi:Y\to\mathbb C_{>\delta}$.

	Let $X$ be a projective variety and let $Z$ be a closed subscheme of $X$. Let $f: Y \rightarrow X$ be a holomorphic map such that $f(Y)\not\subset Z$. 
	Since $Y$ is one dimensional, the pull-back $f^{*} Z$ is a divisor on $Y$. The counting function and truncated function are defined to be 
	$$
	\begin{gathered}
		N_f(r, Z):=\frac{1}{\operatorname{deg}  \pi} \int_{2\delta}^{r}\left[\sum_{y \in Y(t)} \operatorname{ord}_{y} f^{*} Z\right] \frac{d t}{t}, \\
		{N}_f^{(k)}(r, Z):=\frac{1}{\operatorname{deg} \pi} \int_{2\delta}^{r}\left[\sum_{y \in Y(t)} \min \left\{k, \operatorname{ord}_{y} f^{*} Z\right\}\right] \frac{d t}{t} ,
	\end{gathered}
	$$
	where $k\in\mathbb Z_{\geq 1}$ is a positive integer.
	We define the proximity function by
	$$
	m_f(r, Z):=\frac{1}{\operatorname{deg} \pi} \int_{y\in \pi^{-1}(\{|z|=r\})}\lambda_Z(f(y))\ \frac{d\mathrm{arg}\ \pi(y)}{2\pi},
	$$
	where $\lambda_Z:X-\mathrm{supp}\ Z\to \mathbb R_{\geq 0}$ is a Weil function for $Z$ (cf. \cite[Prop 2.2.3]{Yam04}).

	Let $L$ be a line bundle on $X$.
	Let $f:Y\to X$ be a holomorphic map.
	We define the order function $T_f(r,L)$ as follows.
	First suppose that $X$ is smooth.
	We equip with a smooth hermitian metric $h_L$, and let $c_1(L,h_L)$ be the curvature form of $(L, h_L)$.  Set
	$$
	T_f(r,L):=\frac{1}{\operatorname{deg} \pi} \int_{2\delta}^{r}\left[\int_{Y(t)} f^{*}c_1(L,h_L)\right] \frac{d t}{t}.
	$$
	This definition is independent of the choice of  the hermitian metric up to a function  $O(\log r)$.  
	For the general case, let $V\subset X$ be the Zariski closure of $f(Y)$ and let $V'\to V$ be a smooth model of $V$.
	This induces a morphism $p:V'\to X$. 
	We have a natural lifting $f':Y\to V'$ of $f:Y\to V$.
	Then we set
	$$
	T_f(r,L):=T_{f'}(r,p^*L)+O(\log r).
	$$
	This definition does not depend on the choice of $V'\to V$ up to a function  $O(\log r)$.

	\begin{thm}[First Main Theorem]\label{thm:first}
		Let $X$ be a projective variety and let $D$ be an effective Cartier divisor on $X$.
		Let $f:Y\to X$ be a holomorphic curve such that $f(Y)\not\subset D$.
		Then
		$$N_f(r,D)+m_f(r,D)=T_f(r,\mathcal{O}_X(D))+O(\log r).$$
	\end{thm}
	
	\begin{proof}
		Let $V\subset X$ be the Zariski closure of $f(Y)$ and let $p:V'\to V$ be a desingularization.
		Let $f':Y\to V'$ be the canonical lifting of $f$.
		Replacing $X$, $D$ and $f$ by $V'$, $p^*D$ and $f'$, respectively, we may assume that $X$ is smooth.
		\par

		Let $||\cdot ||$ be a smooth Hermitian metric on $\mathcal{O}_X(D)$ and let $s_D$ be the associated section for $D$.
		By the Poinca{\'e}-Lelong formula, we have 
		$$2dd^c\log (1/||s_D\circ f(y)||)=-\sum_{y\in Y}(\mathrm{ord}_yf^{*}D)\delta _y+f^{*}c_1(\mathcal{O}_X(D),||\cdot ||),$$
		where $\delta_y$ is Dirac current suported on $y$.
		Integrating over $Y(t)$, we get 
		$$
		2\int_{Y(t)}dd^c\log (1/||s_D\circ f(y)||)=-\sum_{y\in Y(t)}\mathrm{ord}_yf^{*}D+\int_{Y(t)}f^{*}c_1(\mathcal{O}_X(D), ||\cdot ||)
		$$
		Hence, we get
		\begin{equation*}
			\begin{split}
				-N_f(r,D)+T_f(r,\mathcal{O}_X(D))
				&=\lim_{\delta'\to \delta+}\frac{2}{\deg\pi}\int_{2\delta'}^r\frac{dt}{t}\int_{\pi^{-1}\big( \mathbb C_{>2\delta'}(t) \big)}dd^c\log \left( \frac{1}{||s_D\circ f(z)||}\right)\\
				&=\lim_{\delta'\to \delta+}\frac{2}{\deg\pi}\int_{2\delta'}^r\frac{dt}{t}\int_{\partial \pi^{-1}\big( \mathbb C_{>2\delta'}(t) \big)}d^c\log \left( \frac{1}{||s_D\circ f(z)||}\right)\\
				&=m_f(r,D)-m_f(2\delta,D)-C_f\log (r/2\delta),
			\end{split}
		\end{equation*}
		where we set
		$$
		C_f=\lim_{\delta'\to \delta+}\frac{2}{\deg\pi}\int_{\pi^{-1}\{ |z|=2\delta'\}}d^c\log \left( \frac{1}{||s_D\circ f(z)||}\right),$$
		which is a constant independent of $r$. 
	\end{proof}

	For any effective divisor $D_L \in  |L|$ with $f(Y)\not\subset D_L$,  the First Main theorem (cf. \cref{thm:first}) implies the following Nevanlinna inequality:
	\begin{align}\label{eqn:20221120}
		N_f(r, D_L) \leqslant T_f(r, L)+O(\log r) .
	\end{align}
	When $L$ is an ample line bundle on $X$, then we have $T_f(r,L)+O(\log r)>0$.
	If $L'$ is another ample line bundle on $X$, then we have $T_f(r,L')=O(T_f(r,L))+O(\log r)$.
	We write $T_f(r)$ for short instead of $T_f(r,L)$ when the order of magnitude as $r\to\infty$ is concerned.

	\begin{lem}\label{lem:20230411}
		Let $X$ be a smooth projective variety and let $f:Y\to X$ be a holomorphic map.
		Assume that $T_f(r)=O(\log r)||$ and $N_{\ram\pi}(r)=O(\log r)||$. 
		Then $f$ does not have essential singularity over $\infty$.
	\end{lem}
	
	Here, we use the symbol $||$ to denote that the stated estimate holds for $r\to\infty$  outside an exceptional set of finite Lebesgue measure.

	\begin{proof}
		For $s>2\delta$, we set $\nu(s)=\sum_{y\in {Y}(s)} \ord_y \ram \pi$.
		Then for $2\delta <s<r$, we have
		\begin{equation}\label{eqn:20230410}
			\begin{split}
				N_{{\rm ram}\,  \pi}(r)-N_{{\rm ram}\,  \pi}(s)
				&=\frac{1}{{\rm deg} \pi}\int_{s}^{r}\left[\sum_{y\in {Y}(t)} \ord_y \ram \pi \right]\frac{dt}{t}
				\\
				&\geq \frac{1}{{\rm deg} \pi}\int_{s}^{r}\nu(s)\frac{dt}{t}=\frac{\nu(s)}{{\rm deg} \pi}(\log r-\log s).
			\end{split}
		\end{equation}
		Set $K=\varliminf_{r\to\infty}N_{\ram\pi}(r)/\log r$.
		By $N_{\ram\pi}(r)=O(\log r)||$, we have $K<\infty$.
		By \eqref{eqn:20230410}, we have $\nu(s)\leq K{\rm deg} \pi$.
		Since $s>2\delta$ is arbitraly, the ramification divisor $\ram\pi\subset Y$ consists of finite points.
		Thus we may take $s_0>\delta$ such that $\pi^{-1}(\mathbb C_{>s_0})\to \mathbb C_{>s_0}$ is unramified covering.
		Thus $\pi^{-1}(\mathbb C_{>s_0})$ is a disjoint union of punctured discs.
		Hence we may take an extension $\overline{\pi}:\overline{Y}\to \mathbb C_{>\delta}\cup\{\infty\}$ of the covering $\pi:Y\to \mathbb C_{>\delta}$.
		
		In the following, we replace $\delta$ by $s_0$ and take a connected component of $\pi^{-1}(\mathbb C_{>s_0})$.
		Then we may assume that $\pi:Y\to \mathbb C_{>\delta}$ is unramified and $Y$ is a punctured disc.
		Let $\omega$ be a smooth positive $(1,1)$-form on $X$.
		For $s>2\delta$, we set $\alpha(s)=\int_{Y(s)}f^*\omega$.
		Then by the similar computation as in \eqref{eqn:20230410}, the assumption $T_f(r)=O(\log r)||$ yields that $\alpha(s)$ is bounded on $s>2\delta$.
		Thus $\int_Yf^*\omega<\infty$.
		By \cref{lem:picard} below, we conclude the proof.
	\end{proof}
	The following lemma is well-known to the experts. We refer the readers to \cite[Lemma 3.3]{CD21} for a simpler proof based on Bishop's theorem. 
	\begin{lem}\label{lem:picard}
		Let $X$ be a compact K\"ahler manifold and let $f:\mathbb D^*\to X$ be a holomorphic map from the punctured disc $\mathbb D^*$.
		Let $\omega$ be a smooth K\"ahler form on $X$.
		Suppose $\int_{\mathbb D^*}f^*\omega<\infty$.
		Then $f$ has holomorphic extension $\bar{f}:\mathbb D\to X$. \qed
	\end{lem}
	
We state the following lemma, which will be applied later.
	
	\begin{lem}\label{lem:20230415}
		Let $X$ be a smooth projective variety and let $f:Y\to X$ be a holomorphic map.
		If $\varliminf_{r\to\infty}T_f(r)/\log r<+\infty$, then $T_f(r)=O(\log r)$.
	\end{lem}

	\begin{proof}
		For $s>2\delta$, we set $\alpha(s)=\int_{Y(s)}f^*\omega$, where $\omega$ is a smooth positive $(1,1)$-form on $X$.
		Then by the similar computation as in \eqref{eqn:20230410}, the assumption $\varliminf_{r\to\infty}T_f(r)/\log r<+\infty$ yields that $\alpha(s)$ is bounded, i.e., there exists a positive constant $C>0$ such that $\alpha(s)<C$ for all $s>2\delta$.
		Hence for all $r>2\delta$, we have
		$$
		\frac{1}{\operatorname{deg} \pi} \int_{2\delta}^{r}\left[\int_{Y(t)} f^{*}\omega\right] \frac{d t}{t}<\frac{C}{\operatorname{deg} \pi}\log \frac{r}{2\delta}.
		$$
		Thus we get $T_f(r)=O(\log r)$.
	\end{proof}

	\section{Lemma on logarithmic derivatives} \label{sec:20250904}
	Let $\pi _Y:Y\to \mathbb C_{>\delta}$ be a Riemann surface with a proper surjective holomorphic map.
	Set $D_Y=\pi_Y^*(\partial /\partial z)$.
	Then $D_Y$ is a meromorphic vector field on $Y$.
	For a meromorphic function $f:Y\to\mathbb P^1$ on $Y$, we set $f'=D_Y(f)$.
	For $r>2\delta$, we set as follows:
	$$
	m(r,f)=\frac{1}{\deg \pi_Y}\int_{y\in \pi_Y^{-1}(\{|z|=r\})}\log \sqrt{1+|f(y)|^2}\frac{d\arg \pi_Y(y)}{2\pi},
	$$
	$$
	T(r,f)=\frac{1}{\deg \pi_Y}\int_{2\delta}^r\frac{dt}{t}\int_{Y(t)}f^*\omega_{\mathrm{F.S.}},
	$$
	where $\omega_{\mathrm{F.S.}}$ is the Fubini-Study form:
	$$
	\omega_{\mathrm{F.S.}}=\frac{1}{(1+|w|^2)^2}\frac{\sqrt{-1}}{2\pi}dw\wedge d\overline{w}.
	$$
	Note that $m(r,f)$ is a proximity function function for $f:Y\to\mathbb P^1$ with respect to the divisor $(\infty)$ on  $\mathbb P^1$, and $
	T(r,f)$ is a order function with respect to the ample line bundle $\mathcal{O}_{\mathbb P^1}(1)$.
	In this section, we provide a simple proof for Nevanlinna's lemma on logarithmic derivatives, \cref{thm:20231119}.
	
	\begin{lem}\label{lem:202311191}
		For $n\in\mathbb Z_{\geq 1}$, we have
		$T(r,f^n)=nT(r,f)+O(\log r)$.
	\end{lem}
	
	\begin{proof}
		Let $\varphi:\mathbb P^1\to\mathbb P^1$ be defined by $\varphi(w)=w^n$, then we have $f^n=\varphi\circ f$.
		By $\varphi^*\mathcal{O}_{\mathbb P^1}(1)=\mathcal{O}_{\mathbb P^1}(n)$, we obtain our lemma.
	\end{proof}
	
	\begin{lem}\label{lem:202311192}
		$m(r,f'/f)\leq 3T(r,f)+O(\log r)\ ||$
	\end{lem}
	
	\begin{proof}
		Set 
		$$
		f^{\#}=\frac{|f'|}{1+|f|^2}
		$$
		and
		$$
		m^{\#}(r,f)=\frac{1}{\deg \pi _Y}\int_{y\in \pi_Y^{-1}(\{|z|=r\})}\log \sqrt{1+(f^{\#}(y))^2} \frac{d \arg \pi _Y(y)}{2\pi}.
		$$
		We first show
		\begin{equation}\label{eqn:1aa}
			m^{\#}(r,f)\leq T(r,f)+O(\log r)\ ||.
		\end{equation}
		Indeed using convexity of $\log$, we have
		\begin{equation*}
			m^{\#}(r,f)\leq \frac{1}{2}\log \left( 1+\frac{1}{\deg \pi _Y}\int_{y\in \pi_Y^{-1}(\{|z|=r\})}f^{\#}(y)^2\frac{d \arg \pi _Y(y)}{2\pi}\right) .
		\end{equation*}
		Using polar coordinate, we get
		$$
		\int_{Y(r)}f^*\omega_{\mathrm{F.S.}}=\frac{1}{\pi}\int_{2\delta}^rtdt\int_{y\in \pi_Y^{-1}(\{|z|=t\})}f^{\#}(y)^2d\arg\pi_Y(y).
		$$
		This shows
		$$
		\frac{1}{2r}\frac{d}{dr}\left( r\frac{d}{dr}T(r,f)\right) =\frac{1}{\deg \pi _Y}\int_{y\in \pi_Y^{-1}(\{|z|=r\})}f^{\#}(y) ^2\frac{d \arg \pi _Y(y)}{2\pi}.
		$$
		Hence using Borel's growth lemma \cite[p. 13]{NW13}, we get
		\begin{equation*}
			\begin{split}
				m^{\#}(r,f)&\leq  \frac{1}{2}\log \left( 1+\frac{1}{2r}\frac{d}{dr}\left( r\frac{d}{dr}T(r,f)\right)\right) \\
				&\leq  \frac{1}{2}\log \left( 1+\frac{1}{2r}\left( r\frac{d}{dr}T(r,f)\right)^{1+\delta}
				\right) \ ||_{\delta}\\
				&\leq  \frac{1}{2}\log \left( 1+\frac{1}{2}r^{\delta}T(r,f)^{(1+\delta )^2}\right) \ ||_{\delta}\\
				&\leq  T(r,f) +O(\log r)\ ||.
			\end{split}
		\end{equation*}
		This shows our estimate \eqref{eqn:1aa}.
		Here we take $\delta =1$ in the final equation.
		
		\par
		
		Now by
		\begin{multline*}
			\log \sqrt{1+\left\vert \frac{1}{f}\right\vert^2}+\log \sqrt{1+|f|^2}+\log \sqrt{ 1+\left( \frac{|f'|}{1+|f|^2}\right) ^2}
			\\
			=\log \sqrt{ \left(|f|+\frac{1}{|f|}\right)^2+\left\vert \frac{f'}{f}\right\vert ^2} 
			\geq
			\log\sqrt{1+\left\vert \frac{f'}{f}\right\vert ^2},
		\end{multline*}
		we get
		$$
		m(r,f'/f)\leq m(r,1/f)+m(r,f )+m^{\#}(r,f).
		$$
		Using the first main theorem (cf. \cref{thm:first}) and \eqref{eqn:1aa}, we get \cref{lem:202311192}.
	\end{proof}
	
	We prove Nevanlinna's lemma on logarithmic derivatives (cf. \cite[Lemma 1.2.2]{NW13}) in the following form.
	
	\begin{thm}\label{thm:20231119}
		$m(r,f'/f)=o(T(r,f))+O(\log r)\ ||.
		$
	\end{thm}
	
	Several methods exist for proving Nevanlinna's lemma on logarithmic derivatives. 
	The proof presented here is an application of the ramified covering method used in \cite[Thm 4.8]{yamanoi2015kobayashi} to a simpler context.

	\begin{proof}
		We first show that for every $\varepsilon>0$, we have
		\begin{equation}\label{eqn:202311193}
			m(r,f'/f)\leq \varepsilon T(r,f))+O_{\varepsilon}(\log r)\ ||_{\varepsilon}.
		\end{equation}
		Here $O_{\varepsilon}$ indicates that the implicit constant in the Landau symbol $O$ may depend on $\varepsilon$.
		
		Indeed we take $n\in\mathbb Z_{\geq 1}$ so that $3/n<\varepsilon$.
		We set $f_n=\sqrt[n]{f}$.
		Then $f_n$ is a multi-valued meromorphic function on $Y$.
		Let $\pi_{Y_n}:Y_n\to \mathbb C_{>\delta}$ be the Riemann surface for $f_n$.
		Both $f$ and $f_n$ are considered as meromorphic funtions on $Y_n$.
		We have $f'/f=nf_n'/f_n$.
		By \cref{lem:202311191}, we get
		$$
		m(r, f'/f)= m(r,f_n'/f_n )+O(1)\leq 3T(r,f_n)+O(\log r)\ ||. 
		$$
		Hence by \cref{lem:202311192}, we have 
		$$
		m(r, f'/f)\leq \frac{3}{n}T(r,f)+O(\log r)\ ||. 
		$$
		This shows \eqref{eqn:202311193}.
		
		Suppose that $\varliminf_{r\to\infty}T(r,f)/\log r<+\infty$.
		Then by \cref{lem:20230415}, we have $T(r,f)=O(\log r)$.
		Then by \eqref{eqn:202311193}	, we have $m(r, f'/f)			=O(\log r)$, in particular $m(r,f'/f)=o(T(r,f))+O(\log r)\ ||$.
		
		Next we assume $\varliminf_{r\to\infty}T(r,f)/\log r=+\infty$.
		Then $\log r=o(T(r,f))$.		
		Hence by \eqref{eqn:202311193}	, we have 
		\begin{equation}\label{eqn:20231120}
			m(r, f'/f)			\leq \ep T(r, f)+	o(T(r,f))\ ||_{\varepsilon}
		\end{equation}
		for all $\varepsilon>0$.
		This implies $m(r, f'/f)=o(T(r,f))\ ||$.
		Indeed we take a sequence $2\delta=r_0<r_1<r_2<\cdots$ with $r_n\to\infty$ as follows.
		By \eqref{eqn:20231120}, we have	$m(r, f'/f)\leq \frac{1}{n}T(r,f)$ for all $r>2\delta$ outside some exceptional set $E_n\subset (2\delta,\infty)$ with $|E_n|<\infty$.
		We take $r_n$ such that $|(r_n,\infty)\cap E_n|<1/2^n$.
		We set $\varepsilon(r)=1/n$ if $r_n\leq r<r_{n+1}$, and $\varepsilon(r)=1$ if $r_0<r<r_1$.
		Then $\varepsilon(r)\to0$ if $r\to\infty$.
		Set $E=(r_0,r_1)\cup\bigcup\left((r_n,r_{n+1})\cap E_n\right)$.
		Then we have $m(r, f'/f)\leq \varepsilon(r)T(r,f)$ for all $r>2\delta$ outside $E$, and $|E|<r_1+1$.
		Hence we get $m(r, f'/f)=o(T(r,f))\ ||$, in particular we get $m(r,f'/f)=o(T(r,f))+O(\log r)\ ||$.
		This conclude the proof of the theorem.	
	\end{proof}
	
	Let $V$ be a smooth projective variety and let $D\subset V$ be a simple normal crossing divisor.
	Let $T(V;\log D)$ be the logarithmic tangent bundle. 
	Set $\overline{T}(V;\log D)=P(T(V;\log D)\oplus \mathcal{O}_V)$, which is a smooth compactification of $T(V;\log D)$.
	Let $\partial T(V;\log D)\subset \overline{T}(V;\log D)$ be the boundary divisor.
	Let $f:Y\to V$ be a holomorphic map such that $f(Y)\not\subset D$.
	Then we get a derivative map $j_1(f):Y\to \overline{T}(V;\log D)$.
	By \cref{thm:20231119}, we obtain the following estimate
	\begin{equation}\label{eqn:202311195}
		m_{j_1 (f)}(r,\partial T(V;\log D))=o(T(r,f))+O(\log r)\ ||.
	\end{equation}
	
	This estimate is a direct modification of \cite[Thm 5.1.7 (2)]{Yam04}, whose proof applies to that of \eqref{eqn:202311195}.
	The original proof uses\cite[Prop 5.1.2 (1)]{Yam04} together with \cite[Thm 2.5.1 (2)]{Yam04}. 
	Our modification is to substitute \cite[Thm 2.5.1 (2)]{Yam04} with \cref{thm:20231119} and then follow the same line of reasoning.
	Here we note that we write $j_1 (f)$ instead of "$j_1^{\log} (f)$" as in \cite{Yam04} for simplicity.
	See also \cite[Appendix A]{Vojta2000abc}.
	 The same modification shows that the corresponding estimate for \cite[Thm 5.1.7 (1)]{Yam04} is also valid with the error term $S(r,f)=o(T(r,f))+O(\log r)\ ||$ by the same manner using \cite[Prop 5.1.2 (2)]{Yam04} combined with \cref{thm:20231119}.
	 This will be used in the estimate \eqref{eqn:20250904}.

Let us discuss an application of \eqref{eqn:202311195} to an estimate for logarithmic 1-forms as in \cite{Nog81}.
Let $D$ be a simple normal crossing divisor on a projective manifold $V$ as above.
Let $f:Y\to V$ be a holomorphic map such that $f(Y)\not\subset D$. 
Let $\omega \in H^{0}\left(V, \Omega_{V}(\log D)\right)$ be a logarithmic \(1\)-form. 
 Set $\xi=f^{*} \omega / \pi^{*}(d z)$. 
 Then $\xi$ is a meromorphic function on $Y$, which gives a holomorphic map $\xi:Y\to \bP^1$.  

We claim 
\begin{align}\label{eq:Noguchi}
	 m(r,\xi)=O(\log r)+o(T(r,f))\lVert. 
\end{align}  
The proof goes as follows. 
The logarithmic one form $\omega$ defines a morphism $\varphi:T(V;\log D)\to \mathbb A^1$, which extends to a rational map $\overline{\varphi}:\overline{T}(V;\log D)\dashrightarrow \mathbb P^1$.
By $f(Y)\not\subset D$, we get the lifting $j_1(f):Y\to \overline{T}(\overline{V};\log D)$.
Then we have $\xi=\overline{\varphi}\circ j_1(f)$.
Hence by \cite[Prop 2.3.2 (6)]{Yam04}, we have
$$
m(r, \xi)=m_{\xi}(r,[\infty])=O(m_{j_1(f)}(r,\partial T(V;\log D)))+O(1).$$
Thus by \eqref{eqn:202311195}, we get \eqref{eq:Noguchi}.

It follows from the First Main theorem (cf. \cref{thm:first}) that
\begin{equation*}
 N_{\xi}(r,[0])+m_{\xi}(r,[0])=N_{\xi}(r,[\infty])+ m_{\xi}(r,[\infty])+O(\log r).
\end{equation*}
	Hence by \eqref{eq:Noguchi}, we get
	\begin{equation}\label{eq:First} 
	N_{\xi}(r,[0])\leq N_{\xi}(r,[\infty])+O(\log r)+o(T(r,f))\lVert. 
	\end{equation}

	\subsection*{Convention and notation.} 
	In the rest of this paper, we use the following convention and notation.
	Let $\Sigma$ be a quasi-projective variety and let $\overline{\Sigma}$ be a projective compactification.
	We denote by $f:Y\da \Sigma$ a holomorphic map $\bar{f}:Y\to \overline{\Sigma}$ such that $\bar{f}^{-1}(\Sigma)\not=\emptyset$.
	Moreover for a Zariski closed set $W\subset \Sigma$, we denote by $f(Y)\sqsubset W$ if $\bar{f}(Y)\subset \overline{W}$, where $\overline{W}\subset \overline{\Sigma}$ is the Zariski closure.
	The choice of $\overline{\Sigma}$ is not canonical. 
	Thus, we will use these notation to avoid mentioning the compactification explicitly.

	Let $A$ be a semi-abelian variety and let $S$ be a projective variety.
	For $f:Y\da A\times S$, we set
	$$
	\pN{f}(r):=\frac{1}{\operatorname{deg}  \pi} \int_{2\delta}^{r}\mathrm{card}\left(Y(t)\cap \bar{f}^{-1}(\partial A\times S)\right)\frac{d t}{t}.
	$$
	This definition does not depend on the choice of $\overline{A}$.
	For $f:Y\da A\times S$, we define $f_A:Y\da A$ and $f_S:Y\to S$ by the compositions of $f$ and the projections $A\times S\to A$ and $A\times S\to S$, respectively.
	See \cref{rem:20250904} for the reason why we need to treat the situation above in the proof.

	\section{Preliminaries for the proof of \cref{thm2nd}}
	\label{subsec:4.2}
	
	Let $V$ be a smooth algebraic variety.
	Let $TV$ be the tangent bundle of $V$.
	Then we have $TV=\mathbf{Spec}\, (\mathrm{Sym}\ \Omega^1_V)$. 
	Set $\overline{T}V=P(TV\oplus \mathcal{O}_V)$, which is a smooth compactification of $TV$.
	Then we have $\overline{T}V=\mathbf{Proj}\, ((\mathrm{Sym}\ \Omega^1_V)\otimes_{\mathcal{O}_V}\mathcal O_V[\eta ])$.
	Let $Z\subset V$ be a closed subscheme.
	We define a closed subscheme $Z^{(1)}\subset \overline{T}(V)$ as follows (cf. \cite[p. 38]{Yam04}): 
	Let $U\subset V$ be an affine open subset.
	Let $(g_1,\cdots ,g_l)\subset \Gamma (U,\mathcal O_V)$ be the defining ideal of $Z\cap U$.
	Then $Z^{(1)}\cap \overline{T}(U)$ is defined by the homogeneous ideal
	$$(g_1,\cdots ,g_l,dg_1\otimes 1,\cdots ,dg_l\otimes 1)\subset \Gamma (U,(\mathrm{Sym}\ \Omega^1 _V)\otimes_{\mathcal O_V}\mathcal O_V[\eta ]).
	$$
	We glue $Z^{(1)}\cap \overline{T}(U)$ to define $Z^{(1)}$.
	If $Z\subset V$ is a closed immersion of a smooth algebraic variety, then $Z^{(1)}\cap T(V)=T(Z)$.

	Let $S$ be a projective variety, and let $W\subset A\times S$ be a Zariski closed set, where $A$ is a semi-Abelian variety.
	Let $q:A\times S\to S$ be the second projection and let $q_W:W\to S$ be the restriction of $q$ on $W$.
	Given $f:Y\da A\times S$, we denote by $\Sigma\subset S$ the Zariski closure of $f_S(Y)\subset S$.
	We define $e(f,W)$ to be the dimension of the generic fiber of $q_W^{-1}(\Sigma)\to \Sigma$.
	We set $e(f,W)=-1$ if the generic fiber of $q_W^{-1}(\Sigma)\to \Sigma$ is an empty set.

	Now assume that $W\subset A\times S$ is irreducible and $\overline{q(W)}=S$.
	We also assume that $\dim \mathrm{St}(W)=0$, where the action $A\curvearrowright A\times S$ is defined by $(a,s)\mapsto (a+\alpha,s)$ for $\alpha\in A$.
	Since $W$ and $S$ are integral, and $q_W$ is dominant, there exists a non-empty Zariski open subset $W^o\subset W$ such that $q_W$ is a smooth morphism over $W^o$.
	Let $S^o\subset S$ be a non-empty Zariski open subset such that every $s\in S^o$ satisfies $q_W^{-1}(s)\not=\emptyset$ and:
	\begin{enumerate}[label=(\arabic*)]
	\item
	every irreducible component of $q_W^{-1}(s)$ has non-trivial intersection with $W^o$, and 
	\item
	the stabilizer of every irreducible component of $q_W^{-1}(s)$ is 0-dimensional.
	\end{enumerate}
	
	To justify the second condition, we argue as follows.
	We note that the assertion is local on $S$.
	Therefore, we may shrink $S$ to ensure that any property holding generically on $S$ holds over all of $S$.
	Let $\overline{A}$ be a projective, equivariant compactification and let $\overline{W}\subset \overline{A}\times S$ be the closure.
	By the generic reduced theorem, the fibers $\overline{W}_s$ are reduced for general $s\in S$ (cf. \cite[Tag 054Z, Tag 020I, Tag 0578]{stacks-project}).
Hence by shrinking $S$, we may assume that all fibers are reduced.
	Let $\overline{W}'\to\overline{W}$ be a desingularization, let $\overline{W}'\to S^{\dagger}\overset{u}{\to} S$ be the Stein factorization and let $V\subset \overline{A}\times S^{\dagger}$ be the image of the induced map $\overline{W}'\to \overline{A}\times S^{\dagger}$.
Note that general fibers of $\overline{W}'\to S^{\dagger}$ are smooth and connected, hence irreducible.
Therefore by shrinking $S$, all fibers of $V\to S^{\dagger}$ are irreducible, and are also reduced by the generic reduced theorem.
Now we have a natural surjection $v:V\to \overline{W}$.
Hence by $\dim \mathrm{St}(W)=0$, we have $\dim \mathrm{St}(V)=0$.
Thus by \cite[Lemma 11.12]{Y22}, we have $\dim \mathrm{St}(V_y)=0$ for general fibers $V_y$, thus for all $y\in S^{\dagger}$ by shrinking $S$.
Then for any $s\in S$ and any irreducible component of $\overline{W}_s$, there exists $y\in S^{\dagger}$ such that $u(y)=s$ and $V_y$ is isomorphic to that component.
This confirms the validity of the second assertion.
	
	Assume $S$ is smooth and projective.
	We note 
	\begin{equation}\label{eqn:20221203}
		\overline{T}(A\times S)=A\times S',
	\end{equation}
	where $S'=\overline{\mathrm{Lie}(A)\times TS}$.
	We may define $e(j_1f,W^{(1)})$ from $j_1f:Y\da A\times S'$ and $W^{(1)}\subset A\times S'$.

	\begin{lem}\label{lem:202210041}
		Assume that $S$ is smooth and projective, and that $W\subset A\times S$ is irreducible with $\overline{q(W)}=S$.
		Let $f:Y\da A\times S$ satisfies $e(j_1f,W^{(1)})=e(f,W)$ and $f_S(Y)\not\subset S- S^o$.
		Then we have
		$$T_{f_A}(r)=O(T_{f_S}(r))+O(\log r).$$
	\end{lem}

	\begin{rem}
		In the statement of the lemma, by $T_{f_A}(r)$ we mean $T_{f_A}(r,L)$ for some ample line bundle on a compactification $\overline{A}$.
		So this notion has ambiguity, but we have $T_{f_A}(r,L')=O(T_{f_A}(r,L))+O(\log r)$ for other choices.
		Hence the term $O(T_{f_A}(r))+O(\log r)$ has fixed meaning.
		Similar for $O(T_{f_S}(r))+O(\log r)$.
	\end{rem}
	
	\begin{proof}[Proof of \cref{lem:202210041}]
		We define $\mathcal{Z}(W)\subset A\times S\times S$ by
		$$
		\mathcal{Z}(W)=\{ (a,s_1,s_2)\in A\times S\times S;\dim (q_W^{-1}(s_1)\cap (a+q_W^{-1}(s_2)))\geq \dim W-\dim S\} ,
		$$
		where we consider $q_W^{-1}(s_1)$ and $q_W^{-1}(s_2)$ as Zariski closed subsets of $A$.
		For $s\in S$, let $\mathcal{Z}_{s}(W)\subset A\times S$ be the fiber of the composite map $\mathcal{Z}(W)\hookrightarrow A\times S\times S\to S$ over $s\in S$, where the second map $A\times S\times S\to S$ is the third projection.

		\begin{claim}\label{claim:20221018}
			Let $s\in S^o$.
			Let $g:\mathbb D \to A\times S$ be a holomorphic map such that $e(j_1g,W^{(1)})=e(g,W)$ with $g(0)=(a,s)\in A\times S^o$.
			Then $g(\mathbb D )\subset \mathcal{Z}_{s}(W)+a$.
		\end{claim}
		
		\begin{proof}
			This is proved in the sublemma of \cite[Lemma 3]{Yam15} when $A$ is compact.
Our argument proceeds along the same lines as that proof.
			
			By $q\circ g(0)=s\in S^o$, we may take a non-empty open subset $\Omega\subset \mathbb D$ such that $q\circ g(\Omega)\subset S^o$. 
		Hence the condition (1) of $S^o$ implies 
		$$e(g,W)=\dim W-\dim S.$$
Let $q'_{W^{(1)}}:W^{(1)}\to S'$ be the restriction of the second projection $q':A\times S'\to S'$ to $W^{(1)}$.
By shrinking $\Omega$ if necessary, the assumption $e(j_1g,W^{(1)})=e(g,W)$ implies that
\begin{equation*}
\dim (q'_{W^{(1)}})^{-1}(q'\circ j_1g(z))=\dim W-\dim S
\end{equation*}
for all $z\in \Omega$.
Indeed denoting by $\Sigma\subset S'$ the Zariski closure of $q'\circ j_1g(\mathbb D)\subset S'$, the set $\{s\in \Sigma; \dim (q'_{W^{(1)}})^{-1}(s)=\dim W-\dim S\}$ is constructible (cf. \cite[Tag 05F9]{stacks-project}).
Hence we may shrink $\Omega$ as desired above.

Now for each $z\in\Omega$, we may take an irreducible component $I_z$ of $(q'_{W^{(1)}})^{-1}(q'\circ j_1g(z))$ such that 
\begin{equation}\label{eqn:20250928}
\dim I_z=\dim W-\dim S.
\end{equation}
Since $(q'_{W^{(1)}})^{-1}(q'\circ j_1g(z))\subset q_W^{-1}(q\circ g(z))$ as Zariski closed subsets of $A$,  we conclude that $I_z$ is an irreducible component of $q_W^{-1}(q\circ g(z))$.
Hence the condition (1) of $S^o$ implies 
\begin{equation}\label{eqn:20250930}
I_z\cap W^o\not=\emptyset.
\end{equation}

Define
$$\psi:\,A\times \mathbb D\to A\times S,\qquad \psi(b,z)=(b,\;q\circ g(z))$$
and set 
$$W^o_{\mathbb D}:=\psi^{-1}(W^o).$$
On $W^o$, the map $q_W$ is smooth, hence a submersion.
Therefore $W^o_{\mathbb D}$ is a (possibly non-connected) smooth complex manifold, and the induced map $q_{W^o_{\mathbb D}}:W^o_{\mathbb D}\to\mathbb D$ is a submersion.
Hence every connected component of $W^o_{\mathbb D}$ has dimension $\dim W-\dim S+1$.

Define an automorphism $\tau:A\times \mathbb D\to A\times \mathbb D$ by
$$\tau(b,z)=(\,b-a+p\circ g(z),\; z\,).$$
Let $\mathcal{X}$ be the holomorphic vector field on $A\times \mathbb D$ given by
$$\mathcal{X}:=\tau_*(0,\partial/\partial z),$$
where $(0,\partial/\partial z)$ is the standard vector field on $A\times \mathbb D$ and
$\tau_*$ denotes the push-forward. For each $b\in A$ set
$$\tau_b:\mathbb D\to A\times \mathbb D,\qquad \tau_b(z)=\tau(b,z).$$
Each map $\tau_b$ is an integral curve of the vector field $\mathcal{X}$.

We shall show that there exists a connected component $\widetilde W^o_{\mathbb D}$ of $W^o_{\mathbb D}$
such that $\mathcal{X}$ is tangent to $\widetilde W^o_{\mathbb D}$ at every point of
$\widetilde W^o_{\mathbb D}$. 
We specify $\widetilde W^o_{\mathbb D}$ as follows.
Set
\[
\mathcal W:=\{x\in W^{(1)};\ \dim_x (q'_{W^{(1)}}{}^{-1}(q'_{W^{(1)}}(x)))\ge \dim W-\dim S\}.
\]
Then $\mathcal W\subset W^{(1)}$ is Zariski closed.
Put
\[
V:=\{(b,z)\in W^o_{\mathbb D} \ ;\ \psi_*(\mathcal{X}_{(b,z)})\in\mathcal W\},
\]
where $\psi_*:T(A\times \mathbb D)\to T(A\times S)$ is the push-forward of $\psi$.
Then $V\subset W^o_{\mathbb D}$ is an analytic subset.
Note that $\psi_*(\mathcal{X}_{(b,z)})$ is a translate
of $j_1(g)(z)\in T_{g(z)}(A\times S)$.
Thus taking into account \eqref{eqn:20250928}, we get $I_z\cap W^o\subset \varphi^{-1}(z)\cap V$ for $z\in \Omega$, where $\varphi:A\times \mathbb D\to \mathbb D$ is the projection.
In particular, by \eqref{eqn:20250930}, we have 
$$\Omega\subset \varphi(V),$$ 
and by the definition of $V$,
$$\dim_x(\varphi^{-1}(z)\cap V)=\dim W-\dim S$$ 
for every $z\in \Omega$ and $x\in \varphi^{-1}(z)\cap V$.
Since $V\subset  W^o_{\mathbb D}$ is an analytic subset, these two conditions imply that $V$ contains a connected component of $W^o_{\mathbb D}$.
We denote this component by $\widetilde W^o_{\mathbb D}$.

Next we show that $\mathcal{X}$ is tangent to $\widetilde W^o_{\mathbb D}$ at every point of
$\widetilde W^o_{\mathbb D}$.
Since $q_W|_{W^o}:W^o\to S$ is submersion, we have the
identity of subbundles of $T(A\times \mathbb D)$:
\begin{equation*}
T W^o_{\mathbb D} \;=\; (\psi_*)^{-1}(T W^o).
\end{equation*}
Therefore it is enough to show
\begin{equation}\label{eq:TWod}
\psi_*(\mathcal{X}_{(b,z)})\in W^{(1)}
\end{equation}
for every $(b,z)\in\widetilde W^o_{\mathbb D}$.
However by the definition of $V$, this holds for all points of $V$, in particular holds for the subset $\widetilde W^o_{\mathbb D}\subset V$.
Therefore $\mathcal{X}$ is tangent to $\widetilde W^o_{\mathbb D}$ at every point of $\widetilde W^o_{\mathbb D}$.

Now we take $\zeta\in \Omega$, so that $\varphi^{-1}(\zeta)\cap\widetilde W^o_{\mathbb D}\not=\emptyset$.
For each $w\in \varphi^{-1}(\zeta)\cap\widetilde W^o_{\mathbb D}$, we set $c=w+a-p\circ g(\zeta)$.
Then $\tau_c(\zeta)=(w,\zeta)$.
Since $\tau_c:\mathbb D\to A\times\mathbb D$ is an integral curve of $\mathcal{X}$ and $\mathcal{X}$ is tangent
to $\widetilde W^o_{\mathbb D}$ at every point of $\widetilde W^o_{\mathbb D}$, we have $\tau_c(U)\subset \widetilde W^o_{\mathbb D}$ for some neibourhood $U\subset \mathbb D$ of $\zeta$.
Put $W_{\mathbb D}:=\psi^{-1}(W)$, then $W_{\mathbb D}\subset A\times \mathbb D$ is an analytic set, and $W^o_{\mathbb D}\subset W_{\mathbb D}$ is an open subset.
Therefore we have $\tau_c(\mathbb D)\subset W_{\mathbb D}$. 
Thus for every $z\in \mathbb D$ and every $w\in \varphi^{-1}(\zeta)\cap\widetilde W^o_{\mathbb D}$, we have
$\tau_c(z)\in \varphi^{-1}(z)\cap W_{\mathbb D}$. 
Equivalently,
\[
\bigl(p\circ g(z)-p\circ g(\zeta)\bigr) + \bigl(\varphi^{-1}(\zeta)\cap\widetilde W^o_{\mathbb D}\bigr)
\ \subset\ \varphi^{-1}(z)\cap W_{\mathbb D}.
\]
Considering this relation for $z=0$ and arbitrary $z$, we obtain
$$
-p\circ g(\zeta) + \bigl(\varphi^{-1}(\zeta)\cap\widetilde W^o_{\mathbb D}\bigr)
\subset\left\{-p\circ g(0)+(\varphi^{-1}(0)\cap W_{\mathbb D})\right\}\cap \left\{-p\circ g(z)+(\varphi^{-1}(z)\cap W_{\mathbb D})\right\}
$$
for all $z\in \mathbb D$.
This implies
\[
\bigl(p\circ g(z)-p\circ g(0),\; q\circ g(z),\; q\circ g(0)\bigr)\in \mathcal{Z}(W)
\]
for all $z\in \mathbb D$, which exactly means $g(\mathbb D)\in\mathcal{Z}_{s}(W)+a$. 
This completes the proof of \Cref{claim:20221018}.
		\end{proof}

	Now we return to the proof of \Cref{lem:202210041}.
	We remark that $\mathcal{Z}_s(W)\subset A\times S$ is a constructible subset for $s\in S$.
		To see this, we consider the composite map 
$$\phi : W \times W_s \hookrightarrow (A \times S) \times A\to A \times S,$$ 
where the second map is defined by 
\[
(A \times S) \times A  \ni ((a, t), a') \mapsto (a - a', t) \in A \times S.
\]
Then $\mathcal{Z}_s(W)$ is given by
\begin{equation*}
\mathcal{Z}_s(W) = \{x \in A \times S  \mid \dim \phi^{-1}(x) \geq \dim W - \dim S\}.
\end{equation*}
Hence, $\mathcal{Z}_s(W)\subset A\times S$ is a constructible subset (cf. \cite[Tag 05F9]{stacks-project}).

		Now we take $y\in Y$ such that $f_S(y)\in S^o$ and $f(y)=(a,s)\in A\times S^o$.
		By the second condition of the definition of $S^o$, the stabilizer of every irreducible component of $q_W^{-1}(s)$ is 0-dimensional.
		Therefore, the restriction of the natural map 
		$$q|_{\mathcal{Z}_s(W)} :\mathcal{Z}_s(W)\to S$$ 
		over $S^o$ is quasi-finite.
		Let $E\subset S$ be the Zariski closure of $f_S(Y)$ and let $\overline{A}$ be a projective compactification of $A$.
		By \cref{claim:20221018}, we have 
		$$f(Y)\subset \overline{q^{-1}(E)\cap(\mathcal{Z}_s(W)+a)}\subset \overline{A}\times S.$$
		Hence there exists an irreducible component $Z$ of $\overline{q^{-1}(E)\cap(\mathcal{Z}_s(W)+a)}$ such that $f(Y)\subset Z$.
		Since $q^{-1}(E)\cap (\mathcal{Z}_s(W)+a)\subset A\times S$ is a constructible subset, it contains a dense Zariski open subset of $Z$.
		Hence the projection $\mu:Z\to S$ is generically-finite onto its image $\mu(Z)=E$.
		Thus $\dim Z=\dim E$, and $f(Y)\subset Z$ is Zariski dense.
		Then by \cite[Lemma 4]{Yam15}, we have $T_{f_A}(r)=O(T_{f_S}(r))+O(\log r)$.
	\end{proof}

	We recall $S'=\overline{\mathrm{Lie}A\times TS}$ from \eqref{eqn:20221203}, so that $\overline{T}(A\times S)=A\times S'$.
	
	\begin{lem}\label{lem:202210042}
		For $f:Y\da A\times S$, we have
		$$T_{(j_1f)_{S'}}(r)=O( T_{f_S}(r)+N_{\ram \pi _Y}(r)+\pN{f}(r))+O(\log r)+o(T(r,f_A)) \ ||,$$
		where $j_1(f)_{S'}:Y\to S'$ is the composition of $j_1(f):Y\da A\times S'$ and the second projection $A\times S'\to S'$.
	\end{lem}
	
	\begin{proof}
		(cf. \cite[Lemma 2]{Yam15})
		Let $\overline{A}$ be an equivariant compactification, which is smooth and projective.
		Set $D=(\partial A)\times S$.
		Then $D$ is a simple normal crossing divisor on $\overline{A}\times S$.
		Let $T(\overline{A}\times S;\log D)$ be the logarithmic tangent bundle.
		We set $\overline{T}(\overline{A}\times S;\log D)=P(T(\overline{A}\times S;\log D)\oplus \mathcal{O}_{\overline{A}\times S})$, which is a smooth compactification of $T(\overline{A}\times S;\log D)$.
		We set $E=\overline{T}(\overline{A}\times S;\log D)-T(\overline{A}\times S;\log D)$, which is a divisor on $\overline{T}(\overline{A}\times S;\log D)$.
		By $T(A\times S)\subset \overline{T}(\overline{A}\times S;\log D)$, we have $j_1f:Y\to \overline{T}(\overline{A}\times S;\log D)$.
		By \cite[(2.4.8)]{Yam04} and \eqref{eqn:202311195}, we have 
		$$
		T_{j_1f}(r,E)\leq N_{\ram \pi}(r)+\pN{f}(r)+O(\log r)+o(T_f(r))\ ||.
		$$ 
		Note that $T(\overline{A};\log\partial A)=\overline{A}\times \mathrm{Lie}A$ (cf. \cite[Prop 5.4.3]{NW13}).
		Hence we have $\overline{A}\times S'=\overline{T}(\overline{A}\times S;\log D)$ and $S'=\mathbf{Proj}((\mathrm{Sym}\ \Omega^1_S)\otimes_{\mathcal{O}_S}\mathcal O_S[\eta,dz_1,\ldots,dz_{\dim A} ])$, where $\{dz_1,\ldots,dz_{\dim A}\}\subset H^0(\overline{A},\Omega^1_{\overline{A}}(\log \partial A))$ is a basis.
		Let $F\subset S'$ be the divisor defined by $\eta=0$.
		Then we have $p^*F=E$, where $p:\overline{A}\times S'\to S'$ is the second projection.
		Hence we have
		$$T_{(j_1f)_{S'}}(r,F)\leq N_{\ram \pi _Y}(r)+\pN{f}(r)+O(\log r)+o(T(r,f_A)) \ ||.$$
		By $\mathrm{Pic}(S')=\mathrm{Pic}(S)\oplus \mathbb Z[F]$, we obtain our lemma.
	\end{proof}

	\section{Refinement of log Bloch-Ochiai Theorem}\label{sec:rbc} 
	
	Let $A$ be a semi-Abelian variety.
	Let $\mathcal{S}_0(A)$ be the set of all semi-abelian subvarieties of $A$.
	Let $W\subset A\times S$ be a Zariski closed subset.
	For $B\in \mathcal{S}_0(A)$, we set
	\begin{equation}\label{eqn:20221101}
		W^B=\{ x\in W;\ x+B\subset W\}.
	\end{equation}
	Then $W^B=\cap_{b\in B}(b+W)\subset W$ is a Zariski closed subset.
	When $B=\{0\}$, we have $W^{\{ 0\}}=W$.
	
	\begin{proposition}\label{pro:20220807}
		Let $S$ be a projective variety.
		Let $W\subset A\times S$ be an irreducible Zariski closed subset.
		Then there exists a finite subset $P\subset \mathcal{S}_0(A)$ such that 
		for every $f:Y\da W$, there exists $B\in P$ such that $f(Y)\sqsubset W^B$ and
		\begin{equation}\label{eqn:20220921}
			T_{q_B\circ f_A}(r)=O(N_{\ram \pi}(r)+T_{f_S}(r)+\pN{f}(r))+O(\log r)+o(T_f(r))||,
		\end{equation}
		where $q_B:A\to A/B$ is the quotient.
	\end{proposition}
	
	\begin{proof}
		Given $f:Y\da W$, we have $e(f,W)\not=-1$, hence $e(f,W)\in\{0,1,\ldots,\dim A\}$.
		\cref{pro:20220807} is reduced to the following claim by setting $k=\dim A$.
		
		\medskip
		
		\begin{claim}
		Let $(k,d)\in\mathbb Z_{\geq 0}^2$.
		Let $S$ be a projective variety such that $\dim S=d$, and let $A$ be a semi-Abelian variety.
		Let $W\subset A\times S$ be an irreducible Zariski closed subset.
			Then there exists a finite subset $P\subset \mathcal{S}_0(A)$ such that for every $f:Y\da W$ with $e(f,W)\leq k$, there exists $B\in P$ such that $f(Y)\sqsubset W^B$ and \eqref{eqn:20220921}. 
		\end{claim}
		\medskip
		
		We prove this claim by induction on the pair of integers $(k,d)\in\mathbb Z_{\geq 0}^2$, with respect to the lexicographical order on $\mathbb Z_{\geq 0}^2$.
		We therefore prove the claim for the pair $(k,d)$ by assuming that it holds for all pairs $(k',d')$ such that $(k',d')<(k,d)$.
		
		Let $q:A\times S\to S$ be the second projection.
		Replacing $S$ by the Zariski closure of $q(W)$, we may assume that $q(W)\subset S$ is dominant. 
		
\textbf{Step 1.}
In the preceding proof, an exceptional locus  $V\subsetneqq S$ naturally arises, which leads us to distinguish between two cases: $f_S(Y)\subset V$ or not.
 As a preliminary step, we treat the first case by fixing a proper Zariski closed subset $V\subsetneqq S$ and assuming $f_S(Y)\subset V$.
Let $V_1,\ldots ,V_l$ be the irreducible components of $V$.
		For $j=1,\ldots ,l$, we set $W_j=W\cap (A\times V_j)$.
		Let $W_j^1,\ldots,W_j^{t_j}$ be the irreducible components of $W_j$.
		By $\dim V_j<d$, the induction hypothesis applies: There exists a finite subset $P_{V_j,W_j^i}\subset \mathcal{S}_0(A)$ such that if $f(Y)\sqsubset W_j^i$ and $e(f,W_j^i)\leq k$, then there exists $B\in P_{V_j,W_j^i}$ such that $f(Y)\sqsubset (W_j^i)^B$ and \eqref{eqn:20220921}.
		Set $P_V=\cup_j\cup_iP_{V_j,W_j^i}$.
		Thus if $f_S(Y)\subset V$ and $e(f,W)\leq k$, then there exists $B\in P_{V}$ such that $f(Y)\sqsubset W^B$ and \eqref{eqn:20220921}.

		\textbf{Step 2.}
		We consider the case that $S$ is smooth and $\mathrm{St}^0(W)=\{0\}$.
		Let $f:Y\da W$ satisfy $e(f,W)\leq k$.
		Suppose $f_S(Y)\subset S\backslash S^o$, then by the above consideration (Step 1), there exists $B\in P_{S\backslash S^o}$ such that $f(Y)\sqsubset W^B$ and \eqref{eqn:20220921}.
		So we consider the case $f_S(Y)\not\subset S\backslash S^o$.
		We consider the first jet $j_1f:Y\da \overline{T}(A\times S)=A\times \overline{(\mathrm{Lie}A\times TS)}$.
		Set $S'=\overline{\mathrm{Lie}A\times TS}$.
		If $e(j_1f,W^{(1)})=e(f,W)$, then \cref{lem:202210041} yields 
		$$T_{f_A}(r)=O(T_{f_{S}}(r))+O(\log r).$$
		This shows \eqref{eqn:20220921} for $B=\{0\}$, and obviously $f(Y)\sqsubset W= W^{\{0\}}$.
		If $e(j_1f,W^{(1)})<e(f,W)\leq k$, then $e(j_1f,W^{(1)})\leq k-1$.
		Thus the induction hypothesis yields a finite set $P'\subset \mathcal{S}_0(A)$ such that there exists $B\in P'$ such that $j_1f(Y)\sqsubset (W^{(1)})^B$ and 
		$$T_{q_B\circ f_A}(r)=O(N_{\ram\pi}(r)+\pN{f}(r)+T_{(j_1f)_{S'}}(r))+O(\log r)+o(T_f(r))||.$$
		By \cref{lem:202210042}, we have 
		$$T_{(j_1f)_{S'}}(r)=O(N_{\ram\pi}(r)+\pN{f}(r)+T_{f_S}(r))+O(\log r)+o(T_{f_A}(r))||.$$
		Hence we have $f(Y)\sqsubset W^B$ and \eqref{eqn:20220921}.
		We set $P=P'\cup P_{S\backslash S^o}\cup\{0\}$.
		This concludes the proof of the claim if $S$ is smooth and $\mathrm{St}^0(W)=\{0\}$.
		
		\textbf{Step 3.}
		We remove the assumption that $S$ is smooth.
		Let $\widetilde{S}\to S$ be a smooth modification which is an isomorphism outside a proper Zariski closed set $E\subsetneqq S$.
		If $f_S(Y)\subset E$, then we apply Step 1: there exists $B\in P_{E}$ such that $f(Y)\sqsubset W^B$ and \eqref{eqn:20220921}.
		If $f_S(Y)\not\subset E$, then there exists a unique lift $f:Y\da A\times \widetilde{S}$.
		Let $\widetilde{W}\subset A\times \widetilde{S}$ be the proper transform of $W$.
		Then by the consideration above (Step 2), there exists $P'\subset \mathcal{S}_0(A)$ such that $f(Y)\sqsubset \widetilde{W}^B$ and \eqref{eqn:20220921} for some $B\in P'$.
		We set $P=P'\cup P_E$ to conclude the proof of the claim when $\mathrm{St}^0(W)=\{0\}$.

		\textbf{Step 4.}
		Finally we remove the assumption $\mathrm{St}^0(W)=\{0\}$.
		Suppose that $\mathrm{St}^0(W)\not=\{0\}$.
		Set $C=\mathrm{St}^0(W)$.
		Let $f':Y\da W/C\subset (A/C)\times S$ be the induced map.
		Then we have $e(f',W/C)<e(f,W)$.
		Hence by the induction hypothesis, there exists $P'\subset \mathcal{S}_0(A/C)$ such that $f'(Y)\sqsubset (W/C)^{B'}$ and \eqref{eqn:20220921} for some $B'\in P'$.	 	
		We define $P$ to be the set of all $B$ such that $C\subset B$ and $B/C\in P'$.
		This concludes the proof of the claim.
		Thus the proof of \cref{pro:20220807} is completed.
	\end{proof}

	\begin{cor}\label{cor:20220806}
		In Proposition \ref{pro:20220807}, we may take $P$ so that $\mathrm{St}^0(W)\subset B$ for all $B\in P$.
		Moreover there exists a proper Zariski closed set $\Xi\subsetneqq W$ such that for every $f:Y\da W$, either one of the followings holds:
		\begin{enumerate}
			\item
			$f(Y)\sqsubset \Xi$.
			\item
			$T_{q_{\mathrm{St}^0(W)}\circ f_A}(r)=O(N_{\ram\pi}(r)+\pN{f}(r)+T_{f_S}(r))+O(\log r)+o(T_f(r))||$,
			where $q_{\mathrm{St}^0(W)}:A\to A/\mathrm{St}^0(W)$ is the quotient map
		\end{enumerate}
	\end{cor}
	
	\begin{proof}
		We apply Proposition \ref{pro:20220807} for $W/\mathrm{St}^0(W)\subset (A/\mathrm{St}^0(W))\times S$ to get $P_0\subset \mathcal{S}_0(A/\mathrm{St}^0(W))$.
		Then we define $P\subset \mathcal{S}_0(A)$ as the set of all $B$ such that $\mathrm{St}^0(W)\subset B$ and $B/\mathrm{St}^0(W)\in P_0$.
		We set $P'=P\backslash \{\mathrm{St}^0(W)\}$ and $\Xi=\cup_{B\in P'}W^B$.
		Then $\Xi$ is a proper Zariski closed set.
	\end{proof}

	\section{Second main theorem type estimate with weak truncation}
	\label{sec:20221125}
	Let $A$ be a semi-Abelian variety.
	Let $\mathcal{S}(A)$ be the set of all positive dimensional semi-abelian subvarieties of $A$.
	Hence $\mathcal{S}(A)=\mathcal{S}_0(A)\backslash \{\{0\}\}$.

	\begin{proposition}\label{pro:202208062}
		Let $\overline{A}$ be an equivariant compactification of $A$ such that $\overline{A}$ is smooth and projective.
		Let $W\subsetneqq \overline{A}\times S$ be a closed subscheme, where $S$ is a projective variety.
		Then there exist a finite subset $P\subset \mathcal{S}(A)\backslash\{A\}$ and a positive integer $\rho\in\mathbb Z_{>0}$ with the following property:
		Let $f:Y\da A\times S$ satisfies $f(Y)\not\sqsubset \mathrm{supp}\, W$ and $f_S(Y)\not\subset p(\mathrm{Sp}_AW)$, where $\mathrm{Sp}_AW=\cap_{a\in A}(a+W)\subset \overline{A}\times S$ and $p:\overline{A}\times S\to S$ is the second projection.
		Then either one of the followings are true:
		\begin{enumerate}
			\item
			There exists $B\in P$ such that
			$$T_{q\circ f_A}(r)=O(T_{f_S}(r)+N_{\ram\pi}(r)+\pN{f}(r))+O(\log r)+o(T_{f}(r))||,$$
			where $q:A\to A/B$ is the quotient map.
			\item
			The following estimate holds:
			\begin{multline}
				m_f(r,W)+N_f(r,W)-N^{(\rho)}_f(r,W)\\
				=O(N_{\ram\pi}(r)
				+\pN{f}(r)+T_{f_S}(r))+O(\log r)+o(T_f(r))||.
			\end{multline}
		\end{enumerate}
	\end{proposition}
	
	\begin{proof}
		For $W\subset \overline{A}\times S$, we set $W_o=W\cap (A\times S)$.
		Given $f:Y\da A\times S$, we have $e(f,W_o)\in\{-1,0,\ldots,\dim A\}$.
		\cref{pro:202208062} is reduced to the following claim by setting $k=\dim A$.

		\begin{claim}
		Let $(m,k,d)\in\mathbb Z_{\geq 0}\times \mathbb Z_{\geq -1}\times\mathbb Z_{\geq 0}$.
		Let $A$ be a semi-Abelian variety such that $\dim A=m$ and let $S$ be a projective variety such that $\dim S=d$.
		Let $\overline{A}$ and $W\subsetneqq \overline{A}\times S$ be the same as in \cref{pro:202208062}.
		Then there exist a finite subset $P\subset \mathcal{S}(A)\backslash\{A\}$ and a positive integer $\rho\in\mathbb Z_{>0}$ with the following property:
		Let $f:Y\da A\times S$ satisfies $f(Y)\not\sqsubset \mathrm{supp}\, W$, $f_S(Y)\not\subset p(\mathrm{Sp}_AW)$, and $e(f,W_o)\leq k$.
Then either one of the assertions of \cref{pro:202208062} is true.
		\end{claim}

		We prove this claim by the induction on the triple $(m,k,d)\in\mathbb Z_{\geq 0}\times \mathbb Z_{\geq -1}\times\mathbb Z_{\geq 0}$ with the dictionary order.  
		We therefore prove the claim for the triple $(m,k,d)$ by assuming that it holds for all triples $(m',k',d')$ such that $(m',k',d')<(m,k,d)$.

	\textbf{Step 1.}	
		As in the step 1 of the proof of \cref{pro:20220807}, we fix a proper Zariski closed subset $V\subsetneqq S$ and assume $f_S(Y)\subset V$.
Let $V_1,\ldots ,V_l$ be the irreducible components of $V$.
		For $j=1,\ldots ,l$, we set $W_j=W_o\cap (A\times V_j)$.
		By $\dim V_j<d$, the induction hypothesis applies: 
		There exist a finite subset $P_{j}\subset \mathcal{S}(A)\backslash \{A\}$ and $\rho_{j}\in\mathbb Z_{\geq 1}$ such that if $f:Y\da A\times S$ satisfies $f_S(Y)\subset V_j$ and $e(f,W_j)\leq k$ in addition to $f(Y)\not\sqsubset W_j$ and $f_S(Y)\not\subset p(\mathrm{Sp}_AW_j)$, either of the two assertions of \cref{pro:202208062} holds, with $P$ and $\rho$ replaced by $P_{j}$ and $\rho_{j}$, respectively.
		Set $P_{V}=\cup_jP_{j}$ and $\rho_{V}=\max_{j}\rho_{j}$.
		Then if $f_S(Y)\subset V$ and $e(f,W_o)\leq k$ in addition to $f(Y)\not\sqsubset W$ and $f_S(Y)\not\subset p(\mathrm{Sp}_AW)$, either of the two assertions of \cref{pro:202208062} holds, with $P$ and $\rho$ replaced by $P_{V}$ and $\rho_{V}$, respectively.

		\textbf{Step 2.}	
		We assume that the following three conditions are satisfied.
		\begin{enumerate}
			\item
			$q(W)=S$, where $q:\overline{A}\times S\to S$ is the second projection,
			\item
			$W$ is reduced, and irreducible,
			\item
			$S$ is smooth.
		\end{enumerate}

		This step is separated into several cases.
		
		\textbf{Step 2-1.}	
		 $W_o\not=\emptyset$.
		Set $C=\mathrm{St}^0(W_o)$.
		We define $S^o\subset S$ from $W_o/C\subset (A/C)\times S$ (cf. \cref{lem:202210041}).
		We consider $f:Y\da A\times S$ with $e(f,W_o)=k$ such that $f(Y)\not\sqsubset W$ and $f_S(Y)\not\subset p(\mathrm{Sp}_AW)$.
		If $f_S(Y)\subset S\backslash S_o$, then by the above consideration (Step 1), either of the two assertions of \cref{pro:202208062} holds, with $P$ and $\rho$ replaced by $P_{S\backslash S_o}$ and $\rho_{S\backslash S_o}$, respectively.
		So we assume that $f_S(Y)\not\subset S\backslash S_o$.
		We consider the first jet $j_1f:Y\da A\times S'$, where $S'=\overline{\mathrm{Lie}(A)\times TS}$.
		
		\textbf{Step 2-1-1.}	
		 $e(j_1f,W_o^{(1)})=e(f,W_o)$, where $W_o^{(1)}\subset A\times S'$.	
Note that $C=\mathrm{St}^0(W_o^{(1)})$.
		Let $f':Y\to (A/C)\times S$ be the composition of $f$ and the projection $A\times S\to (A/C)\times S$.
		Then we have $e(f,W_o)=e(f',W_o/C)+\dim C$ and $e(j_1f,W_o^{(1)})=e(j_1f',(W_o/C)^{(1)})+\dim C$. 		
		Hence $e(f',W_o/C)=e(j_1f',(W_o/C)^{(1)})$.
		Thus \cref{lem:202210041} yields that 
		\begin{equation}\label{eqn:202212015}
			T_{q_C\circ f_A}(r)=O(T_{f_S}(r))+O(\log r),
		\end{equation}
		where $q_C:A\to A/C$ is the quotient.
		Since we are assuming that $W\to S$ is dominant, we have $C\not=A$, for otherwise $f(Y)\sqsubset W$.
		Thus if $\dim C>0$,  the estimate \eqref{eqn:202212015} shows the first assertion of \cref{pro:202208062}, provided $C\in P$.
		If $\dim C=0$, then \eqref{eqn:202212015} yields that $T_{f_A}(r)= O(T_{f_S}(r))+O(\log r)$.  
		Hence $m_f(r,W)+N_f(r,W)=O(T_{f_S}(r))+O(\log r)$.
		This is stronger than the second assertion of \cref{pro:202208062}.
		Thus assuming $C\in P$ if $C\not= \{0\}$, we conclude that one of the assertions of \cref{pro:202208062} holds.
		
		\textbf{Step 2-1-2.} $e(j_1f,W_o^{(1)})<e(f,W_o)$.
		Then the induction hypothesis yields $P'\subset \mathcal{S}(A)\backslash\{A\}$ and $\rho'$ such that either the first assertion of  \cref{pro:202208062} for $P=P'$ or the estimate
		\begin{multline*}
			m_{j_1f}(r,W_{\log}^{(1)})+N_{j_1f}(r,W_{\log}^{(1)})-N^{(\rho')}_{j_1f}(r,W_{\log}^{(1)})
			\\
			=O(N_{\ram\pi}(r)+\pN{f}(r)
			+T_{j_1f_{S'}}(r))+O(\log r)+o(T_f(r))||
		\end{multline*}
		holds.
		Here $W_{\log}^{(1)}\subset \overline{A}\times S'=\overline{T}(\overline{A}\times S;\log \partial (A\times S))$ is defined in \cite[Section 5]{Yam04} so that $W_{\log}^{(1)}\cap (A\times S')=W_o^{(1)}$.
		By the same argument for the proof of \cite[Lemma 5]{Yam15}, using \cite[Thm 5.1.7]{Yam04}, we have
		\begin{multline}\label{eqn:20250904}
			m_f(r,W)+N_f(r,W)-N^{(\rho'+1)}_f(r,W)
			\\
			\leq m_{j_1f}(r,W_{\log}^{(1)})+N_{j_1f}(r,W_{\log}^{(1)})-N^{(\rho')}_{j_1f}(r,W_{\log}^{(1)})+N_{\ram\pi}(r)+o(T_{f}(r))||.
		\end{multline}
		By \cref{lem:202210042}, we have 
		$$T_{(j_1f)_{S'}}(r)=O(T_{f_S}(r)+N_{\ram\pi}(r)+\pN{f}(r))+O(\log r)+o(T_f(r))||.$$
		Thus we get
		\begin{multline*}
			m_f(r,W)+N_f(r,W)-N^{(\rho'+1)}_f(r,W)=O(T_{f_S}(r)+N_{\ram\pi}(r)+\pN{f}(r))+O(\log r)+o(T_f(r))||. 	
		\end{multline*}		
		Therefore, combining this with Case 2-1-1, we complete the proof of the induction step for Case 2-1.
Here we set $P=(P_{S\backslash S_o}\cup P' \cup \{C\})\backslash \{0\}$ and $\rho=\rho_{S\backslash S_o}+\rho'+1$.

\textbf{Step 2-2.}		
		 $W_o=\emptyset$.
		In this case, we have $W\subset \partial A\times S$ and $e(f,W_o)=-1$.
		Let $I\subset A$ be the isotropy group for $W$ and let $D\subset \partial A$ be the irreducible component of $\partial A$ such that $W\subset D\times S$.
		Then $D$ is an equivariant compactification of $A/I$.
		By \cite[Lem A.11]{Y22}, there exist an $A$-invariant Zariski open set $U\subset \overline{A}$ and an equivariant map $\psi:U\to D$ such that $D\subset U$ and $\psi$ is an isomorphism over $D\subset U$.
		We define $g:Y\da (A/I)\times S\subset D\times S$ by $g=(\psi\circ f_A,f_S)$.
		By $\mathrm{Sp}_A(W)=\mathrm{Sp}_{A/I}(W)$, we have $g_S(Y)\not\subset p(\mathrm{Sp}_{A/I}W)$.
		Suppose $g(Y)\not\sqsubset W$.
		Note that $\dim (A/I)<\dim A$.
		Hence by the induction hypothesis, there exist $P'\subset \mathcal{S}(A/I)\backslash \{A/I\}$ and $\rho'\in\mathbb Z_{>0}$, which are independent of the choice of $f$, such that either the first assertion of \cref{pro:202208062} or the following estimate 
		$$m_g(r,W)+N_g(r,W)-N^{(\rho')}_g(r,W)=O(N_{\ram\pi}(r)+\pN{g}(r)
		+T_{f_S}(r))+O(\log r)+o(T_f(r))||
		$$
		holds.
		By $\pN{g}(r)\leq \pN{f}(r)$, $m_f(r,W)\leq m_g(r,W)$ and $\mathrm{ord}_yf^{*}W\leq \mathrm{ord}_yg^{*}W$ for all $y\in Y$, this estimate implies the second assertion of \cref{pro:202208062} for $\rho=\rho'$.
		If $g(Y)\sqsubset W$, then we apply \cref{pro:20220807} to get $P''\subset \mathcal{S}_0(A/I)$, which is independent of the choice of $f$.
		Then there exists $B\in P''$ such that 
		$$T_{q\circ g_{A/I}}(r)=O(T_{f_S}(r)+N_{\ram\pi}(r)+\pN{f}(r))+O(\log r)+o(T_{f}(r))||,
		$$
		where $q:A/I\to (A/I)/B$ is the quotient.
		We note $B\not=A/I$, for otherwise we have $g(Y)\sqsubset W^{A/I}$, which contradicts to $g_S(Y)\not\subset p(\mathrm{Sp}_{A/I}W)$.
		We define $P$ to be the set of all $B$  such that $I\subset B$ and $B/I\in P'\cup (P''\backslash \{A/I\})$.
		This concludes the proof of the induction step for our case $W_o=\emptyset$.
		Thus we have completed the induction step under the three assumptions of Step 2.
		
		\textbf{Step 3.}
		We remove the three assumptions on $W$ and $S$ made in Step 2.
		First we remove the first assumption.
		Suppose $q(W)\subsetneqq S$, where $q:\overline{A}\times S\to S$ is the second projection.
		If $f_{S}(Y)\not\subset q(W)$, then we have 
		$$m_f(r,W)+N_f(r,W)=O(T_{f_S}(r))+O(\log r).$$
		This is stronger than the second assertion of \cref{pro:202208062}.
		Thus we may consider the case $f_{S}(Y)\subset q(W)$.
		We replace $S$ by $q(W)$.
		Then, by the induction hypothesis, we get our claim.	
		Hence the first assumption is removed.	
		
		Next we remove the second assumption.
		Let $W_1,\ldots W_n$ be the irreducible components of $\mathrm{supp} W$.
		Then for each $j\in \{1,\ldots,n\}$, by the argument above (Step 2), there exist a finite subset $P_{j}\subset \mathcal{S}(A)\backslash \{A\}$ and a positive integer $\rho_{j}$ such that for every $f:Y\da A\times S$ with $e(f,W_o)\leq k$, either the first assertion of \cref{pro:202208062} for $P=P_{j}$ or the following estimate holds:
		\begin{multline}\label{eqn:9.3.1}
			m_f(r,W_j)+N_f(r,W_j)-N^{(\rho _{j})}_f(r,W_j)
			\\
			=O(T_{f_S}(r)+N_{\ram \pi _Y}(r)+\pN{f}(r))+O(\log r)+o( T_f(r)\ ||.
		\end{multline}
		There is a positive integer $l$ such that
		$$\big( \mathcal I_{W_1}\cdots \mathcal I_{W_n}\big) ^l\subset \mathcal I_W,
		$$
		where $\mathcal I_W\subset \mathcal O_{\overline{A}\times S}$ (resp. $\mathcal I_{W_i}$) is the defining ideal sheaf of $W$ (resp. $W_i$).
		Then we have
		\begin{equation}\label{eqn:9.3.2}
			m_f(r,W)\leq l\sum_{i=1}^nm_f(r,W_i).
		\end{equation}
		For $y\in Y$, we have
		$$\ord _yf^*W\leq l\sum_{i=1}^n\ord _yf^*W_i.$$
		Hence setting $\widetilde{\rho}=\rho _{1}+\cdots +\rho _{n}$, we get
		\begin{equation*}
			\begin{split}
				\max \{ 0,\ord _yf^*W-l\widetilde{\rho} \}&\leq \max \left\{ 0,l\sum_{i=1}^n\left( \ord _yf^*W_i-\rho _{i}\right) \right\}\\
				&\leq l\sum_{i=1}^n\max \left\{ 0,\left( \ord _yf^*W_i-\rho _{i}\right) \right\} .
			\end{split}
		\end{equation*}
		Hence we get
		\begin{equation}\label{eqn:9.3.3}
			N_f(r,W)-N^{(l\widetilde{\rho})}_f(r,W)\leq  l\sum_{i=1}^n\left( N_f(r,W_i)-N^{(\rho _{i})}_f(r,W_i)\right) .
		\end{equation}
		By \eqref{eqn:9.3.1}, \eqref{eqn:9.3.2}, \eqref{eqn:9.3.3}, we get
		\begin{equation*}
			m_f(r,W)+N_f(r,W)-N^{(l\widetilde{\rho})}_f(r,W) \\
			\leq O(T_{f_S}(r)+N_{\ram\pi _Y}(r)+\pN{f}(r))+O(\log r)+o(T_f(r))\ ||.
		\end{equation*}
		Thus we have removed the assumption that $W$ is irreducible and reduced.
		Here we set $\rho=l\widetilde{\rho}$ and $P=P_{1}\cup\cdots\cup P_{n}$.
		
		The assumption that $S$ is smooth is removed similarly as in the proof of \cref{pro:20220807}.
		This completes the induction step of the proof of the claim.
		Thus the claim is proved.
		The proof of \cref{pro:202208062} is completed.
	\end{proof}

	\section{Intersection with higher codimensional subvarieties}

	\begin{lem}\label{lem:202209151}
		Let $L$ be an ample line bundle on $\overline{A}$, where $\overline{A}$ is a smooth equivariant compactification.
		Let $Z\subset A\times S$ be a closed subscheme whose codimension is greater than one, where $S$ is a projective variety.
		Let $\varepsilon>0$.
		Then there exist a finite subset $P\subset \mathcal{S}(A)\backslash\{A\}$ and a proper Zariski closed subset $E\subsetneqq S$ with the following property:
		Let $f:Y\da A\times S$ satisfies $f(Y)\not\sqsubset \mathrm{supp}\, Z$.
		Then either one of the followings is true:
		\begin{enumerate}
			\item
			$
			N^{(1)}_f(r,Z)\leq \varepsilon T_{f_A}(r,L)+O_{\varepsilon}(N_{\ram\pi}(r)+\pN{f}(r)
			+T_{f_S}(r))+O(\log r)+o(T_f(r))||_{\varepsilon}$.
			\item
			$f_S(Y)\subset E$.
			\item
			There exists $B\in P$ such that 
			$$
			T_{q_B\circ f_A}(r)=O(N_{\ram\pi}(r)+\pN{f}(r)+T_{f_S}(r))+O(\log r)+o(T_f(r))||, 
			$$
			where $q_B:A\to A/B$ is the quotient map.
		\end{enumerate}
	\end{lem}

	\begin{proof}
		We first prove this lemma under the additional three assumptions:
		\begin{itemize}
			\item
			$q(Z)\subset S$ is dense, where $q:A\times S\to S$ is the second projection,
			\item
			$Z$ is irreducible,
			
			\item
			$S$ is smooth.
		\end{itemize}
		
		We use higher jet spaces.
		Let $J_l(A\times S)$ be the $l$-th jet space.
		We have the natural splitting $J_l(A\times S)=A\times (J_l(A\times S)/A)$ induced from the splitting $TA=A\times\mathrm{Lie}A$.
		There exists a partial compactification $J_l(A\times S)\subset \bar{J}_l(A\times S)$ such that the natural map $\bar{J}_l(A\times S)\to A\times S$ is projective and $\bar{J}_l(A\times S)$ is $\mathbb Q$-factorial (\cite[2.4]{Yam04}).	
		We have the natural splitting $\bar{J}_l(A\times S)=A\times (\bar{J}_l(A\times S)/A)$.
		When $l=1$, this reduces to \eqref{eqn:20221203}.

		\begin{claim}\label{claim:20221021}
			There exist a sequence of positive integers $n(1),n(2),n(3),\cdots$ and an ample line bundle $L_o$ on $\overline{A}$ with the following conditions:
			\renewcommand{\labelenumi}{(\theenumi )}
			\begin{enumerate}
				\item $\frac{n(l)}{l}\to 0 $ when $l\to \infty$.
				\item For $l\geq 1$, let $S_l\subset J_l(A\times S)/A$ be an irreducible Zariski closed subset whose
				image under the projection $J_l(A\times S)/A\to S$ is Zariski dense.
				Let $\overline{S_l}\subset \bar{J}_l(A\times S)/A$ be the compactification of $S_l$ in $\bar{J}_l(A\times S)/A$.
				Then there exists an effective Cartier divisor $F_l\subset \overline{A}\times \overline{S_l}$ with the following properties:
				\begin{enumerate}
					\item $F_l$ corresponds to a global section of
					$$H^0(\overline{A}\times \overline{S_l},p_l^*L_o^{\otimes n(l)}\otimes q _l^*M_l),$$
					where $p_l:\overline{A}\times \overline{S_l}\to \overline{A}$ is the first projection, $q _l:\overline{A}\times \overline{S_l}\to \overline{S_l}$ is the second projection, and $M_l$ is an ample line bundle on $\overline{S_l}$.
					
					\item Let $f:\mathbb D \to A\times S$ be a holomorphic map such that $j_l(f)(\mathbb D )\subset A\times S_l$ and $j_l(f)(\mathbb D )\not\subset F_l$.
					Assume $f(0)\in Z$, then $\ord _{0}j_{l}(f)^*F_l\geq l+1$.
				\end{enumerate}
			\end{enumerate}
		\end{claim}
		
		\begin{proof}
			When $A$ is compact, this is \cite[Prop 6]{Yam15}.	
			The same proof as in \cite[Prop 6]{Yam15} can be adapted to our setting, with the only difference being that it requires the compactification of several objects.
			For the sake of completeness we give the proof of the claim in \cref{rem:2025:09041}.
			See also \cite[Lemma 6.5.42]{NW13}.
		\end{proof}

		Now we are given $\varepsilon>0$.
		Let $\mu>0$ be a positive integer such that $L^{\otimes\mu}\otimes L_o^{-1}$ is ample.
		We take $l\in\mathbb N$ such that 
		\begin{equation}\label{eqn:201308171}
			n(l)/(l+1)<\varepsilon/\mu .
		\end{equation}
		Let $\mathcal{V}$ be the set of all irreducible Zariski closed subsets of $J_l(A\times S)/A$.
		Let $\mathcal{W}\subset \mathcal{V}$ be the subset of $\mathcal{V}$ which consists of $W\in \mathcal{V}$ whose image under the projection $J_l(A\times S)/A\to S$ is Zariski dense in $S$.
		We define the sequence
		$\mathcal{V}_1, \mathcal{V}_2, \ldots$
		of subsets of $\mathcal{V}$ and the sequence
		$\mathcal{W}_1, \mathcal{W}_2, \ldots$
		of subsets of $\mathcal{W}$
		by the following inductive rule. 
		Put $\mathcal{V}_1 =\mathcal{W}_1= \{ J_l(A\times S)/A\}$. 
		For each $W\in \mathcal{W}_i$, let $F_{W}\subset A\times W$ be the divisor defined as the restriction of the divisor in \cref{claim:20221021} to $A\times W$.
		Set $(F_W)^A=A\times W'$, where $W'\subset W$ is a proper Zariski closed subset.
		Put
		$$
		\mathcal{V}_{i+1}=\bigcup_{W\in \mathcal{W}_i} \{ V\in \mathcal{V};\text{$V$ is an irreducible component of $W'$}\} ,
		$$
		$$
		\mathcal{W}_{i+1}=\mathcal{V}_{i+1}\cap \mathcal{W}.
		$$
		Since the number of the irreducible components of $W'$ is finite,
		each $\mathcal{V}_i$ is a finite set. 
		Since $\dim W'  < \dim W$, we have $\mathcal{V}_i=\mathcal{W}_i = \emptyset$ for
		$i \geq \dim (J_l(A\times S)/A)+2$.
		We apply \cref{pro:20220807} for $F_W\subset A\times W$ to get $P_{W}\subset \mathcal{S}_0(A)$.
		(More precisely, we apply \cref{pro:20220807} for each irreducible component of $F_W$.) 
		Set
		$$
		P=\bigcup_{i}\bigcup_{W_i\in \mathcal{W}_i}P_{W_i}\backslash \{\{0\},A\}.
		$$
		Then $P\subset \mathcal{S}(A)-\{ A\}$ is a finite subset.
		For $V\in \mathcal{V}_i\backslash \mathcal{W}_i$, let $S_{V}\subset S$ be the Zariski closure of the image of $V$ under the projection $J_l(A\times S)/A \to S$.
		Then by the construction of $\mathcal{W}_i$, we have $S_{V}\not= S$.
		Set
		$$
		E=\bigcup_{i}\bigcup_{V\in \mathcal{V}_i\backslash \mathcal{W}_i}S_{V}.
		$$
		Then $E\subset S$ is a proper Zariski closed subset.
		
		\par

		Now let $f:Y\da A\times S$ with $f(Y)\not\sqsubset Z\cup q^{-1}(E)$. 
		Let $q':J_l(A\times S)\to J_l(A\times S)/A$ be the projection under the splitting $J_l(A\times S)=A\times (J_l(A\times S)/A)$.
		We may take $W_i\in \mathcal{W}_i$ such that $q'\circ j_l(f)(Y)\sqsubset W_i$ but $q'\circ j_l(f)(Y)\not\sqsubset W_{i+1}$ for all $W_{i+1}\in \mathcal{W}_{i+1}$.
		Then we have
		\begin{equation}\label{eqn:130504}
			j_l(f)(Y)\not\sqsubset (F_{W_i})^A.
		\end{equation}
		We consider the three possible cases.
		
		\par
		
		\textbf{Case 1.} $j_l(f)(Y)\not\sqsubset F_{W_i}$.
		We remark that $j_l(f)$ hits the boundary of $\bar{J}_l(A\times S)$ only at the ramification points of $\pi _Y:Y\to \mathbb C$ and the points in $\bar{f}^{-1}(\partial A\times S)$, where $\bar{f}:Y\to \overline{A}\times S$ is the extension of $f$.
		Hence, we have
		$$(l+1)N^{(1)}_f(r,Z)\leq N_{j_l(f)}(r,F_{W_i})+(l+1)N_{\ram \pi _Y}(r)+(l+1)\pN{f}(r).$$
		We recall that $F_{W_i}$ corresponds to $H^0(\overline{A}\times W_i,p_l^*L_o^{\otimes n(l)}\otimes q _l^*M_l)$, where $M_l$ is an ample line bundle on $W_i$ and $q_l:A\times W_i\to W_i$ is the second projection.
		Hence by the Nevanlinna inequality (cf. \eqref{eqn:20221120}), we have
		$$
		N_{j_l(f)}(r,F_{W_i})\leq n(l)T_f(r, p_l^*L_o)+T_{j_l(f)}(r,q_l^*M_l)+O(\log r).
		$$
		By the similar argument for the proof of \cref{lem:202210042}, we get 	
		\begin{equation}\label{eqn:12092710}
			T_{ j_l(f)}(r,q_l^*M_l)=O (T_{f_S}(r)+N_{\ram \pi _Y}(r)+\pN{f}(r))+O(\log r)+o(T_f(r))\ ||.
		\end{equation}
		Hence we obtain
		\begin{equation*}\label{eqn:1209271}
			N^{(1)}_f(r,Z)\leq \varepsilon T_{f_A}(r,L)+O (T_{f_S}(r)+N_{\ram \pi _Y}(r)+\pN{f}(r))+O(\log r)+o(T_f(r))\ ||.
		\end{equation*}

		\par

		\textbf{Case 2. }$j_l(f)(Y)\sqsubset F_{W_i}$.
		We apply \cref{pro:20220807} to get $B\in P_{W_i}$.
		By $j_lf(Y)\not\sqsubset (F_{W_i})^A$, we have $B\not=A$.
		
		\par 
		
		\textbf{Case 2-1.}
		$B=\{0\}$.
		Then by \cref{pro:20220807}, we have
		$$
		T_{f_A}(r)=O(N_{\ram\pi}(r)+\pN{f}(r)+T_{j_lf_{W_i}}(r))+O(\log r)+o(T_f(r))||.
		$$
		Thus by \eqref{eqn:12092710}, we have
		$$
		T_f(r)=O (T_{f_S}(r)+N_{\ram \pi _Y}(r)+\pN{f}(r))+O(\log r)+o(T_f(r))\ ||.
		$$
		Hence
		$$
		N_f(r,Z)=O (T_{f_S}(r)+N_{\ram \pi _Y}(r)+\pN{f}(r))+O(\log r)+o(T_f(r))\ ||.
		$$
		
		\par

		\textbf{Case 2-2.} $B\not=\{0\}$.
		Then $B\in P$.
		By \cref{pro:20220807}, we have
		$$
		T_{q_B\circ f_A}(r)
		=O(N_{\ram\pi}(r)+\pN{f}(r)+T_{j_lf_{W_i}}(r))+O(\log r)+o(T_f(r))||.
		$$
		Thus by \eqref{eqn:12092710}, we have
		$$
		T_{q_B\circ f_A}(r)=O(N_{\ram\pi}(r)+\pN{f}(r)+T_{f_S}(r))+O(\log r)+o(T_f(r))||.
		$$
		
		\par
		
		We combine the three cases above to conclude the proof of the lemma (\cref{lem:202209151}) under the three assumptions above.
		
		We remove the three assumptions.
		First suppose $q(Z)\subsetneqq S$.
		We set $E=\overline{q(Z)}$.
		Then $E\subsetneqq S$ is a Zariski closed set.
		If $f_S(Y)\not\subset E$, then we have $m_f(r,Z)+N_f(r,Z)=O(T_{f_S}(r))+O(\log r)$.
		This is stronger than the first assertion in \cref{lem:202209151}.
		Thus we have removed the assumption that $q(Z)\subset S$ is dense.	
		
		Next we remove the assumption that $Z$ is irreducible.
		Let $Z=Z_1\cup \cdots\cup Z_k$ be the irreducible decomposition.
		We apply \cref{lem:202209151} for $Z_i$ and $\varepsilon/k$ to get $P_i\subset \mathcal{S}(A)\backslash \{A\}$ and $E_i\subsetneqq S$.
		We set $P=\cup_{i}P_i$ and $E=\cup_iE_i$.
		Then $P\subset \mathcal{S}(A)\backslash\{A\}$ is a finite subset and $E\subsetneqq S$ is a proper Zariski closed subset.
		If $f:Y\da A\times S$ does not satisfy the second and the third assertions in \cref{lem:202209151}, then we have 
		$$
		N^{(1)}_f(r,Z_i)\leq (\varepsilon/k) T_{f_A}(r,L)+O_{\varepsilon}(N_{\ram\pi}(r)+\pN{f}(r)
		+T_{f_S}(r))+O(\log r)+o(T_f(r))||_{\varepsilon}.
		$$
		By $N^{(1)}_f(r,Z)\leq \sum_{i=1}^kN^{(1)}_f(r,Z_i)$, we have removed the assumption that $Z$ is irreducible.	
		
		Finally we remove the assumption that $S$ is smooth.
		Let $E_o\subset S$ be the singular locus of $S$ and let $\tau:S'\to S$ be a smooth modification, which is isomorphic outside $E_o$.		Let $Z'\subset A\times S'$ be the Zariski closure of $Z\cap (A\times (S\backslash E_o))$.
		Then the codimension of $Z'\subset A\times S'$ is greater than one.
		We apply  \cref{lem:202209151} for $Z'$ to get $P\subset \mathcal{S}(A)\backslash \{A\}$ and $E'\subsetneqq S'$.
		We set $E=E_o\cup\tau(E')$.
		If $f:Y\da A\times S$ does not satisfy the second and the third assertions in \cref{lem:202209151}, then we have a lift $f':Y\to A\times S'$ and 
		$$
		N^{(1)}_{f'}(r,Z')\leq \varepsilon T_{f_A}(r,L)+O_{\varepsilon}(N_{\ram\pi}(r)+\pN{f}(r)
		+T_{f_S}(r))+O(\log r)+o(T_f(r))||_{\varepsilon}.
		$$	
		By 
		$$N^{(1)}_f(r,Z)\leq N^{(1)}_{f'}(r,Z')+N^{(1)}_{f_S}(r,E_o)\leq N^{(1)}_{f'}(r,Z')+O(T_{f_S}(r))+O(\log r),$$
		we have removed the assumption that $S$ is smooth.
		The proof of \cref{lem:202209151} is completed.		\end{proof}
	
	\medskip

	Let $Z\subset A\times S$ be a closed subscheme.
	For $f:Y\da A\times S$, we recall $e(f,Z)$ from \cref{subsec:4.2}.

	\begin{proposition}\label{pro:202208061}
		Let $Z\subset A\times S$ be a closed subscheme, where $S$ is a projective variety.
		Let $L$ be an ample line bundle on $\overline{A}$, where $\overline{A}$ is a smooth equivariant compactification. 	Let $\varepsilon>0$.
		Then there exists a finite subset $P\subset \mathcal{S}(A)\backslash\{A\}$ with the following property:
		Let $f:Y\da A\times S$ satisfies $f(Y)\not\sqsubset Z$.
		Then either one of the followings is true:
		\begin{enumerate}
			\item
			$
			N^{(1)}_f(r,Z)\leq \varepsilon T_{f_A}(r,L)+O_{\varepsilon}(N_{\ram\pi}(r)+\pN{f}(r)
			+T_{f_S}(r))+O(\log r)+o(T_f(r))||_{\varepsilon}$.
			\item
			$e(f,Z)\geq \dim A-1$.
			\item
			There exists $B\in P$ such that 
			$$
			T_{q_B\circ f_A}(r)=O(N_{\ram\pi}(r)+\pN{f}(r)+T_{f_S}(r))+O(\log r)+o(T_f(r))||, 
			$$
			where $q_B:A\to A/B$ is the quotient map.
		\end{enumerate}
	\end{proposition}
	
	\begin{proof}
		We prove by Noether induction on $S$.
		So we assume that the the assertion is true for all irreducible $V\subsetneqq S$ and $Z|_{A\times V}\subset A\times V$.
		We are given $\varepsilon>0$.
		
		We set 
		$$
		\Psi(A,Z)=\{ x\in A\times S;\ \dim ((x+A)\cap Z)\geq \dim A-1\}.
		$$
		Note that $\Psi(A,Z)/A\subset S$ is a constructible set.
		
		First we cosnider the case that $\Psi(A,Z)/A\subset S$ is Zariski dense.
		Then there exists a non-empty Zariski open set $U\subset S$ such that $U\subset \Psi(A,Z)/A$.
		Set $V=S\backslash U$.
		Then $V\subsetneqq S$ is a proper Zariski closed subset.
		If $f_S(Y)\not\subset V$, then we have $e(f,Z)\geq \dim A-1$.
		Hence the second assertion of  \cref{pro:202208061} is valid.
		If $f_S(Y)\subset V$, then the induction hypothesis yields a finite subset $P'\subset \mathcal{S}(A)\backslash\{A\}$ so that one of the three assertions in \cref{pro:202208061} is valid.
		Hence our proposition is proved by setting $P=P'$.

		Next we consider the case $V_o=\overline{\Psi(A,Z)/A}\subsetneqq S$.
		Let $Z'\subset A\times S$ be the Zariski closure of $Z\cap (A\times (S\backslash V_o))$.
		Then the codimension of $Z'\subset A\times S$ is greater than one.
		We apply \cref{lem:202209151} for $Z'\subset A\times S$ to get $E\subsetneqq S$ and $P'\subset \mathcal{S}(A)\backslash \{A\}$.
		Then $E\subsetneqq S$.
		Let $V_1,\ldots,V_l$ be the irreducible components of $E\cup V_o$.
		By the induction hypothesis applied for each $V_i\subsetneqq S$, we get $P_i\subset \mathcal{S}(A)\backslash\{ A\}$.
		We set $P=P'\cup \cup_iP_i$.
		Now let $f:Y\da A\times S$ satisfies $f(Y)\not\sqsubset Z$.
		Assume that the second and the third assertions of  \cref{pro:202208061} are not valid.
		If $f_S(Y)\subset V_i$ for some $V_i$, then by the induction hypothesis, we get the first assertion of \cref{pro:202208061}.
		Hence we assume that $ f_S(Y)\not\subset V_i$ for all $V_i$.
		Then we have $f_S(Y)\not\subset E$.
		By \cref{lem:202209151}, we get
		$$
		N^{(1)}_{f}(r,Z')\leq \varepsilon T_{f_A}(r,L)+O_{\varepsilon}(N_{\ram\pi}(r)+\pN{f}(r)
		+T_{f_S}(r))+O(\log r)+o(T_f(r))||_{\varepsilon}.
		$$	
		By 
		$$N^{(1)}_f(r,Z)\leq N^{(1)}_{f}(r,Z')+N^{(1)}_{f_S}(r,V_o)\leq N^{(1)}_{f}(r,Z')+O(T_{f_S}(r))+O(\log r),$$
		we get the first assertion of \cref{pro:202208061} for the case that $ f_S(Y)\not\subset V_i$ for all $V_i$.
		Thus the proof is completed.
	\end{proof}

\begin{rem}\label{rem:2025:09041}
For the sake of completeness, we provide a proof of  \Cref{claim:20221021}. 
The proof is identical to that of \cite[Prop 6]{Yam15}, except for the compactifications of several objects.
Before we start the proof, we note that the statement of \cite[Cor 3]{Yam15} also holds for semi-abelian varieties $A$, not just abelian varieties as stated in the original paper.
This is because the proof of \cite[Cor 3]{Yam15} relies only on the natural splitting $J_l(A\times S)=A\times (J_l(A\times S)/A)$ induced from the splitting $TA=A\times\mathrm{Lie}A$, a property that also holds for semi-abelian varieties.
\begin{proof}[Proof of  \Cref{claim:20221021}]
First we constract the sequence of positive integers $n(1),n(2),n(3),\cdots$ and the ample line bundle $L$ on $A$.
Let $\overline{Z}\subset \overline{A}\times S$ be the schematic closure.
Let $r:\overline{Z}\to S$ be the composition of the morphisms
\begin{equation}\label{eqn:sqm}
\overline{Z}\hookrightarrow \overline{A}\times S\overset{q}{\to}S,
\end{equation}
where $q$ is the second projection and $\overline{Z}\hookrightarrow \overline{A}\times S$ is the closed immersion.
By the generic flatness theorem, there exists a non-empty Zariski open subset $S^o\subset S$ such that the restriction of the family $r:\overline{Z}\to S$ over $S^o$ is a flat family.
\par

Consider the pull back of the sequence of morphisms \eqref{eqn:sqm} by the second projection $q:A\times S\to S$:
$$A\times \overline{Z}\overset{s_1}{\hookrightarrow}A\times \overline{A}\times S\overset{s_2}{\to}A\times S,$$
where $s_2$ maps as
$$A\times \overline{A}\times S\ni (a,a',s)\mapsto (a,s)\in A\times S.$$
Put $s=s_2\circ s_1:A\times \overline{Z}\to A\times S$. 
\par

Let $L$ be an ample line bundle on $\overline{A}$.
(Later, we replace $L$ by a suitable translate.)
We define $\phi : A\times \overline{A}\times S\to \overline{A}$ by
$$A\times \overline{A}\times S\ni (a,a',s)\mapsto a+a' \in \overline{A}.$$
Let $L^{\dagger}$ be the line bundle on $A\times \overline{Z}$ which is the pull back of $L$ by the composition of morphisms
$$A\times \overline{Z}\overset{ s_1}{\to} A\times \overline{A}\times S\overset{\phi}{\to} \overline{A}.$$
Since the restriction of $s:A\times \overline{Z}\to A\times S$ to $A\times S^o$ is a flat family, the semicontinuity theorem \cite[p.288]{hartshorne1977}
implies
that
there exists a Zariski open subset $U_n\subset A\times S^o$ $(n>0)$ such that $H^0((A\times \overline{Z})_{P},L^{\dagger \otimes n}_{P})$ are all the same dimensional $\mathbb C$ vector spaces for $P\in U_n$; 
put this number $G_n$. 
Here $(A\times \overline{Z})_{P}$ denotes the fiber of the morphism $s:A\times \overline{Z}\to A\times S$ over $P\in A\times S$, and $L^{\dagger \otimes n}_{P}$ is the induced line bundle.
Since the intersection $\underset{n}{\cap}U_n$ is non-empty, we may take a point $(a,w)\in \underset{n}{\cap}U_n$. 
Replacing $L$ by the pull back by the morphism 
$$\overline{A}\ni x \mapsto x+a \in \overline{A}$$
we may assume that 
\begin{equation}\label{eqn:10.29.1}
\text{$\{ 0\}\times S\subset A\times S$ and $\underset{n}{\cap}U_n$ have non-trivial intersection.}
\end{equation}
Since $s:A\times \overline{Z}\to A\times S$ is flat over $A\times S^o$ and $r:\overline{Z}\to S$ is dominant, we have
$$\mathrm{dim}\ (A\times \overline{Z})_{P}=\mathrm{dim}\ (A\times Z)-\mathrm{dim}\ (A\times S)
$$
for all $P\in \underset{n}{\cap}U_n\subset A\times S^o$.
Since $\mathrm{codim}(Z, A\times S)\geq 2$, we have
$$\mathrm{dim}\ (A\times \overline{Z})_{P}\leq \mathrm{dim}\ A-2.$$
Hence there exist a positive integer $n_0$ and positive constants $C_1$, $C_2$ such that
$$G_n<C_1 n^{\dim A -2}, \qquad \dim _{\mathbb C}H^{0}(\overline{A},L^{\otimes n})>C_2 n^{\dim A}$$
for $n>n_0$.
Hence for a positive integer $l$, we may take a positive integer $n(l)$ such that
\begin{equation}\label{eqn:dim}
\dim _{\mathbb C}H^{0}(\overline{A},L^{\otimes n(l)})>(l+1)G_{n(l)}
\end{equation}
and
$$\underset{l\to \infty}{\lim}
\frac{n(l)}{l}\to 0.$$
For instance, $n(l)\sim l^{2/3}$.

\par
Now let $S_l\subset J_l(A\times S)/A$ be an irreducible Zariski closed subset whose
image under the projection $J_l(A\times S)/A\to S$ is Zariski dense.
Let $\mathcal{T} _l\subset J_{l}(A\times S)\times S$ by the closed subscheme considered in \cite[Corollary 3]{Yam15}.
Let $\mathcal{S} _l \subset A\times S_l\times S$ be the closed subscheme obtained by the pull-back of $\mathcal{T} _l$ by the closed immersion
$$A\times S_l\times S\hookrightarrow J_{l}(A\times S)\times S.$$
Note that the under lying topological spaces of $\mathcal{S} _{l}$ and $S_l$ are the same.
We consider the following commutative diagram \eqref{eqn:cdm} which is obtained by the base change of \eqref{eqn:sqm} with a sequence of morphisms
$$\mathcal{S} _l \hookrightarrow  A\times S_l\times S\to  A\times S\to S.$$
Here $A\times S_l\times S\to  A\times S$ is the product of the first and third projections.
\begin{equation}\label{eqn:cdm}
\begin{CD}
\mathcal{Z} _l @>v'>> A\times S_l\times \overline{Z}@>>> A\times \overline{Z}@>>> \overline{Z}\\
@VVu_1V @VVt_1V @VVs_1V @VVV \\
\cdot @>>> A\times S_l\times \overline{A}\times S @>>> A\times \overline{A}\times S@>>> \overline{A}\times S\\
@VVu_2V @VVt_2V @VVs_2V @VVq V \\
\mathcal{S} _l @>v>> A\times S_l\times S@>>> A\times S@>>> S
\end{CD}
\end{equation}
Let $\mathcal{L} ^{\dagger}$ be a line bundle on $\mathcal{Z} _l$ obtained by the pull back of $L^{\dagger}$ by the morphisms in the above diagram \eqref{eqn:cdm}.  
Let $\mathcal{S}_{l,n}\subset \mathcal{S}_l$ be a Zariski open set of $\mathcal{S}_l$ obtained by the inverse image of $U_n\subset A\times S$.
Since the image of $S_l$ under the projection $J_l(A\times S)/A\to S$ is Zariski dense and we are assuming \eqref{eqn:10.29.1}, $\mathcal{S}_{l,n}$ is non-empty.
\par

Since $\dim _{\mathbb C}H^0((A\times \overline{Z})_{P},L^{\dagger \otimes n}_{P})=G_n$ for all $P\in U_n$,
the direct image sheaf $s_{*}L^{\dagger \otimes n}$ is a locally free sheaf of rank $G_n$ on $U_n$ and the natural map
$$s_{*}L^{\dagger \otimes n}\otimes \mathbb C (P)\to H^0((A\times \overline{Z})_{P},L^{\dagger \otimes n}_{P})$$
is an isomorphism for $P\in U_n$. 
This follows by the Theorem of Grauert \cite[p.288]{hartshorne1977} since $U_n$ is reduced and irreducible. 
Let $u:\mathcal{Z}_l\to \mathcal{S}_l$ be the morphism obtained by the composition $u=u_2\circ u_1$, where $u_1$ and $u_2$ are the morphisms in \eqref{eqn:cdm}.
Then the natural map 
$$u_{*}\mathcal{L}^{\dagger \otimes n}\otimes \mathbb C (P)\to H^0((\mathcal{Z}_{l})_{P},\mathcal{L}^{\dagger \otimes n}_{P})$$
is also surjective, so isomorphism on $P\in \mathcal{S}_{l,n}$. 
This follows by the Theorem of Cohomology and Base Change \cite[p.290]{hartshorne1977}. 
Hence $u_{*}\mathcal{L}^{\dagger \otimes n}$ is locally generated by $G_n$ elements as $\mathcal{O}_{\mathcal{S}_{l}}$ module on $\mathcal{S}_{l,n}\subset \mathcal{S} _{l}$. 
Let $\pi :\mathcal{S}_l\to S_l$ be the composition of the morphisms
$$\mathcal{S}_l\overset{v}{\hookrightarrow} A\times S_l\times S \overset{\text{2nd proj}}{\longrightarrow} S_l,$$
and let $S_{l,n}=\pi (\mathcal{S}_{l,n})$ be a Zariski open set (note that the under lying topological spaces of $\mathcal{S}_{l}$ and $S_l$ are the same). 
By \cite[Corollary 3]{Yam15}, $\pi$ is a finite morphism and the direct image sheaf $\pi _*\mathcal O_{\mathcal{S}_{l}}$ is locally generated by $(l+1)$ elements as an $\mathcal{O}_{S_l}$ module on $S_l\cap J_{l}^{\text{reg}}(A\times S)/A$.
Hence $(\pi \circ u)_{*}\mathcal{L}^{\dagger \otimes n}$ is locally generated by $(l+1)G_n$ elements as an $\mathcal{O}_{S_l}$ module on $S_{l,n}\cap J_{l}^{\text{reg}}(A\times S)/A$.
\par

Now look at the following commutative diagram
\begin{equation*}
\begin{CD}
\mathcal{Z}_l @>t_1\circ v'>> A\times S_l\times \overline{A}\times S @>\psi >> \overline{A}\times S_l@>p_l|_{S_l}>> \overline{A}\\
@VV\pi \circ uV @VV\text{2nd proj}V @VVq_l|_{S_l} V \\
S_l @= S_l @=  S_l
\end{CD}
\end{equation*}
where $\psi$ is the morphism
$$\psi :A\times S_l\times \overline{A}\times S \ni (a,s,a',s')\mapsto (a+a',s)\in \overline{A}\times S_l,$$
and $p_l|_{S_l}$ (resp. $q_l|_{S_l}$) is the restriction of $p_l$ (resp. $q_l$) over $\overline{A}\times S_l$.
Since $(p_l|_{S_l}\circ \psi \circ t_1\circ v')^*L=\mathcal{L}^{\dagger}$, we have a natural morphism
\begin{equation}\label{eqn:di}
(q_l|_{S_l})_*(p_l|_{S_l})^{*}L^{\otimes n}=H^{0}(\overline{A},L^{\otimes n})\otimes _{\mathbb C}\mathcal{O}_{S_l}\to (\pi \circ u)_{*}\mathcal{L}^{\dagger \otimes n}.
\end{equation}
Let $\mathcal F$ be the kernel of \eqref{eqn:di} for $n=n(l)$.
Then by \eqref{eqn:dim}, we have $\mathcal{F}\not= 0$.
Let 
$$\mathcal{F}'\subset  (q_l)_{*}p_{l}^{*}L^{\otimes n} = H^0(\overline{A}, L^{\otimes n})\otimes_{\mathbb C} \mathcal{O}_{\overline{S_l}}$$
be a coherent sheaf on $\overline{S_l}$ which is an extension of $\mathcal{F}$ (cf. [2, II Ex. 5.15]).
By tensoring a sufficiently ample line bundle $M_l$ on $\overline{S_l}$ with $\mathcal{F}'$, we may assume that $
H^0( \overline{S_l},\mathcal{F}'\otimes M_l)\not= 0$.
Since we have
\begin{equation*}
\begin{split}
H^0( \overline{S_l},\mathcal{F}'\otimes M_l)&\subset H^0(\overline{S_l}, ((q _l)_{*}p_{l}^{*}L^{\otimes n})\otimes M_l)\\
&=H^0(\overline{S_l}, (q _l)_{*}(p_{l}^{*}L^{\otimes n}\otimes q _l^*M_l))\\
&=H^0(\overline{A}\times \overline{S_l}, p_{l}^{*}L^{\otimes n}\otimes q _l^*M_l),
\end{split}
\end{equation*}
we may take a divisor $F_l\subset \overline{A}\times \overline{S_l}$ which is defined by a non-zero global section of $H^0( \overline{S_l},\mathcal{F}'\otimes M_l)$.
Let $F_l|_{S_l}$ be the restriction of $F_l$ on $\overline{A} \times  S_l$.
Then we have
$$ \mathcal{Z}_l\subset  \psi ^{*}(F_l|_{S_l})$$
as closed subschemes of $A\times S_l\times \overline{A}\times S$.
\par

Now let $f:\mathbb D \to A\times S$ be a holomorphic map such that $j_l(f)(\mathbb D )\subset A\times S_l$, $j_l(f)(\mathbb D )\not\subset F_l$, and $f(0)=(a,s)\in Z$.
We define a holomorphic map $\overset{\sim}{f}:\mathbb D \to A\times S_l\times \overline{A}\times S$ by
$$\overset{\sim}{f}(z)=(p\circ f(z)-a,q _l \circ j_{l}(f)(z),a,s),$$
where $p:A\times S\to A$ is the first projection.
Then we have
$$\overset{\sim}{f}(\mathbb D )\subset A\times S_l\times \overline{Z}, \quad  \overset{\sim}{f}(0)\in \mathrm{supp} \mathcal{Z}_l, \quad  \psi \circ \overset{\sim}{f}=j_{l}(f).$$
Since $v'$ is the base change of $v$ in \eqref{eqn:cdm}, we have
$$\ord _{0}\overset{\sim}{f} {}^*\mathcal{Z}_{l}=\ord _{0}(t_2\circ \overset{\sim}{f})^*\mathcal{S}_l
=\ord _{0}(j_{l}(f)-a)^*(\mathcal{T}_{l})_s.
$$
Hence by \cite[Lemma 6 (2)]{Yam15}, we have
$$\ord _{0}\overset{\sim}{f} {}^*\mathcal{Z}_{l}\geq l+1.$$
Hence we have
$$\ord _{0}j_{l}(f)^*F_l=
\ord _{0}\overset{\sim}{f} {}^*(\psi ^{*}(F_l|_{S_l}))\geq \ord _{0}\overset{\sim}{f} {}^*\mathcal{Z}_{l}\geq l+1.$$
This complete the proof of \Cref{claim:20221021}.
\end{proof}
\end{rem}

	\section{Second main theorem type estimate with truncation level one}\label{subsec:4.7a}
	In this subsection, we assume that $S$ is smooth and projective.
	In particular, all divisors on $A\times S$ are Cartier divisors.
	
	To prove the main result of this section, \cref{lem:20220909}, we start from the following lemma.
	
	\begin{lem}\label{lem:20220915}
		Let $D\subset A\times S$ be a reduced divisor.
		Let $L$ be an ample line bundle on $\overline{A}$, where $\overline{A}$ is a smooth equivariant compactification. 	Let $\varepsilon>0$.
		Then there exist a finite subset $P\subset \mathcal{S}(A)\backslash\{A\}$ and a proper Zariski closed subset $E\subsetneqq S$ with the following property:
		Let $f:Y\da A\times S$ satisfies $f(Y)\not\sqsubset D$.
		Then either one of the followings is true:
		\begin{enumerate}
			\item
			$
			N^{(2)}_f(r,D)-N^{(1)}_f(r,D)\leq \varepsilon T_{f_A}(r,L)+O_{\varepsilon}(N_{\ram\pi}(r)+\pN{f}(r)
			+T_{f_S}(r))+O(\log r)+o(T_f(r))||_{\varepsilon}$.
			\item
			$f(Y)\subset A\times E$.
			\item
			There exists $B\in P$ such that 
			$$
			T_{q_B\circ f_A}(r)=O(N_{\ram\pi}(r)+\pN{f}(r)+T_{f_S}(r))+O(\log r)+o(T_f(r))||, 
			$$
			where $q_B:A\to A/B$ is the quotient map.
		\end{enumerate}
	\end{lem}

	\begin{proof}
		We prove in the following three cases.
		
		\textbf{Case 1.}\ 
		$D$ is irreducible and $\mathrm{St}^0(D)=\{0\}$.
		We recall a Zariski open set $S^o\subset S$ (cf. \cref{lem:202210041}).
		By replacing $S^o$ by a smaller non-empty Zariski open set, we may assume that $D_s\subset A$ is a reduced divisor for all $s\in S^o$.
		Set $E=S\backslash S^o$.
		Set $S'=\overline{(\mathrm{Lie}A\times TS)}$. 	
		We apply \cref{pro:202208061} for $D^{(1)}\subset A\times S'$ and $\varepsilon>0$ to get a finite subset $P\subset \mathcal{S}(A)\backslash\{A\}$.
		
		Let $f:Y\da A\times S$ satisfies $f(Y)\not\sqsubset D$.
		We assume that the assertions 2 and 3 of the statement do not occur.
		Then $e(f,D)=\dim A-1$.
		If $e(j_1f,D^{(1)})= \dim A-1$, then by \cref{lem:202210041}, we have $T_f(r)=O(T_{f_{S}}(r))+O(\log r)$.
		Hence
		$$
		m_f(r,D)+N_f(r,D)=O(T_{f_S}(r))+O(\log r).
		$$
		This is stronger than the assertion 1.
		Thus in the following, we assume $e(j_1f,D^{(1)})<\dim A-1$.
		Then by \cref{pro:202208061}, we get
		$$
		N^{(1)}_{j_1f}(r,D^{(1)})\leq \varepsilon T_{f_A}(r,L)+O_{\varepsilon}(N_{\ram\pi}(r)+\pN{f}(r)
		+T_{(j_1f)_{S'}}(r))+O(\log r)+o(T_f(r))||_{\varepsilon}.
		$$
		By  \cref{lem:202210042} and
		$$
		N^{(2)}_f(r,D)-N^{(1)}_f(r,D)\leq N^{(1)}_{j_1f}(r,D^{(1)})+N_{\ram\pi}(r)+\pN{f}(r),
		$$
		we conclude the proof in the case 1.
		
		\textbf{Case 2.}\
		$D$ is irreducible, but $\mathrm{St}^0(D)$ is general.
		We set $B=\mathrm{St}^0(D)$.
		We apply the case 1 above for $A'=A/B$, $D'=D/B$ to get $P'\subset \mathcal{S}(A')\backslash\{ A'\}$ and $E\subsetneqq S$.
		We define $P\subset \mathcal{S}(A)\backslash \{A\}$ by $C\in P$ if and only if $B\subset C$ and $C/B\in P'$.
		
		\textbf{Case 3.}\
		We treat the general case.
		Let $D=D_1+\cdots+D_k$ be the irreducible decomposition.
		Set $Z_{ij}=D_i\cap D_j$.
		Then $Z_{ij}\subset A\times S$ has codimension greater than one.
		We apply \cref{pro:202208061} for $Z_{ij}$ to get $P_{ij}\subset \mathcal{S}(A)\backslash \{A\}$ and $E_{ij}\subsetneqq S$ such that $\overline{\Psi(A,Z_{ij})}=A\times E_{ij}$.
		We apply the case 2 above for each $D_i$ to get $P_i\subset \mathcal{S}(A)\backslash \{A\}$ and $E_i\subsetneqq S$.
		Then we set $P=\cup_iP_i\cup \cup_{i,j}P_{ij}$ and $E=\cup_{i}E_i\cup \cup_{i,j}E_{ij}$.
	\end{proof}
	
	\begin{proposition}\label{lem:20220909}
		Let $D\subset A\times S$ be a reduced divisor.
		Let $\overline{A}$ be a smooth equivariant compactification such that $p(\mathrm{Sp}_A\overline{D})\subsetneqq S$ is a proper Zariski closed set, where $p:\overline{A}\times S\to S$ is the second projection and $\overline{D}\subset \overline{A}\times S$ is the Zariski closure.
		Let $L$ be an ample line bundle on $\overline{A}$.
		Let $\varepsilon>0$.
		Then there exist
		a finite subset $P\subset \mathcal{S}(A)\backslash \{ A\}$, and a proper Zariski closed subset $E\subsetneqq S$
		such that for every $f:Y\da A\times S$ with $f(Y)\not\sqsubset D$, either one of the followings holds:
		\begin{enumerate}
			\item
			There exists $B\in P$ such that 
			$$
			T_{q_B\circ f_A}(r)=O(N_{\ram\pi}(r)+\pN{f}(r)+T_{f_S}(r))+O(\log r)+o(T_f(r))||, 
			$$
			where $q_B:A\to A/B$ is the quotient map.
			\item
			$f(Y)\sqsubset A\times E$.
			\item
			\begin{multline*}
				m_f(r,\overline{D})+N_f(r,\overline{D})\leq N^{(1)}_f(r,\overline{D})
				+\varepsilon T_{f_A}(r,L)\\
				+O_{\varepsilon}(N_{\ram\pi}(r)+\pN{f}(r)
				+T_{f_S}(r))+O(\log r)||.
			\end{multline*}
		\end{enumerate}
	\end{proposition}

	\begin{proof}
		We apply \cref{pro:202208062} for $\overline{D}\subset \overline{A}\times S$ to get a finite subset $P_1\subset \mathcal{S}(A)\backslash\{A\}$ and $\rho\in\mathbb Z_{>0}$.
		We apply \cref{lem:20220915} for $D\subset A\times S$ and $\varepsilon'=\varepsilon/(\rho-1)$ to get a finite subset $P_2\subset \mathcal{S}(A)\backslash\{A\}$ and $E'\subsetneqq S$.
		We set $P=P_1\cup P_2$ and $E=p(\mathrm{Sp}_A\overline{D})\cup E'$.

		Now let $f:Y\da A\times S$ satisfies $f(Y)\not\sqsubset D$.
		We assume that the assertions 1 and 2 of our lemma do not hold.
		By \cref{pro:202208062}, we get
		\begin{equation*}
			m_f(r,D)+N_f(r,D)-N^{(\rho)}_f(r,D)=O(N_{\ram\pi}(r)
			+\pN{f}(r)+T_{f_S}(r))+O(\log r)+o(T_f(r))||.
		\end{equation*}
		We apply \cref{lem:20220915} to get
		$$N^{(2)}_f(r,D)-N^{(1)}_f(r,D)\leq \varepsilon' T_{f_A}(r,L)+O(N_{\ram\pi}(r)+\pN{f}(r)
		+T_{f_S}(r))+O(\log r)+o(T_f(r))||.$$
		By $N^{(\rho)}_f(r,D)-N^{(1)}_f(r,D)\leq (\rho-1) (N^{(2)}_f(r,D)-N^{(1)}_f(r,D))$, we get
		\begin{multline}\label{eqn:20220915}
			m_f(r,D)+N_f(r,D)-N^{(1)}_f(r,D)\leq (\rho-1)\varepsilon'T_{f_A}(r,L)
			\\
			+
			O(N_{\ram\pi}(r)+\pN{f}(r)
			+T_{f_S}(r))+O(\log r)+o(T_f(r))||.
		\end{multline}
		This completes the proof.
	\end{proof}

	\section{Proof of \cref{thm2nd}} 
	\label{subsec:4.7}

	In this section, we complete the proof of \cref{thm2nd} by proving a stronger result \cref{thm:20221201}.
	Let $B\subset A$ be a semi-abelian subvariety and let
	$\overline{A}$ be an equivariant compactification.
	
	\begin{dfn}
		We say that $\overline{A}$ is compatible with $B$ if there exists an equivariant compactification $\overline{A/B}$ for which the quotient map $q_{B}:A\to A/B$ extends to a morphism $\overline{q_B}:\overline{A}\to \overline{A/B}$.
	\end{dfn}

	Let $W\subset A\times S$ be a $B$-invariant closed subvariety.
	Then we have $W/B\subset (A/B)\times S$.
	If $\overline{A}$ is compatible with $B$, then $\overline{q_B}$ yields $p:\overline{A}\times S\to(\overline{A/B})\times S$.

	\begin{lem}\label{lem:20221124}
		Let $W\subset A\times S$ be a $B$-invariant closed subvariety.
		Let $D\subset W$ be a reduced divisor.
		Then there exists a smooth, projective, equivariant compactification $\overline{A}$ which is compatible with $B$ and $p(\mathrm{Sp}_B\overline{D})\subsetneqq \overline{W/B}$, where $\overline{D}\subset \overline{A}\times S$ is the Zariski closure.
	\end{lem}
	
	\begin{proof}
		We consider $W\to W/B$.
		We take the schematic generic point $\eta\in W/B$.
		Let $D_{\eta}\subset W_{\eta}$ be the schematic generic fiber.
		Let $\bar{\eta}:\mathrm{Spec}\, \overline{\mathbb C(\eta)}\to W/B$ be the geometric point.
		Then $W_{\bar{\eta}}=B\times_{\mathbb C}\overline{\mathbb C(\eta)}$ is a semi-abelian variety over $\overline{\mathbb C(\eta)}$.
		By \cite[Lem 12.7]{Y22}, 
		there exists a smooth equivariant compactification $\overline{W_{\bar{\eta}}}$ such that $\overline{D_{\bar{\eta}}}\subset \overline{W_{\bar{\eta}}}$ satisfies 
		\begin{equation}\label{eqn:20221028}
			\mathrm{Sp}_{W_{\bar{\eta}}}\overline{D_{\bar{\eta}}}=\emptyset.
		\end{equation}

		We remark that there exists a smooth equivariant compactification $\overline{B}$ such that $\overline{W_{\bar{\eta}}}$ is obtained by the base change of $\overline{B}$.
		Namely
		$$\overline{W_{\bar{\eta}}}=\overline{B}\times_{\mathbb C}\overline{\mathbb C(\eta)}.$$
		To see this, we consider the canonical extension $0\to T_B\to B\to B_0\to 0$, where $T_B$ is an algebraic torus and $B_0$ is an abelian variety.
		By \cite[Lem A.6]{Y22}, there exists a smooth equivariant compactification $\overline{(T_B)_{\bar{\eta}}}$ such that $\overline{W_{\bar{\eta}}}=(\overline{(T_B)_{\bar{\eta}}}\times W_{\bar{\eta}})/(T_B)_{\bar{\eta}}$, where $(T_B)_{\bar{\eta}}=T_B\times_{\mathbb C}\overline{\mathbb C(\eta)}$.
		By Sumihiro's theorem, $\overline{(T_B)_{\bar{\eta}}}$ is a torus embedding associated to some complete fan (cf. \cite[Sec. A.2]{Y22}).
		Hence there exists a smooth equivariant compactification $\overline{T_B}$ such that $\overline{(T_B)_{\bar{\eta}}}=\overline{T_B}\times_{\mathbb C}\overline{\mathbb C(\eta)}$.
		Let $\overline{B}=(\overline{T_B}\times B)/T_B$.
		Then $\overline{B}\times_{\mathbb C}\overline{\mathbb C(\eta)}=\overline{W_{\bar{\eta}}}$ as desired.
		By \cite[Lem A.8]{Y22}, we may assume moreover that $\overline{B}$ is projective by replacing $\overline{B}$ by a modification $\hat{B}\to \overline{B}$, for the property \eqref{eqn:20221028} remains true under this modification.
		Then $\overline{T_B}$ is projective.

		Next we shall construct an equivariant compactification $\overline{A}$ which is compatible with $B$ and the general fibers of $\overline{A}\to\overline{A/B}$ are $\overline{B}$ constructed above.
		We have the canonical extensions $0\to T_B\to B\to B_0\to 0$, $0\to T_A\to A\to A_0\to 0$ and $0\to T_{A/B}\to A/B\to A_0/B_0\to 0$.
		Then we have $0\to T_B\to T_A\to T_{A/B}\to 0$.
		We may take a section $T_{A/B}\to T_A$ so that $T_A=T_B\times T_{A/B}$.
		Set $A'=A/T_{A/B}$.
		Then $B\subset A'$ and $A'/B=A_0/B_0$.
		The $B$-torsor $A\to A/B$ is the pull-back of $A'\to A'/B$ by $A/B\to A'/B$. 
		
		Let $\overline{B}$ be the compactification above so that $\overline{B}=(\overline{T_B}\times B)/T_B$.
		Set $\overline{A'}=(\overline{T_B}\times A')/T_B$.
		Then $\overline{A'}=(\overline{B}\times A')/B$. 
		Then the general fibers of $\overline{A'}\to A'/B$ are $\overline{B}$.
		Let $\overline{T_{A/B}}$ be a smooth projective equivariant compactification.
		Set $\overline{A}=((\overline{T_B}\times \overline{T_{A/B}})\times A)/(T_B\times T_{A/B})$ and $\overline{A/B}=(\overline{T_{A/B}}\times (A/B))/T_{A/B}$.
		Then $\overline{A}$ is compatible with $B$ and the general fibers of $\overline{A}\to\overline{A/B}$ are $\overline{B}$.
		Moreover $\overline{A}$ and $\overline{A/B}$ are smooth and projective by the proof of \cite[Lem A.8]{Y22}.

		Now let $\overline{D}\subset \overline{W}$ be the compactification of $D$.
		Then we claim
		\begin{equation}\label{eqn:202210281}
			(\overline{D})_{\bar{\eta}}=\overline{D_{\bar{\eta}}}
		\end{equation}
		in $(\overline{W})_{\bar{\eta}}
		$.
		To see this, we note that the induced map $\overline{W}_{\eta}\to \overline{W}$ is homeomorphism onto its image (cf. \cite[\href{https://stacks.math.columbia.edu/tag/01K1}{Tag 01K1}]{stacks-project}). 
		Hence $(\overline{D})_{\eta}=\overline{D_{\eta}}$.
		Note that $(\overline{W})_{\bar{\eta}}\to (\overline{W})_{\eta}$ is a closed morphism (cf.  \cite[\href{https://stacks.math.columbia.edu/tag/01WM}{Tag 01WM}]{stacks-project}).
		Hence $(\overline{D_{\eta}})_{\bar{\eta}}=\overline{D_{\bar{\eta}}}$.
		Combining these two observations, we have
		$(\overline{D})_{\bar{\eta}}=(\overline{D_{\eta}})_{\bar{\eta}}=\overline{D_{\bar{\eta}}}$.
		This shows \eqref{eqn:202210281}.
		Since $\mathrm{Sp}_B\overline{D}\subset \overline{W}$ is $B$-invariant, $(\mathrm{Sp}_B\overline{D})_{\bar{\eta}}\subset (\overline{W})_{\bar{\eta}}$ is $B_{\bar{\eta}}$-invariant.
		Moreover $(\mathrm{Sp}_B\overline{D})_{\bar{\eta}}\subset (\overline{D})_{\bar{\eta}}$.
		Hence $(\mathrm{Sp}_B\overline{D})_{\bar{\eta}}\subset\mathrm{Sp}_{B_{\bar{\eta}}}((\overline{D})_{\bar{\eta}})$.
		Combining this with \eqref{eqn:20221028} and \eqref{eqn:202210281},
		we have
		$(\mathrm{Sp}_B\overline{D})_{\bar{\eta}}\subset\mathrm{Sp}_{B_{\bar{\eta}}}\overline{D_{\bar{\eta}}}=\emptyset$.
		Hence $p(\mathrm{Sp}_B\overline{D})\subset \overline{W/B}$ is a proper Zariski closed set.
	\end{proof}

	Let $X$ be a quasi-projective variety.
	Let $D\subset X$ be a reduced (Weil) divisor.
	Let $E\subset D$ be a subset.
	Let $f:Y\to X$ be a holomorphic map such that $f(Y)\not\subset D$.
	We set
	$$
	\overline{N}_f(r, E)=\frac{1}{\mathrm{deg}\ \pi} \int_{2\delta}^{r}\mathrm{card}\ (Y(t)\cap f^{-1}E)\ \frac{d t}{t} .
	$$
	When $D$ is a Cartier divisor, then we have $\overline{N}_f(r,D)=N^{(1)}_f(r,D)$.

	Let $S$ be a projective variety and let $B\subset A$ be a semi-abelian subvariety.
	We denote by $I_B$ the set of all holomorphic maps $f:Y\to A\times S$ such that the following three estimates hold:
	\begin{itemize}
		\item 
		$T_{q_B\circ f_A}(r)=O(\log r)+o(T_f(r))||$, where $q_B:A\to A/B$ is the quotient,
		\item
		$N_{\ram\pi}(r)=O(\log r)+o(T_f(r))||$,
		\item
		$T_{f_S}(r)=O(\log r)+o(T_f(r))||$.
	\end{itemize}
	
	\begin{rem}\label{rem:20221203}
		If $f\in I_{\{0\}}$, then $f:Y\to A\times S$ does not have essential singularity over $\infty$.
		Indeed, $f\in I_{\{0\}}$ yields $T_{f_A}(r)=O(\log r)+o(T_f(r))||$ and  $T_{f_S}(r)=O(\log r)+o(T_f(r))||$.
		Hence $T_f(r)=O(\log r)||$, thus $N_{\ram\pi}(r)=O(\log r)||$.
		These two estimates implies that $f$ does not have essential singularity over $\infty$ (cf. \cref{lem:20230411} ). 
	\end{rem}
	
	For a proper birational morphism $\varphi:V\to W$, we denote by $\mathrm{Ex}(\varphi)\subset V$ the exceptional locus of $\varphi$.
	If $f:Y\to W$ is a holomorphic map such that $f(Y)\not\subset \varphi(\mathrm{Ex}(\varphi))$, then there exists a unique lift $g:Y\to V$ of $f$.

	The following lemma brings together the results from \cref{sec:rbc} to \cref{subsec:4.7a} and is essential for the applications that follow.

	\begin{lem}\label{lem:20220909gg}
Let $B \subset A$ be a semi-abelian subvariety, and let $W \subset A \times S$ be a $B$-invariant closed subvariety, where $S$ is projective. Let $D \subset W$ be a reduced Weil divisor.  
Then the following objects exist with the specified property described later: 
\begin{itemize}
    \item a smooth projective equivariant compactification $A \subset \overline{A}$,
    \item a proper birational morphism $\varphi: V \to \overline{W}$ with $V$ smooth, such that 
    \(\varphi^{-1}(\overline{D} \cup \partial W) \subset V\) is a simple normal crossing divisor.
\end{itemize}
The satisfied property: Let $Z \subset W$ be a Zariski closed subset with $\mathrm{codim}(Z,W) \ge 2$ and $Z \subset D$, let $L$ be an ample line bundle on $V$, and let $\varepsilon > 0$. Then there exist
\begin{itemize}
    \item a finite subset $P \subset \mathcal{S}(B) \setminus \{B\}$, and
    \item a proper Zariski closed subset $\Phi \subsetneqq W$ with $\varphi(\mathrm{Ex}(\varphi)) \subset \overline{\Phi}$,
\end{itemize}
such that for every holomorphic map $f: Y \to W$ satisfying $f(Y) \not\subset D \cup \Phi$ and $f \in I_B$, one of the following holds:
\begin{enumerate}
    \item There exists $C \in P$ such that $f \in I_C$.
    \item Let $f': Y \to V$ be the lift of $f$. Then
    \[
    T_{f'}(r, K_V(\varphi^{-1}(\overline{D} \cup \partial W))) \le \overline{N}_f(r, D \setminus Z) + \varepsilon\, T_{f'}(r,L) + O(\log r) + o(T_{f'}(r))||.
    \]
\end{enumerate}
\end{lem}

Our main technical issue in the proof is that $W$ cannot be decomposed into $B\times (W/B)$, this obstructs the application of the previous results to $f: Y \to W$.
We handle this issue by considering a generically finite surjective map $\sigma:\Sigma\to \overline{W/B}$ which gives a rational map $\overline{B}\times \Sigma\dashrightarrow \overline{W}$.
We can then apply our previous results to the lift
 $Y'\to \overline{B}\times \Sigma$ and translate them to obtain the estimates for $f: Y \to W$.
 A careful argument is required to ensure that these strategies work correctly.
 This involves constructing several additional objects not mentioned above.
 See also \cref{rem:20250904} for additional discussion.

	\begin{proof}
	We divide the proof into several steps.
	
	\textbf{Step 1.} The construction of $\overline{A}$ and $\varphi: V \to \overline{W}$.
By \cref{lem:20221124}, we may take an equivariant compactification $\overline{A}$ such that
		\begin{itemize}
			\item
			$\overline{A}$ is smooth and projective, 
			\item
			$\overline{A}$ is compatible with $B$, and
			\item $p(\mathrm{Sp}_B\overline{D})\subsetneqq \overline{W/B}$, where $p:\overline{W}\to \overline{W/B}$ is induced from $\overline{A}\to\overline{A/B}$.
		\end{itemize}
		Let $U\subset W/B$ be the smooth locus of $W/B$.
		Then $p^{-1}(U)\subset \overline{W}$ is smooth.
		We take a smooth modification $\varphi_o:V_o\to \overline{W}$ which is isomorphism over $p^{-1}(U)$.
		We apply the embedded resolution of Szab{\'o}  \cite{Sza94} 
		for $\varphi_o^{-1}(\overline{D}+\partial W)\subset V_o$ to get a smooth modification $\tau:V\to V_o$ such that $\widetilde{D}=(\varphi_o\circ \tau)^{-1}(\overline{D}+\partial W)$ is simple normal crossing and $\tau$ is an isomorphism outside the locus where $\varphi_o^{-1}(\overline{D}+\partial W)$ is not simple normal crossing.
		Let $\varphi:V\to \overline{W}$ be the composite of $\tau:V\to V_o$ and $\varphi_o:V_o\to \overline{W}$.
		We have 
		\begin{equation}\label{eqn:20221126}
			K_{V}(\varphi^{-1}(\overline{D}\cup \partial W))
			\leq
			\tau^*K_{V_o}(\varphi_o^{-1}(\overline{D}+\partial W))+E
		\end{equation}
		for some effective divisor $E\subset V$ such that $\tau(E)\subset V_o$ is contained in the locus where $\varphi_o^{-1}(\overline{D}+\partial W)$ is not simple normal crossing.
		Since $p^{-1}(U)\cap \partial W$ is simple normal crossing, we have $p^{-1}(U)\cap \tau(E)\subset \overline{D}$.
		
	\textbf{Step 2.}
	The construction of $\Sigma$.
	Let $\Psi:B\times \overline{W}\to \overline{W}$ be the $B$-action.
		We may take a smooth subvariety $S'\subset W$ such that the induced map $S'\to U$ is {\'e}tale.
		Then $\Psi$ yields an {\'e}tale map 
		$$\psi:B\times S'\to p^{-1}(U)\cap W.$$
		Let $\overline{B}\subset\overline{A}$ be the compactification.
		Then $\overline{B}$ is an equivariant compactification.
		The map $\psi$ extends to a map
		$$
		\bar{\psi}:\overline{B}\times S'\to p^{-1}(U).
		$$ 
		Since $\overline{A}$ is compatible with $B$, $\bar{\psi}$ is {\'e}tale.
		We take a smooth compactification $S'\subset \Sigma$ such that the maps $S'\hookrightarrow W$ and $S'\to W/B$ extend to $\Sigma\to \overline{W}$ and $\sigma:\Sigma\to \overline{W/B}$.
		Then $\bar{\psi}:\overline{B}\times S'\to p^{-1}(U)\subset V_o$ induces a rational map $\widehat{\psi}:\overline{B}\times \Sigma\dashrightarrow V_o$.

		\[
		\begin{tikzcd}
			\overline{B}\times S' \arrow[r, hookrightarrow] \arrow[d] \arrow[bend left, "\bar{\psi}"]{rrr}&	\overline{B}\times \Sigma \arrow[r, dashrightarrow]
			\arrow[dashed, bend left, "\widehat{\psi}" ']{rr}
			\arrow[d, "q"] & \overline{W} \arrow[d, "p"] & V_o\arrow[l,"\varphi_o" ] &V\arrow[l,"\tau" ] \\
			S' \arrow[r, hookrightarrow] &	\Sigma \arrow[r, "\sigma"] & \overline{W/B}& &
		\end{tikzcd}
		\]
		We take a closed subscheme $I\subset \overline{B}\times\Sigma$ such that $q(\mathrm{supp} I)\subset \Sigma\backslash S'$ and the ratinal map $\widehat{\psi}:\overline{B}\times\Sigma\dashrightarrow V_o$ extends to a morphism 
		$$\widetilde{\psi}:\mathrm{Bl}_I(\overline{B}\times\Sigma)\to V_o.$$

\textbf{Step 3.}
		The construction of $P$ and $\Phi$ from the given $Z$, $L$, and $\varepsilon > 0$.
		Let $Z\subset W$ be a Zariski closed subset such that $\mathrm{codim}(Z,W)\geq 2$ and $Z\subset D$. 
		We take a closed subscheme $\widehat{Z}\subset V_o$ such that
		\begin{itemize}
			\item
			$\mathrm{supp}\widehat{Z}=\overline{p^{-1}(U)\cap (\tau(E)\cup\varphi_o^{-1}(Z))}$,
			\item
			$(p\circ \varphi)^{-1}(U)\cap E\subset (p\circ\varphi)^{-1}(U)\cap \tau^*\widehat{Z}$, as closed subschemes of $(p\circ\varphi)^{-1}(U)$.
		\end{itemize}
		Then we have $\mathrm{supp}(\widehat{Z})\subset \varphi_o^{-1}(\overline{D})$ and $\mathrm{codim}(\widehat{Z},V_o)\geq 2$.	
		
		Let $L$ be an ample line bundle on $V$.
		Let $L_{V_o}$ be an ample line bundle on $V_o$.
		Then, there exists a positive constant $c_1>0$ such that for every $g:Y\to V$, we have the estimate:
		\begin{equation}\label{eqn:202211303}
			T_{\tau\circ g}(r,L_{V_o})\leq c_1T_g(r,L)+O(\log r).
		\end{equation}
		Let $L_{\overline{B}}$ be an ample line bundle on $\overline{B}$.
		Let $\kappa:\mathrm{Bl}_I(\overline{B}\times\Sigma)\to \overline{B}$ be the composite of $\mathrm{Bl}_I(\overline{B}\times\Sigma)\to \overline{B}\times\Sigma$ and the first projection $\overline{B}\times\Sigma\to\overline{B}$.
		Then, there exists a positive constant $c_2>0$ such that for every $g:Y\to V_o$ with $g(Y)\not\subset \widetilde{\psi}(\mathrm{Bl}_I(\overline{B}\times\Sigma)\backslash (\overline{B}\times S'))$, we have the estimate (cf. \cite[Lemma 1]{Yam15}):
		\begin{equation}\label{eqn:202211304}
			T_{g'}(r,\kappa^*L_{\overline{B}})\leq c_2T_g(r,L_{V_o})+O(\log r),
		\end{equation}
		where $g':Y\to \mathrm{Bl}_I(\overline{B}\times\Sigma)$ is a lift of $g$.
		
		Let $D'\subset \overline{B}\times \Sigma$ be the reduced divisor defined by the Zariski closure of $\bar{\psi}^{-1}(\overline{D}\cap p^{-1}(U))\subset \overline{B}\times S'$.
		Then we have
		$$
		q(\mathrm{Sp}_BD')\subsetneqq \Sigma.
		$$
		Let $Z'\subset \overline{B}\times \Sigma$ be the schematic closure of $\bar{\psi}^{*}(\widehat{Z}\cap p^{-1}(U))\subset \overline{B}\times S'$.
		Then $\mathrm{codim}(Z',\overline{B}\times \Sigma)\geq 2$ and $\mathrm{supp}Z'\subset D'$.
		Hence we have $q(\mathrm{Sp}_BZ')\subsetneqq \Sigma$.
		We apply \cref{pro:202208062} for $Z'\subset \overline{B}\times \Sigma$ to get $P_0\subset \mathcal{S}(B)\backslash\{B\}$ and $\rho\in \mathbb Z_{>0}$.

		Let $\varepsilon>0$.
		We set $\varepsilon'=\frac{\varepsilon}{c_1c_2(\rho+2)}$.
		We may apply \cref{lem:20220909} for $D'\subset \overline{B}\times \Sigma$, $L_{\overline{B}}$ and $\varepsilon'$ to get $P_1\subset \mathcal{S}(B)\backslash\{B\}$ and $E_1\subsetneqq \Sigma$.
		We apply \cref{lem:202209151} for $Z'\subset \overline{B}\times \Sigma$, $L_{\overline{B}}$ and $\varepsilon'$ to get $P_2\subset \mathcal{S}(B)\backslash\{B\}$ and $E_2\subsetneqq \Sigma$.
		We take a non-empty Zariski open set $U'\subset U$ such that $S'\to U$ is finite over $U'$.
		We set 
		$$\Phi=\big( \varphi(\mathrm{Ex}(\varphi))\cup p^{-1}(\overline{W/B}\backslash U')\cup p^{-1}(\sigma (E_1\cup E_2\cup q(\mathrm{Sp}_BZ'))\big) \cap W.$$
		We set $P=P_0\cup P_1\cup P_2$.
		
		\textbf{Step 4.} The estimates for the lift $Y'\to \overline{B}\times \Sigma$.
		Now let $f:Y\to W$ be a holomorphic map such that $f(Y)\not\subset D\cup \Phi$ and $f\in I_B$.
		We assume that the first assertion of \cref{lem:20220909gg} does not valid.
		By $p\circ f(Y)\not\subset \overline{W/B}\backslash U'$, we may take a lift 
		$$g:Y'\to \overline{B}\times \Sigma$$ 
		such that 
		$$g^{-1}(\partial B\times \Sigma)\subset g_{\Sigma}^{-1}(\Sigma\backslash S').$$
		Hence we have 
		\begin{equation*}
			\pN{g}(r)\leq N_{g_{\Sigma}}(r,\Sigma\backslash S')=O(\log r)+o(T_f(r))||,
		\end{equation*}
		\begin{equation*}
			T_{g_{\Sigma}}(r)=O(\log r)+o(T_f(r))||.
		\end{equation*}
		Note that the covering $\mu:Y'\to Y$ is unramified outside $g_{\Sigma}^{-1}(\Sigma\backslash S')$.
		Hence we have
		\begin{equation*}
			N_{\ram\pi'}(r)=O(\log r)+o(T_f(r))||,
		\end{equation*}
		where $\pi':Y'\to \mathbb C_{>\delta}$ is the composite $\pi'=\pi\circ \mu$.

		Hence by \cref{lem:20220909}, we get 
		\begin{equation*}
			m_{g}(r,D')+N_{g}(r,D')-\overline{N}_{g}(r,D')\leq\varepsilon' T_{g_B}(r,L_{\overline{B}})
			+O(\log r)+o(T_f(r))||.
		\end{equation*}
		By \cref{pro:202208062} 
		\begin{equation*}
			m_{g}(r,Z')+N_{g}(r,Z')\leq N^{(\rho)}_{g}(r,Z')+
			O(\log r)+o(T_f(r))||.
		\end{equation*}
		By \cref{lem:202209151}, we get
		$$
		N^{(1)}_{g}(r,Z')\leq \varepsilon' T_{g_B}(r,L_{\overline{B}})
		+O(\log r)+o(T_f(r))||.
		$$
		Hence
		$$
		m_{g}(r,Z')+N_{g}(r,Z')\leq \rho\varepsilon'T_{g_B}(r,L_{\overline{B}})
		+O(\log r)+o(T_f(r))||.
		$$  
		Note that $K_{\overline{B}\times \Sigma}(\partial (B\times \Sigma))=q^*K_{\Sigma}$.
		Hence
		\begin{multline}\label{eqn:202211301}
			m_{g}(r,Z')+N_{g}(r,Z')+T_{g}(r,K_{\overline{B}\times \Sigma}(D'+\partial (B\times \Sigma)))\leq \overline{N}_{g}(r,D'\backslash Z')
			\\
			+(\rho+2)\varepsilon' T_{g_B}(r,L_{\overline{B}})
			+O(\log r)+o(T_f(r))||.
		\end{multline}

	\textbf{Step 5.}  Translating the estimate \eqref{eqn:202211301} to the original map $f:Y\to W$.
	We continue to consider $f:Y\to W$ from the previous step.
		Since $\bar{\psi}:\overline{B}\times S'\to p^{-1}(U)$ is {\'e}tale, we have
		$$
		\bar{\psi}^*K_{V_o}(\varphi_o^{-1}(\overline{D}+\partial W))=K_{\overline{B}\times \Sigma}(D'+\partial (B\times \Sigma))|_{\overline{B}\times S'}.
		$$
		Let $f':Y\to V$ be the lift of $f:Y\to W$.
		Then $\tau\circ f':Y\to V_o$ is the lift of $f$.
		By \cite[Lemma 3.1]{Yam06} we have  
		\begin{equation}
			T_{g}(r,K_{\overline{B}\times \Sigma}(D'+\partial (B\times \Sigma)))=T_{\tau\circ f'}(r,K_{V_o}(\varphi_o^{-1}(\overline{D}+\partial W)))+O(\log r)+o(T_f(r))||.
		\end{equation}
		Similarly we have (cf. \cite[Lemma 1]{Yam15})
		\begin{equation}
			m_{g}(r,Z')+N_{g}(r,Z')=m_{\tau\circ f'}(r,\widehat{Z})+N_{\tau\circ f'}(r,\widehat{Z})+O(\log r)+o(T_f(r))||,
		\end{equation}
		\begin{equation}\label{eqn:202211302}
			\overline{N}_{g}(r,D'\backslash Z')\leq \overline{N}_{f}(r,D\backslash Z)+O(\log r)+o(T_f(r))||.
		\end{equation}
		Hence combinning \eqref{eqn:202211301}--\eqref{eqn:202211302}, we get
		\begin{multline}
			m_{\tau\circ f'}(r,\widehat{Z})+N_{\tau\circ f'}(r,\widehat{Z})+T_{\tau\circ f'}(r,K_{V_o}(\varphi_o^{-1}(\overline{D}+\partial W)))
			\leq \overline{N}_{f}(r,D\backslash Z)
			\\
			+(\rho+2)\varepsilon' T_{g_B}(r,L_{\overline{B}})
			+O(\log r)+o(T_f(r))||.
		\end{multline}
		By the choice of $\widehat{Z}\subset V_o$, we have $(p\circ \varphi)^{-1}(U)\cap E\subset (p\circ\varphi)^{-1}(U)\cap \tau^*\widehat{Z}$.
		Hence
		$$
		m_{f'}(r,E)+N_{f'}(r,E)\leq m_{\tau\circ f'}(r,\widehat{Z})+N_{\tau\circ f'}(r,\widehat{Z})+O(\log r)+o(T_f(r))||.
		$$
		Hence by \eqref{eqn:20221126}, we get
		\begin{multline}
			T_{f'}(r,K_{V}(\varphi^{-1}(\overline{D}\cup \partial W)))
			\leq \overline{N}_{f}(r,D\backslash Z)
			\\
			+(\rho+2)\varepsilon' T_{g_B}(r,L_{\overline{B}})
			+O(\log r)+o(T_f(r))||.
		\end{multline}
		By \eqref{eqn:202211303} and \eqref{eqn:202211304}, we have $T_{g_B}(r,L_{\overline{B}})\leq c_1c_2T_{f'}(r,L)$.
		Thus we get the second assertion of \cref{lem:20220909gg}.
		The proof is completed.
	\end{proof}

	Let $V$ be a $\mathbb Q$-factorial, projective variety.
	Let $F\subset V$ be an effective Weil divisor.
	Then there exists a positive integer $k$ such that $kF$ is a Cartier divisor.
	Let $g:Y\to V$ be a holomorphic map.
	Then we set
	$$
	T_g(r,F)=\frac{1}{k}T_g(r,\mathcal{O}_V(kF))+O(\log r).
	$$
	This definition does not depend on the choice of $k$.

	The following lemma generalizes \cref{lem:20220909gg} by considering a finite map $X\to A\times S$ instead of a closed subvariety $W\subset A\times S$.
	This lemma is a consequence of \cref{lem:20220909gg}.
	
	\begin{lem}\label{cor:20220805}
		Let $B\subset A$ be a semi-abelian subvariety.
		Let $X$ be a normal variety with a finite map $a:X\to A\times S$, where $S$ is a projective variety.
		Let $D\subset X$ be a reduced Weil divisor.
		Then the following objects exist with the specified property described later: 
		\begin{itemize}
		\item a compactification $\overline{X}$,
		\item
		 a proper birational morphsim $\varphi:\overline{X}'\to \overline{X}$, where $\overline{X}'$ is $\mathbb Q$-factorial.
		 \end{itemize}
		 The satisfied property: 
		Let $Z\subset X$ be a Zariski closed set such that $\mathrm{codim}(Z,X)\geq 2$ and $Z\subset D$.
		Let $L$ be a big line bundle on $\overline{X}'$ and let $\varepsilon>0$.
		Then there exist  
		\begin{itemize}
			\item 
			a finite subset $P\subset \mathcal{S}(B)\backslash\{B\}$,
			\item
			a proper Zariski closed set $\Xi\subsetneqq X$ with $\varphi(\mathrm{Ex}(\varphi))\subset \overline{\Xi}$
		\end{itemize}
		such that for every holomorphic map $f:Y\to X$ with $f(Y)\not\subset D\cup\Xi$ and $a\circ f\in I_B$, one of the following holds:
		\begin{enumerate}
			\item
			There exists $C\in P$ such that $a\circ f\in I_C$.
			\item
			Let $H\subset \overline{X}'$ be a reduced Weil divisor defined by $\varphi^{-1}(\overline{D\cup X_{\mathrm{sing}}\cup \partial X})=H\cup \Gamma$, where $\Gamma\subset \overline{X}'$ is a Zariski closed subset of codimension greater than one.
			Let $g:Y\to \overline{X}'$ be the lift of $f$. 		
			Then
			\begin{equation*}
				T_g(r,K_{\overline{X}'}(H))
				\leq \overline{N}_f(r,D\backslash Z)+\varepsilon T_{g}(r,L)
				\\
				+O(\log r)+o(T_g(r))||.
			\end{equation*}
		\end{enumerate}
	\end{lem}

	\begin{proof}
	We set $W=a(X)$.
	Since $a$ is finite, $W\subset A\times S$ is a Zariski closed set.
	We divide the proof into several steps.

	\textbf{Step 1.}
	We first consider the case that $B\not\subset \mathrm{St}^0(W)$.
		We apply \cref{cor:20220806} to get $\Xi\subset W$.
		We set $B'=B\cap \mathrm{St}^0(W)$.
		Then $B'\subsetneqq B$.
		We set $P=\{ B'\}$.
		Now let $f:Y\to X$ satisfies $f(Y)\not\subset D\cup\Xi$ and $a\circ f\in I_B$.
		By \cref{cor:20220806} we have $a\circ f\in I_{\mathrm{St}^0(W)}$.
		Hence $a\circ f\in I_{B'}$.
		This shows \cref{cor:20220805}, provided $B\not\subset \mathrm{St}^0(W)$.
		In the following, we consider the case $B\subset \mathrm{St}^0W$, which allows us to apply \cref{lem:20220909gg}.
		
		\textbf{Step 2.}
		We find $\overline{X}$ and $\overline{X}'$.
		Let $R\subset W$ be a reduced Weil divisor such that $a:X\to W$ is {\'e}tale outside $R$.
		We take a reduced Weil divisor $D'\subset W$ such that $a(D\cup X_{\mathrm{sing}})\cup R\subset D'$.
		By \cref{lem:20220909gg},  there exist a compactification $W\subset \overline{W}$ and a proper birational morphism $\psi:V\to \overline{W}$, where $V$ is smooth and $\psi^{-1}(\overline{D'}\cup \partial W)\subset V$ is a simple normal crossing divisor.
		Let $\overline{X}\to \overline{W}$ be obtained from $X\to W$ by the normalization.
		Let $p:\overline{X}'\to V$ be obtained from the base change and normalization.

		\begin{equation*}
			\begin{tikzcd}
				X \arrow[r,hookrightarrow] \arrow[d, "a" ']	& \overline{X} \arrow[d]  &\overline{X}' \arrow[l,  "\varphi" ' ] \arrow[d, "p"] \\
				W\arrow[r,hookrightarrow]  & \overline{W}    &V \arrow[l, "\psi"] 
			\end{tikzcd}
		\end{equation*}
		Note that $p:\overline{X}'\to V$ is unramified outside $\widetilde{D}=\psi^{-1}(\overline{D'}\cup \partial W)$.
		Since $\widetilde{D}$ is simple normal crossing, \cite[Lemma 2]{NWY13} yields that $\overline{X}'$ is $\mathbb
		Q$-factorial.

		\textbf{Step 3.}
		We find $P$ and $\Xi$ from the given $Z$, $L$ and $\varepsilon$.
		Let $Z\subset X$ be a Zariski closed set such that $\mathrm{codim}(Z,X)\geq 2$ and $Z\subset D$. 	We set 
		$$Z'=a(Z\cup X_{\mathrm{sing}}
		)\cup (W\cap (\psi\circ p)(\overline{X}'_{\mathrm{sing}})).$$
		Then $\mathrm{codim}(Z',W)\geq 2$ and $Z'\subset D'$.
		Let $L$ be a big line bundle on $\overline{X}'$.
		Let $L_{V}$ be an ample line bundle on $V$.
		Then there exists a positive integer $l\in\mathbb Z_{\geq 1}$ such that $p^*L_{V}(E)=L^{\otimes l}$ for some effective divisor $E\subset \overline{X}'$.
		Let $\varepsilon>0$.
		We set $\varepsilon'=\varepsilon/l$.
		
		Now by \cref{lem:20220909gg} for $Z'\subset D'$, $L_{V}$ and $\varepsilon'$, we get a finite subset $P\subset \mathcal{S}(B)\backslash \{ B\}$ and a proper Zariski closed subset $\Phi\subsetneqq W$.
		We set 
		$$\Xi=a^{-1}(D'\cup \Phi)\cup (X\cap \varphi(E))\cup (X\cap \varphi(\mathrm{Ex}(\varphi))).$$

		\textbf{Step 4.} 
		Let $f:Y\to X$ be a holomorphic map such that $a\circ f\in I_B$ and $f(Y)\not\subset D\cup \Xi$.
		Assume that the first assertion of \cref{cor:20220805} does not valid.
		Let $g:Y\to \overline{X}'$ be the lift of $f$.
		Then $p\circ g:Y\to V$ is the lift of $a\circ f:Y\to W$.
		Hence by \cref{lem:20220909gg}, we get 
		\begin{equation}\label{eqn:202211261}
			T_{p\circ g}(r,K_V(\widetilde{D}))\leq \overline{N}_{a\circ f}(r,D'\backslash Z')
			+\varepsilon' T_{p\circ g}(r,L_V)+O(\log r)+o(T_{p\circ g}(r))||
		\end{equation}
		We define a reduced divisor $F$ on $\overline{X'}$ by $p^{-1}(\widetilde{D})=F+H$.
		By the ramification formula, we have
		\begin{equation*}
			K_{\overline{X'}}(F+H)=p^*K_{V}(\widetilde{D}).
		\end{equation*}
		Thus we have
		\begin{equation*}
			T_{g}(r,K_{\overline{X'}}(F+H)) 	=T_{p\circ g}(r,K_{V}(\widetilde{D}))+O(\log r).
		\end{equation*}
		Combining this estimate with \eqref{eqn:202211261} and $p^*L_{V}(E)=L^{\otimes l}$, we get
		\begin{equation}\label{eqn:20221201}
			T_{g}(r,K_{\overline{X'}}(F+H))\leq \overline{N}_{a\circ f}(r,D'\backslash Z')
			+\varepsilon T_{g}(r,L)+O(\log r)+o(T_{ g}(r))||
		\end{equation}

		We estimate the right hand side of \eqref{eqn:20221201}.
		We have
		\begin{equation}\label{eqn:202212011}
			\overline{N}_{a \circ
				f}(r,\D'\backslash Z')
			\leq \overline{N}_{g}(r,F\backslash  \overline{X}'_{\mathrm{sing}})+\overline{N}_{g}(r,H\backslash  \varphi^{-1}(Z\cup X_{\mathrm{sing}}))
		\end{equation}
		We set $H'$ by $H=H'+\varphi^{-1}(\partial X+\overline{D})$.
		Then $\varphi(H')\subset X_{\mathrm{sing}}$.
		Hence by $\overline{N}_f(r,\partial X)=0$, 
		we have
		\begin{equation}\label{eqn:202208051} 	\overline{N}_g(r,H\backslash \varphi^{-1}(Z\cup X_{\mathrm{sing}}))\leq \overline{N}_f(r,D\backslash Z).
		\end{equation}
		Since $\overline{X}'$ is $\mathbb Q$-factorial, we may take a positive integer $k$ such that $kF$ is a Cartier divisor.
		Since $F\cap (\overline{X}'\backslash \overline{X}'_{\mathrm{sing}})$
		is a Cartier divisor
		on $\overline{X}'\backslash \overline{X}'_{\mathrm{sing}}$, we have
		$$
		k\ \mathrm{ord}_z g^{*}F=\mathrm{ord}_z g^{*}(kF)
		$$
		for $z\in g^{-1} (\overline{X}'\backslash \overline{X}'_{\mathrm{sing}})$,
		and hence
		$$
		k\min \{ 1,\mathrm{ord}_z g^{*}F\}
		\leq \mathrm{ord}_z g^{*}(kF).
		$$
		Thus we get
		\begin{equation*}
			k\overline{N}_g(r,F\backslash \overline{X}'_{\mathrm{sing}})\leq
			N_g(r,kF).
		\end{equation*}
		By the Nevanlinna inequality (cf. \eqref{eqn:20221120}), we have
		$$
		N_g(r,kF)\leq T_{g}(r,kF)+O(\log r).
		$$
		Hence
		\begin{equation*}
			\overline{N}_g(r,F\backslash \overline{X}'_{\mathrm{sing}})\leq T_{g}(r,F)+O(\log r).
		\end{equation*}
		Combining this with \eqref{eqn:202212011} and \eqref{eqn:202208051}, we get
		$$
		\overline{N}_{a \circ
			f}(r,\D'\backslash Z')\leq \overline{N}_f(r,D\backslash Z)+T_{g}(r,F)+O(\log r).	$$
		Combining this estimate with \eqref{eqn:20221201}, we get the second assertion of \cref{cor:20220805}.
		This concludes the proof.
	\end{proof}

	\begin{proposition}\label{prop:20220902}
		Let $\Sigma$ be a smooth quasi-projective variety which is of log general type.
		Assume that there is a morphism $a:\Sigma\to A\times S$ such that $\dim \Sigma=\dim a(\Sigma)$, where $S$ is a projective variety.
		Let $B\subset A$ be a semi-abelian subvariety.
		Then there exist a finite subset $P\subset \mathcal{S}(B)\backslash\{B\}$ and a proper Zariski closed set $\Phi\subsetneqq \Sigma$ with the following property:
		Let $f:Y\to \Sigma$ be a holomorphic map such that $a\circ f\in I_B$ and $f(Y)\not\subset \Phi$. 
		Then either one of the followings holds:
		\begin{enumerate}
			\item
			There exists $C\in P$ such that $a\circ f\in I_C$.
			\item
			$f$ does not have essential singularity over $\infty$.
		\end{enumerate}
	\end{proposition}

	\begin{proof}
		Let $W\subset A\times S$ be the Zariski closure of $a(\Sigma)$.
		Let $\pi :X\to W$ be the normalization with respect to $a$
		and let $\varphi :\Sigma\to X$ be the induced map.
		Then $\varphi$ is birational.
		Let $\overline{\Sigma}$ be a smooth partial compactification such that $\varphi$
		extends to a projective morphism $\overline{\varphi}:\overline{\Sigma}\to X$ and
		$\overline{\Sigma}\backslash \Sigma$ is a divisor on $\overline{\Sigma}$. 
		Since $\overline{\varphi}$ is birational and $X$ is normal, there exists a
		Zariski closed subset $Z\subset X$ whose codimension is greater than one
		such that $\overline{\varphi}:\overline{\Sigma}\to X$ is an isomorphism over $X\backslash Z$.
		In particular $X\backslash Z$ is smooth.
		Let $D$ be the Zariski closure of
		$(X\backslash Z)\cap \bar{\varphi}(\bar{\Sigma}\backslash \Sigma)$ in $X$.
		Then $D$ is a reduced divisor on $X$ and $X\backslash (Z\cup D)$
		is of log-general type.
		Thus by \cite[Lemma 4]{NWY13}, $X\backslash (X_{\mathrm{sing}}\cup D)$
		is of log-general type.
		Note that $(D\cap Z)\cup X_{\mathrm{sing}}\subset Z$.

		We apply \cref{cor:20220805} to get $\overline{X}$ and $\psi:\overline{X}'\to \overline{X}$.
		We define a reduced divisor $H\subset \overline{X}'$ to be 	
		$\psi^{-1}(\overline{D\cup X_{\mathrm{sing}}\cup \partial X})=H\cup \Gamma$, where $\Gamma\subset \overline{X}'$ is a Zariski closed subset of codimension greater than one.	
		Since $X\backslash (D\cup X_{\mathrm{sing}})$ is of log-general type, we deduce that $\psi ^{-1}(X\backslash (D\cup
		X_{\mathrm{sing}}))=\overline{X}'\backslash (H\cup \Gamma)$ is also of log-general type.
		Thus by \cite[Lemma 3]{NWY13}, $K_{\overline{X}'}(H)$ is big.
		
		By Kodaira's lemma, there exist an effective divisor $E\subsetneqq \overline{X}'$ and a positive integer $l\in\mathbb Z_{\geq 1}$ such that $lK_{\overline{X}'}(H)-E$ is ample.	
		Hence if $g:Y\to \overline{X}'$ satisfies $g(Y)\not\subset E$, then 
		$$T_{g}(r)= O( T_{g}(r,K_{\overline{X'}}(H)))+O(\log r).$$

		By \cref{cor:20220805} applied to $Z\cap D$, $L=K_{\overline{X}'}(H)$ and $\varepsilon=1/2$, we get $P\subset \mathcal{S}(B)\backslash\{B\}$ and $\Xi\subsetneqq X$.
		We set $\Phi=\varphi^{-1}(\Xi\cup\psi(E))$.
		Let $f:Y\to \Sigma$ be a holomorphic map such that $a\circ f\in I_B$ and $f(Y)\not\subset \Phi$.
		Then $\varphi\circ f(Y)\not\subset \Xi$. 
		Suppose that the first assertion of \cref{prop:20220902} is not valid.
		Then by \cref{cor:20220805}, we get
		\begin{equation*}
			T_{g}(r,K_{\overline{X}'}(H))
			\leq \overline{N}_{\varphi\circ f}(r,D\backslash Z)+\frac{1}{2}T_{g}(r,K_{\overline{X}'}(H))
			\\
			+O(\log r)+o(T_{g}(r))||,
		\end{equation*}
		where $g:Y\to \overline{X}'$ is the lift of $\varphi\circ f:Y\to X$.
		We have
		$$
		\overline{N}_{\varphi\circ f}(r,D)=\overline{N}_{\varphi\circ f}(r,D\cap Z).$$
		Hence we get
		$$
		T_{g}(r,K_{\overline{X}'}(H))=O(\log r)+o(T_{g}(r))||.
		$$ 	 	
		Thus $T_{\varphi\circ f}(r)=O(\log r) ||$.
		This shows $T_{a\circ f}(r)=O(\log r) ||$. 
		Hence $a\circ f\in I_{\{0\}}$.
		Hence by \cref{rem:20221203}, $f$ does not have essential singularity over $\infty$.
	\end{proof}

	The following theorem implies \cref{thm2nd}, when $S$ is a single point.
	
	\begin{thm}\label{thm:20221201}
		Let $X$ be a smooth quasi-projective variety which is of log general type.
		Assume that there is a morphism $a:X\to A\times S$ such that $\dim X=\dim a(X)$, where $S$ is a projective variety.
		Then there exists a proper Zariski closed set $\Xi\subsetneqq X$ with the following property:
		Let $f:Y\to X$ be a holomorphic map with the following three properties:
		\begin{enumerate}[label=(\alph*)]
			\item \label{item1}
			$T_{(a\circ f)_S}(r)=O(\log r)+o(T_f(r))||$, 
			\item \label{item2}
			$N_{\ram\pi}(r)=O(\log r)+o(T_f(r))||$,
			\item \label{item3}
			$f(Y)\not\subset \Xi$.
		\end{enumerate}
		Then $f$ does not have essential singularity over $\infty$.
	\end{thm}

	\begin{proof}
		We take a Zariski closed set $E\subsetneqq X$ such that if $g:Y\to X$ satisfies $g(Y)\not\subset E$, then $T_{g}(r)=O(T_{a\circ g}(r))+O(\log r)$.
		We take a sequence $\mathcal{P}_i\subset \mathcal{S}(A)$ of finite sets as follows.
		Set $\mathcal{P}_1=\{A\}$.
		Given $\mathcal{P}_i$, we define $\mathcal{P}_{i+1}$ as follows.
		For each $B\in\mathcal{P}_i$, we apply \cref{prop:20220902} to get $\Phi_B\subsetneqq X$ and $P_B\subset \mathcal{S}(B)\backslash\{B\}$.
		We set $\mathcal{P}_{i+1}=\cup_{B\in \mathcal{P}_i}P_B$ and $\Xi_{i+1}=\cup_{B\in\mathcal{P}_i}\Phi_B$.
		Set $\Xi=E\cup \cup_i\Xi_i$.
		Then $\Xi\subsetneqq X$ is a proper Zariski closed set.
		
		Let $f:Y\to X$ satisfy the three properties in \cref{thm:20221201}.
		Then $f\in I_A$.
		We take $i$ which is maximal among the property that there exists $B\in \mathcal{P}_i$ such that $f\in I_B$. 	
		Then by $f(Y)\not\subset \Phi_B\subset \Xi$,
		\cref{prop:20220902} implies that $f$ does not have essential singularity over $\infty$.
	\end{proof}
	We state and prove a more general result than that used in the implications of \cref{cor:20221102}: \ref{being general type1}$\implies$\ref{pseudo Picard1} and \ref{spab}$\implies$\ref{strong LGT1}.
	 This result will be used in the proof of \cite[Theorem C]{CDY22}. 
	\begin{cor}\label{cor:GGL}
		Let $X$ be a smooth quasi-projective variety and let $a:X\to A\times S^{\circ}$ be a morphism such that $\dim X=\dim a(X)$, where $S^{\circ}$ is a smooth quasi-projective variety ($S^{\circ}$ can be a point). Write $b:X\to S^{\circ}$ as the composition of $a$ with the projection map $A\times S^{\circ}\to S^{\circ}$. Assume that $b$ is dominant.
		\begin{thmlist}
			\item\label{coritem1} 
			Suppose $S^{\circ}$ is pseudo Picard hyperbolic.
			If $X$ is of log general type, then $X$  is pseudo Picard hyperbolic.
			\item \label{coritem2} Suppose $\Spalg(S^{\circ})\subsetneqq S^{\circ}$.
			If $\Spab(X)\subsetneqq X$, then $\Spalg(X)\subsetneqq X$.   
		\end{thmlist} 
	\end{cor}
	
	\begin{proof}
		Let $S$ be a smooth projective variety that compactifies $S^\circ$. 
		\\
		\noindent{\em Proof of \cref{coritem1}:}	Since $S^\circ$ is pseudo Picard hyperbolic, there exists a proper Zariski closed subset $Z\subset S^\circ$ such that each holomorphic map $g:\bD^*\to S^\circ$ with $g(\bD^*)\not\subset Z$ has no essential singularity at $\infty$. 
		By applying \cref{thm:20221201} to $X\to A\times S$, we obtain a proper Zariski closed set $\Xi\subsetneqq X$ that satisfies the property given in \cref{thm:20221201}. Set $\Sigma:=b^{-1}(Z)\cup \Xi$.  Since $b$ is dominant, $\Sigma$ is a proper Zariski closed subset of $X$.  Then for any holomorphic map $f:\bD^*\to X$ with $f(\bD^*)\not\subset \Sigma$, we have $b\circ f(\bD^*)\not\subset Z$ and $f(\bD^*)\not\subset \Xi$. It follows that
		$
		T_{b\circ f}(r)=O(\log r).
		$ 
		Hence $f$ verifies Properties \Cref{item1} and \Cref{item3} in   \cref{thm:20221201}. 
		Note that Property \Cref{item2}  in \cref{thm:20221201} is automatically satisfied as  $N_{\ram\pi}(r)=0$. We can now apply \cref{thm:20221201} to conclude that $f$ does not have an essential singularity over $\infty$. Therefore, $X$ is pseudo Picard hyperbolic.
		
		\medspace
		
		\noindent {\em Proof of \cref{coritem2}:}   
		Since $\dim X = \dim a(X)$, there is a proper Zariski closed subset $\Upsilon \subsetneqq X$ such that the restriction of $a|_{X \backslash \Upsilon}:X \backslash \Upsilon\to A\times S^\circ$ is quasi-finite. 
		Let $\Xi := b^{-1}(\Spalg(S^\circ)) \cup \Spab(X) \cup \Upsilon$.  
		Then by the assumptions $\Spalg(S^\circ)\subsetneqq S^\circ$ and $\Spab(X)\subsetneqq X$, we have $\Xi\subsetneqq X$.

		Let $V$ be a closed subvariety of $X$ that is not contained in $\Xi$. Then we have
		\begin{equation}\label{eqn:20230510}
			\Spab(V)\subsetneqq V.
		\end{equation} 
		In the following, we shall prove that $V$ is of log general type to conclude $\Spalg(X)\subset \Xi$.

		We first show $\bar{\kappa}(V)\geq 0$.
		Note that for a general fiber $F$ of $b|_{V}:V\to S^\circ$, we have $\dim F=\dim c(F)$, where $c:X\to A$ is the composition of $a$ with the projection map $A\times S^\circ\to A$. 
		By \cref{prop:Koddimabb}, it follows that $\bar{\kappa}(\overline{c(F)})\geq 0$, and hence $\bar{\kappa}(F)\geq 0$. 
		Also note that $b(V)$ is not contained in $\Spalg(S^\circ)$, and thus $\overline{b(V)}$ is of log general type.  
		We use Fujino's addition formula for logarithmic Kodaira dimensions \cite[Theorem 1.9]{Fuj17} to conclude that $$\bar{\kappa}(V)\geq \bar{\kappa}(F)+\bar{\kappa}(\overline{b(V)})\geq 0.$$  
		Hence we have proved $\bar{\kappa}(V)\geq 0$.
		We may consider the logarithmic Iitaka fibration of $V$.
		
		Next we show $\bar{\kappa}(V)=\dim V$.
		After replacing $V$ with a birational modification, we can assume that $V$ is smooth and that the logarithmic Iitaka fibration $j:V\to J(V)$ is regular. Assume contrary that $V$ is not of log general type. Note that for a very general fiber $F$ of $j$, the followings hold:
		\begin{enumerate}[label*=(\alph*)]
			\item \label{item:positive} $\dim F>0$;
			\item  $F$ is smooth;
			\item   \label{item:zero}   $\bar{\kappa}(F)=0$;
			\item \label{item:avoid} $b(F)\not\subset \Spalg(S^\circ)$;
			\item \label{item:same dim} $F\not\subset \Upsilon$.   
		\end{enumerate} 
		By \Cref{item:avoid}, we have $\bar{\kappa}(\overline{b(F)})=\dim b(F)\geq 0$. 
		\Cref{item:same dim} implies that for a general fiber $Y$ of $b|_{F}:F\to S^\circ$, we have $\dim Y=\dim c(Y)$. 
		Using \cref{prop:Koddimabb}, we have $\bar{\kappa}(\overline{c(Y)})\geq 0$, and thus $\bar{\kappa}(Y)\geq 0$. 
		Using \cite[Theorem 1.9]{Fuj17} again, we can conclude that $$\bar{\kappa}(F)\geq \bar{\kappa}(Y)+\bar{\kappa}(\overline{b(F)})\geq 0.$$
		\Cref{item:zero} implies that $\bar{\kappa}(Y)=0$ and $\bar{\kappa}(\overline{b(F)})=0$. Hence $b(F)$ is a point. 
		This implies that $\dim F=\dim c(F)$. 
		Combining with \Cref{item:positive,item:zero}, \cref{lem:20230509} yields $\Spab(F)=F$. 
		Since $F$ is a very general fiber of $j:V\to J(V)$, we get $\Spab(V)=V$.
		This contradicts to \eqref{eqn:20230510}.
		Thus we have proved that $V$ is of log general type, hence $\Spalg(X)\subset \Xi$.
	\end{proof}

The following proposition finds applications in the third part of this paper series.
	
	\begin{proposition}\label{prop:20230405}
		Let $D\subset A$ be a reduced divisor such that $\mathrm{St}(D)=\{ a\in A; a+D=D\}$ is finite.
		Let $Z\subset D$ be a Zariski closed subset such that $\mathrm{codim}(Z,A)\geq 2$.
		Then there exists a proper Zariski closed subset $\Xi\subsetneqq A$ with the following property:
		Let $f:Y\to A$ be a holomorphic map with the following three properties:
		\begin{enumerate}[label=(\alph*)]
			\item  
			$N_{\ram\pi}(r)=O(\log r)+o(T_f(r))||$, 
			\item  
			$f(Y)\not\subset \Xi\cup D$,
			\item
			$\mathrm{ord}_yf^*D\geq 2$ for all $y\in f^{-1}(D\backslash Z)$.	\end{enumerate}
		Then $f$ does not have essential singularity over $\infty$.
	\end{proposition}
	
	\begin{proof}
		For a semi-abelian variety $B\subset A$, we denote by $J_B$ the set of all holomorphic maps $f:Y\to A$ such that 
		\begin{itemize}
			\item
			$f\in I_B$,  
			\item
			$f(Y)\not\subset D$, and
			\item
			$\mathrm{ord}_yf^*D\geq 2$ for all $y\in f^{-1}(D\backslash Z)$. \end{itemize}
		Given a semi-abelian variety $B\subset A$, we first prove the following claim.
		
		\begin{claim}\label{claim:20230405}
			There exist a finite subset $P_B\subset \mathcal{S}(B)\backslash\{B\}$ and a proper Zariski closed set $\Phi_B\subsetneqq A$ with the following property:
			Let $f:Y\to A$ be a holomorphic map such that $f\in J_B$ and $f(Y)\not\subset \Xi_B$. 
			Then either one of the followings holds:
			\begin{enumerate}[label=(\alph*)]
				\item
				There exists $C\in P_B$ such that $f\in J_C$.
				\item
				$f$ does not have essential singularity over $\infty$.
			\end{enumerate}
		\end{claim}
		
		\begin{proof}[Proof of \cref{claim:20230405}]
			Since $\mathrm{St}(D)$ is finite, $A\backslash D$ is of log-general type.
			We apply \cref{lem:20220909gg} to get a smooth projective equivariant compactification $\overline{A}$ and a proper birational morphism $\varphi:V\to \overline{A}$, where $V$ is smooth and $\varphi^{-1}(\overline{D}\cup \partial A)$ is a simple normal crossing divisor.
			Since $A\backslash D$ is of log-general type, we deduce that $\varphi ^{-1}(A\backslash D)=V\backslash \varphi^{-1}(\overline{D}\cup \partial A)$ is also of log-general type.
			Thus $K_{V}(\varphi^{-1}(\overline{D}\cup \partial A))$ is big.
			
			By Kodaira's lemma, there exist an effective divisor $E\subsetneqq \overline{A}'$ and a positive integer $l\in\mathbb Z_{\geq 1}$ such that $L=lK_{V}(\varphi^{-1}(\overline{D}\cup \partial A))-E$ is ample.	
			Hence if $g:Y\to V$ satisfies $g(Y)\not\subset E$, then 
			$$T_{g}(r)= O( T_{g}(r,K_{V}(\varphi^{-1}(\overline{D}\cup \partial A))))+O(\log r).$$

			By \cref{lem:20220909gg} applied to $(Z\cup \varphi (\mathrm{Ex}(\varphi)))\cap D$, $L$ and $\varepsilon=1/3l$, we get a finite subset $P_B\subset \mathcal{S}(B)\backslash\{B\}$ and a proper Zariski closed set $\Phi_B\subsetneqq A$ with $\varphi (\mathrm{Ex}(\varphi))\subset \overline{\Phi_B}$.
			We set $\Xi_B=\Phi_B\cup\varphi(E)$.
			Let $f:Y\to A$ be a holomorphic map such that $f\in J_B$ and $f(Y)\not\subset \Xi_B$.
			Then $ f(Y)\not\subset D\cup \Xi_B$ and $f\in I_B$.
			Suppose that the first assertion of \cref{claim:20230405} is not valid.
			Then $f\not\in I_C$ for all $C\in P_B$.
			By \cref{lem:20220909gg}, we get
			\begin{equation*}
				T_{g}(r,K_{V}(\varphi^{-1}(\overline{D}\cup \partial A)))
				\leq \overline{N}_{f}(r,D\backslash (Z\cup \varphi (\mathrm{Ex}(\varphi))))+\frac{1}{3l}T_{g}(r,L)
				\\
				+O(\log r)+o(T_{g}(r))||,
			\end{equation*}
			where $g:Y\to V$ is the lift of $ f:Y\to A$.
			By $g(Y)\not\subset E$, we have 
			$$
			T_{g}(r,L)\leq lT_{g}(r,K_{V}(\varphi^{-1}(\overline{D}\cup \partial A)))+O(\log r).
			$$
			Hence we get
			\begin{equation*}
				\frac{2}{3}T_{g}(r,K_{V}(\varphi^{-1}(\overline{D}\cup \partial A)))
				\leq \overline{N}_{f}(r,D\backslash (Z\cup \varphi (\mathrm{Ex}(\varphi))))
				\\
				+O(\log r)+o(T_{g}(r))||.
			\end{equation*}

			We estimate the first term of the right hand side.
			Since $\overline{A}$ is smooth and equivariant, $\partial A$ is a simple normal crossing divisor.
			Hence we may decompose as $\varphi^{-1}(\overline{D}\cup \partial A)=H+F$ so that $F=\varphi^{-1}(\partial A)$.
			Then $H=\overline{\varphi^{-1}(D)}$.
			The induced map $V\backslash (F\cup \mathrm{Ex}(\varphi))\to A\backslash \varphi(\mathrm{Ex}(\varphi))$ is isomorphic.
			Hence by $\mathrm{ord}_yf^*D\geq 2$ for all $y\in f^{-1}(D\backslash Z)$, we have
			$$
			2\overline{N}_{f}(r,D\backslash (Z\cup \varphi (\mathrm{Ex}(\varphi))))\leq T_g(r,H).
			$$
			By $\bar{\kappa}(V\backslash F)=0$ and $g(Y)\not\subset \mathrm{Ex}(\varphi)$, we have $T_g(r,K_V(F))+O(\log r)>0$.
			Hence
			$$
			2\overline{N}_{f}(r,D\backslash (Z\cup \varphi (\mathrm{Ex}(\varphi))))\leq T_g(r,K_V(H+F))+O(\log r).
			$$
			Hence we get
			$$
			T_{g}(r,K_{V}(H+F))=O(\log r)+o(T_{g}(r))||.
			$$ 	 	
			Thus $T_{g}(r)=O(\log r) ||$.
			This shows $T_{f}(r)=O(\log r) ||$. 
			Hence $f\in I_{\{0\}}$.
			Hence by \cref{rem:20221203}, $f$ does not have essential singularity over $\infty$.
		\end{proof}

		We take a sequence $\mathcal{P}_i\subset \mathcal{S}(A)$ of finite sets as follows.
		Set $\mathcal{P}_1=\{A\}$.
		Given $\mathcal{P}_i$, we define $\mathcal{P}_{i+1}$ as follows.
		For each $B\in\mathcal{P}_i$, we apply \cref{claim:20230405} to get $\Xi_B\subsetneqq A$ and $P_B\subset \mathcal{S}(B)\backslash\{B\}$.
		We set $\mathcal{P}_{i+1}=\cup_{B\in \mathcal{P}_i}P_B$ and $\Xi_{i+1}=\cup_{B\in\mathcal{P}_i}\Xi_B$.
		Set $\Xi=\cup_i\Xi_i$.
		Then $\Xi\subsetneqq A$ is a proper Zariski closed set.
		
		Let $f:Y\to A$ satisfy the three properties in \cref{prop:20230405}.
		Then $f\in J_A$.
		We take $i$ which is maximal among the property that there exists $B\in \mathcal{P}_i$ such that $f\in J_B$. 	
		Then by $f(Y)\not\subset \Xi_B\subset \Xi$,
		\cref{claim:20230405} implies that $f$ does not have essential singularity over $\infty$.
	\end{proof}

	\begin{rem}\label{rem:20250904}
	The proof of \cref{thm2nd} is based on the arguments in \cite{Yam15} and \cite{NWY13}.
A major difficulty in treating the non-compact case, compared with the compact case in \cite{Yam15}, is the lack of the Poincar{\'e} reducibility theorem.
To overcome this problem, we use a more general type of "cover" than an {\'e}tale cover.
In \cite{NWY08}, Noguchi, Winkelmann, and the third author utilized a "flat cover" (cf. the commutative diagram just after \cite[(5.9)]{NWY08}).
An important issue with using this flat cover is to construct a lift of holomorphic maps onto it, such that the order function of the lift is bounded by that of the original map (cf. \cite[Lemma 5.8]{NWY08}).
The authors in \cite{NWY08} only considered holomorphic maps from the complex plane $\mathbb C$. Thanks to the simple connectedness of $\mathbb C$, such a lift could be constructed easily.
In contrast, we consider holomorphic maps from the covering space $\pi:Y\to \bC_{>\delta}$, which is not simply connected.
Therefore, in this paper, we choose a different cover $\sigma:\Sigma\to \overline{W/B}$ in the proof of \cref{lem:20220909gg}.
After the base change by this cover, we obtain the lift $g:Y'\to \overline{B}\times \Sigma$ of $f:Y\to W$ from the covering space $Y'\to Y$.
We then apply the results of \cref{subsec:4.2}--\cref{subsec:4.7a} to this lift $g$.
We remark that the map $g_{\overline{B}}:Y'\to \overline{B}$ may hit the boundary $\partial B$. This is the reason why we treat the situation $f:Y\da A\times S$ in \cref{subsec:4.2}--\cref{subsec:4.7a}.
We also remark that even if we restrict the statement of \cref{thm2nd} to $f:\mathbb D^*\to X$, we still need to consider a ramified covering $Y'\to \mathbb D^*$ in the proof to obtain the map $g:Y'\to \overline{B}\times \Sigma$ mentioned above.
	\end{rem}

\section{Proof of \cref{thm:theorem b}}\label{sec:11}
 We consider another application of \cref{thm2nd}, whose proof requires a ramified covering of $\mathbb{C}_{>\delta}$.
The following theorem is used in Part~II of this series.
Although it appears only implicitly in \cite[\S 6]{CDY22original} and is not stated explicitly, we find it convenient to formulate it here as follows, since the original work \cite{CDY22original} was divided into separate papers.
See also \cite[\S 3]{Yam10}.

\begin{thm}\label{thm:20250911}
Let $X$ be a smooth quasi-projective variety with a smooth projective compactification $\overline{X}$ such that the boundary divisor $D=\overline{X}\setminus X$ is a simple normal crossing divisor. 
Assume that there exists a finite surjective morphism $p:\overline{\Sigma}\to \overline{X}$ from a normal projective variety $\overline{\Sigma}$ together with non-zero sections $\tau_1,\ldots,\tau_l \in H^0(\overline{\Sigma}, p^*\Omega_{\overline{X}}(\log D))$ satisfying the following two conditions:
\begin{enumerate}[label=(\roman*)]
\item \label{cond:i}
setting $\Sigma=p^{-1}(X)$, the variety $\Sigma$ is of log-general type and admits a morphism $a:\Sigma\to A$ into a semi-abelian variety $A$ with $\dim \Sigma=\dim a(\Sigma)$;
\item \label{cond:iii}
defining
\[ 
R:=\{s\in \overline{\Sigma} \mid \exists\, i \ \text{such that } \tau_i(s)=0 \}, 
\]
then $R\subsetneqq \overline{\Sigma}$ is a proper Zariski closed subset, and $p:\overline{\Sigma}\to \overline{X}$ is \'etale outside $R$.
\end{enumerate}
Then every holomorphic map $f:\mathbb D^*\to X$ with Zariski dense image has a holomorphic extension $\bar{f}:\mathbb D\to\overline{X}$.
\end{thm}

In the second condition \ref{cond:iii}, we observe that $p^*\Omega_{\overline{X}}(\log D)$ is a locally free sheaf of rank $\dim X$ on $\overline{\Sigma}$.
At a point $s\in\overline{\Sigma}$, we then consider
$\tau_i(s)\in (p^*\Omega_{\overline{X}}(\log D))_s\otimes \mathbb C$, where $(p^*\Omega_{\overline{X}}(\log D))_s$ denotes the stalk at $s\in\overline{\Sigma}$.
Thus the set $\{s\in\overline{\Sigma}; \tau_i(s)=0\}$ is a Zariski closed subset of $\overline{\Sigma}$, since it coincides with the zero locus of a section $\tau_i$ of the corresponding vector bundle on $\overline{\Sigma}$.

We shall see in Part~II of our series that a covering $\overline{\Sigma}\to \overline{X}$ with the property of \cref{thm:20250911} exists provided there is a representation $\varrho:\pi_1(X)\to G(K)$ into an almost simple algebraic group $G$ defined over a $p$-adic field $K$, where $\varrho$ is big, unbounded, and Zariski dense.
This covering is realized by a spectral cover constructed from the $\varrho$-equivariant pluriharmonic map $\widetilde{X}\to\Delta(G)$ from the universal cover $\widetilde{X}$ of $X$ to the Bruhat–Tits building $\Delta(G)$ of $G$.
This pluriharmonic map was developed by the second author and co-authors \cite{BDDM}, and it will be employed in  Part~II of our series to construct the corresponding spectral cover.

To prove \cref{thm:20250911} we start with the following lemma.

\begin{lem}\label{lem:20250911}
Let $X$ be a smooth quasi-projective variety with a smooth projective compactification $\overline{X}$ such that the boundary divisor $D=\overline{X}\setminus X$ is a simple normal crossing divisor.
Let $\omega \in H^0(\overline{X}, \Omega_{\overline{X}}(\log D))$ be a nonzero section, and let $u:Z \to X$ be a morphism from a smooth quasi-projective variety $Z$ such that $u^*\omega \in H^0(Z, \Omega_Z)$ vanishes.
Then, for every holomorphic map $g:Y \to X$ with Zariski dense image and $g^*\omega=0$, we have
$\overline{N}_g(r,  u(Z))= o(T_{g}(r))+O(\log r)||$.
\end{lem}

\begin{proof}
Let $\alpha:X \to A_X$ be the quasi-Albanese map.
Then there exists a non-zero holomorphic one-form $\tilde{\omega}$ on $A_X$ such that $\alpha^*\tilde{\omega}=\omega$.
Let $B \subset A_X$ be the largest semi-abelian subvariety of $A_X$ on which the pull-back of the holomorphic one form $\tilde{\omega}$ vanishes.
Set $A' = A_X / B$.
Let $q:A_X \to A'$ be the quotient map, and let $\beta:X \to A'$ be the composition of $\alpha:X \to A_X$ with $q:A_X \to A'$.

We claim that $\beta(u(Z))$ is a point.
Indeed, we have the following commutative diagram:
	\[
		\begin{tikzcd}
			X \arrow[r, "\alpha"] & A_{X} \arrow[r, "q"] & A_X/B=A' \\
			Z \arrow[u, "u"] \arrow[r, "a"] & A_{Z} \arrow[u, "j"]
		\end{tikzcd}
	\]
	Here $a:Z \to A_{Z}$ denotes the quasi-Albanese map, and $j:A_{Z} \to A_X$ is the induced morphism of algebraic groups.
By the definition of $B$, the image of $j(A_{Z})$ is contained in $B$.
Therefore, $\beta(u(Z))$ is a single point, as claimed.

Let $\beta(u(Z))=\{a\}$ for some $a\in A'$.
Let $\phi:S\to A'$ be the normalization of the Zariski closure $\overline{\beta(X)}\subset A'$.
Since $X$ is smooth, $\beta$ factors $\phi$ as follows:
	\[
	X\stackrel{\psi}{\to} S\stackrel{\phi}{\to} A'.
	\]
	Since $\phi$ is finite and $\beta(u(Z))=\{a\}$, there is a point $s\in S$ so that $\psi(u(Z))=\{s\}$.

	There exists a non-zero logarithmic one-form $\omega'$ on $A'$ such that $\widetilde{\omega} = q^*\omega'$.
Then
\begin{equation}\label{eqn:202509111} 
\beta^*\omega' = \omega. 
\end{equation} 
Since $\omega'$ is a linear logarithmic one-form on $A'$, we have $d\omega' = 0$.
Hence, by the Poincaré lemma, in some analytic neighborhood $U$ of $a$ in $A'$, there exists a holomorphic function $h \in \mathcal{O}_{A'}^{\mathrm{an}}(U)$ such that $dh = \omega'$ on $U$ and $h(a) = 0$.
By \eqref{eqn:202509111}, $\beta^*dh$ is not identically zero.
Consequently, $h \circ \phi\in \mathcal{O}^{\mathrm{an}}_S(\phi^{-1}(U))$ is a non-constant holomorphic function such that $h\circ\phi(s)=0$.

	Let $\mathcal{O}^{\mathrm{an}}_{S,s}$ be the stalk at $s$ in the sense of analytic spaces, and let $\mathfrak{m}_s \subset \mathcal{O}^{\mathrm{an}}_{S,s}$ denote the maximal ideal.
Since $S$ is normal, $\mathcal{O}^{\mathrm{an}}_{S,s}$ is an integral domain.
We have verified that $h \circ \phi \in \mathcal{O}^{\mathrm{an}}_{S,s}$ is nonzero.
Moreover, since $h \circ \phi(s) = 0$, we have $h \circ \phi\in\mathfrak{m}_s$.
Hence we have
\begin{equation}\label{eqn:202509112}
\dim \mathcal{O}^{\mathrm{an}}_{S,s}/(h \circ \phi)=\dim S-1.
\end{equation}
For every $n\in \bZ_{>0}$, consider the zero-dimensional subscheme on $S$ defined by $V_n:=\spec \cO^{\mathrm{an}}_{S,s}/((h\circ\phi)+\mathfrak{m}_s^n)$.  
	Then $V_n$ is supported at $s$.

	We take a   projective compactification $\overline{S}$ for $S$.
	 We fix an ample line bundle $M$ on $\overline{S}$.  
	Consider the following short exact sequence 
	$$
	0\to M^k\otimes \mathcal{I}_{V_n}\to   M^k\to M^k\otimes \mathcal{O}_{V_n}\to 0,
	$$
	which yields a short exact sequence
	$$
	0\to H^0(\overline{S}, M^k\otimes \mathcal{I}_{V_n})\to   H^0(\overline{S}, M^k)\to H^0(\overline{S}, M^k\otimes \mathcal{O}_{V_n})
	$$
	By \eqref{eqn:202509112}, we have
	$$
	h^0(\overline{S}, M^k\otimes \mathcal{O}_{V_n})=h^0(V_n,  \cO_{V_n})=O(n^{\dim S-1})
	$$
	as $n\to \infty$. 
	By Riemann-Roch theorem, there are positive constants $c>0$ and $\kappa>0$ such that
	$h^0(\overline{S}, M^k)>ck^{\dim S}$ for $k>\kappa$.
	If we take $k_n\sim n^{1-\frac{1}{2\dim S}}$ when $n\to \infty$, it follows that 
	$$
	h^0(\overline{S}, M^{k_n})>h^0(V_n, \cO_{V_n})
	$$
	when $n\gg 1$. Hence there are nonzero sections 
	$$\sigma_n\in H^0(\overline{S}, M^{k_n}\otimes \mathcal{I}_{V_{n}})$$
when $n\gg 1$.  
Let $D_n\subset \overline{S}$ be the corresponding divisor.
Then we have $V_n\subset D_n$ as closed subschemes on $\overline{S}$.

We now consider $\psi\circ g:Y\to S$ which is a holomorphic map with Zariski dense image.
Hence $\psi\circ g(Y)\not\subset D_n$.
By \cref{thm:first}, we have
	$$
	N_{\psi\circ g}(r, V_n)\leq N_{\psi\circ g}(r, D_n)\leq T_{\psi\circ g}(r, M^{k_n})+O(\log r)=k_n T_{\psi\circ g}(r, M)+O(\log r).
	$$
	For every $y\in Y$ so that $g(y)\in u(Z)$, it follows that $\psi\circ g(y)=s$. 
	Since $g^*\omega=0$, it follows that
	\begin{align}\label{eq:nontrivial}
		(\beta\circ g)^*\omega'=g^*\omega=0.
	\end{align}
	Hence $(h\circ\phi)\circ (\psi \circ g)$ is constant on the connected component of $(\beta\circ g)^{-1}(U)$ containing $y$ which is thus zero.
	Therefor $\psi\circ g:U\to S$ factors the zero locus $(h\circ \phi=0)$.
	This implies that  $\ord_y(\psi \circ g)^*V_n\geq n$, hence
	$$
	n\overline{N}_{g}(r, u(Z))\leq N_{\psi\circ g}(r, V_n).
	$$
	In conclusion,
	$$
	\overline{N}_{g}(r, u(Z))\leq \frac{k_n}{n} T_{\psi\circ g}(r, M)+O(\log r).
	$$
On the other hand, $T_{\psi\circ g}(r, M)\leq cT_{g}(r, L)+O(\log r)$ for some constant $c>0$ since order functions decrease under rational map $\overline{X}\dashrightarrow \overline{S}$ induced by $\psi$.  
By $k_n/n\to 0$ as $n\to \infty$, we get 
\begin{equation}\label{eqn:202304151}
	\overline{N}_{g}(r, u(Z))\leq \ep T_{g}(r, L)+O_{\varepsilon}(\log r) 
	\end{equation}
	for every $\ep>0$.

	Suppose that $\varliminf_{r\to\infty}T_g(r,L)/\log r<+\infty$.
	Then by \cref{lem:20230415}, we have $T_g(r,L)=O(\log r)$.
			Then by \eqref{eqn:202304151}, we have $\overline{N}_{g}(r, u(Z))=O(\log r)$, in particular 
	\begin{equation}\label{eqn:202304152}
	\overline{N}_{g}(r, u(Z))=o(T_{g}(r,L))+O(\log r).
	\end{equation}	
		Next we assume $\varliminf_{r\to\infty}T_g(r,L)/\log r=+\infty$.
		Then $\log r=o(T_g(r,L))$.		
		Then by \eqref{eqn:202304151}, we have $\overline{N}_{g}(r, u(Z))\leq \ep T_{g}(r, L)+o(T_g(r,L))$ for all $\varepsilon>0$.	
		Hence we get $\overline{N}_{g}(r, u(Z))=o(T_g(r,L))$, in particular we get \eqref{eqn:202304152}.	
		The proof is completed.	
\end{proof}

\begin{proof}[Proof of \cref{thm:20250911}]
Consider a resolution of singularities $\mu:  S\to \Sigma$ and a projective compactification $\overline{S}$ of $S$ with $E:=\overline{S}\backslash S$ a simple normal crossing divisor.
We may assume that $\mu:S\to\Sigma$ extends to $\mu:\overline{S}\to\overline{\Sigma}$.
Let $q:\overline{S}\to\overline{X}$ be the composite of $\mu:\overline{S}\to \overline{\Sigma}$ and $p:\overline{\Sigma}\to\overline{X}$.
Set $\omega_i=\mu^*\tau_i$ as an element of $H^0(\overline{S}, q^*\Omega_{\overline{X}}(\log D))$.
Set $Z_{i}:=\{s\in \overline{S}; \omega_i(s)=0\}$ for $i\in\{1,\ldots, l\}$.
Then
\begin{equation}\label{eqn:202509121}
\mu^{-1}(R)\subset \bigcup_{i} Z_{i}
\end{equation}
as Zariski closed subsets in $\overline{S}$.

Consider a holomorphic map $f:\bC_{>\delta}\to X$ with Zariski dense image.  The generically finite proper surjective morphism $q:S\to X$ induces a proper surjective holomorphic map   $\pi:Y\to \bC_{>\delta}$  from a Riemann surface $Y$ to  $\bC_{>\delta}$ and  a holomorphic map $g:Y\to S$  satisfying the following commutative diagram:
\begin{equation}\label{figure:curve}
 	\begin{tikzcd}
	Y\arrow[r, "g"] \arrow[d, "\pi"] & S\arrow[d, "q"]\\
	\bC_{>\delta}\arrow[r, "f"] & X
\end{tikzcd}
\end{equation}
Note that $\pi$ is unramified outside $(\mu\circ g)^{-1}(R)$.    
 By \eqref{eqn:202509121}, it follows that 
 \begin{equation}\label{eqn:20250912}
\mathrm{supp} ( \ram \pi)\subset \bigcup_{i}g^{-1}(Z_{i}).
 \end{equation}

We claim
\begin{align}\label{eq:ramification}
	 	N_{{\rm ram} \pi}(r) =o( T_{g}(r)) + O(\log r) ||.
\end{align} 
We prove this.
By \eqref{eqn:20250912}, we have 
	\begin{align*}
		N_{{\rm ram}\pi }(r)\leq \deg \pi \cdot 	\overline{N}_g(r, \bigcup_{i}Z_{i}).  
	\end{align*}
	It then suffices to prove that 
	$$\overline{N}_g(r,  Z_{i})= o(T_{g}(r, L))+O(\log r)||$$ 
	for every  $Z_{i}$. 
	By \eqref{figure:curve}, we have a morphism of sheaves $g^*q^*\Omega_X\to \pi^*\Omega_{\mathbb C_{>\delta}}$.
 Since $\omega_i\in H^0(S, q^*\Omega_{X})$, we may consider $g^*\omega_i$ as a holomorphic section of $\pi^*\Omega^1_{\mathbb C_{>\delta}}$.

	\textbf{Case 1.} 
	$g^*\omega_i\neq 0$.
Write
	$$
	\xi:=\frac{ g^*\omega_i}{\pi^*dz},
	$$   
	which gives a holomorphic map $\xi:Y\to  \bP^1\backslash\{\infty\}$, since $g^*\omega_i$ is a holomorphic section of $\pi^*\Omega^1_{\mathbb C_{>\delta}}$. 
Hence $N_{\xi}(r,[\infty])=0$.
By \eqref{eq:First} one has 
	\begin{align*} 
		N_{\xi}(r,[0])= o(T_{g}(r, L))+O(\log r)||.
	\end{align*}
If $g(y)\in Z_{i}$, then $g^*\omega_i(y)=0$ as a section of $\pi^*\Omega^1_{\mathbb C_{>\delta}}$.
Hence $\xi(y)=0$, therefore
	 $$
	 \overline{N}_g(r, Z_{i})\leq N_{\xi}(r,[0]).
	 $$
These two estimates yield
	$$
	\overline{N}_g(r, Z_{i})= o(T_{g}(r, L))+O(\log r)||.
	$$

	\textbf{Case 2.} 
	$g^*\omega_i= 0$.
	By the morphism $q^*\Omega_{\overline{X}}(\log D)\to \Omega_{\overline{S}}(\log E)$, we regard $\omega_i$ as a section of $H^0(\overline{S},\Omega_{\overline{S}}(\log E))$.
	Then we have $g^*\omega_i= 0$ as sections of $\Omega_{Y}$.
	Let $u:Z\to S$ be a smooth modification of an irreducible component of $Z_{i}\cap S$.
	Since $u(Z)\subset Z_{i}$, the definition of $Z_i$ yields $u^*\omega_i=0$ in $H^0(Z,\Omega_Z)$.
	Therefore by \cref{lem:20250911}, we have $\overline{N}_g(r,  u(Z))= o(T_{g}(r))+O(\log r)||$.
This shows $\overline{N}_g(r,  Z_{i})= o(T_{g}(r, L))+O(\log r)||$. 	
	Hence we have proved \eqref{eq:ramification}.

	Now we identify $\bD^*$ with $\bC_{>1}$  by taking a transformation $z\mapsto\frac{1}{z}$.
	By \cref{thm2nd} and \eqref{eq:ramification}, we have an extension $\bar{g}:\bar{Y}\to\overline{S}$.
	Thus we have an extension $\bar{f}:\mathbb D\to\overline{X}$.
	Indeed, given a sequence $(a_n)$ in $\mathbb C_{>\delta}$ satisfing $a_n\to\infty$, we may take a sequence $(b_n)$ in $Y$ such that $p(b_n)=a_n$ and $b_n\to \beta\in \overline{Y}$.
Then $\bar{p}(\beta)=\infty$ and $f(a_n)=\pi\circ g(b_n)\to \pi\circ \bar{g}(\beta)$.
Thus $(f(a_n))$ converges for every sequence $(a_n)$ satisfing $a_n\to\infty$.
Thus $f$ has continuous extension $\bar{f}:\mathbb C_{>\delta}\cup\{\infty\}\to \overline{X}$, which is holomorphic by Riemann's removable singularity theorem.
\end{proof}

	\providecommand{\bysame}{\leavevmode ---\ }
	\providecommand{\og}{``}
	\providecommand{\fg}{''}
	\providecommand{\smfandname}{\&}
	\providecommand{\smfedsname}{\'eds.}
	\providecommand{\smfedname}{\'ed.}
	\providecommand{\smfmastersthesisname}{M\'emoire}
	\providecommand{\smfphdthesisname}{Th\`ese}

%	\bibliography{biblio}
%	\bibliographystyle{smfalpha}
	
\end{document}